\documentclass[10pt,reqno]{amsart}

\usepackage[colorlinks=true]{hyperref}
\usepackage[T1]{fontenc}
\usepackage{amsmath,amsthm,amsfonts,mathtools}
\usepackage[cm]{fullpage}
\usepackage{xcolor}
\usepackage{bm}
\usepackage{dsfont}

\theoremstyle{plain}
\newtheorem{theorem}{Theorem}[section]
\newtheorem{prop}[theorem]{Proposition}
\newtheorem{lemma}[theorem]{Lemma}
\newtheorem{corollary}[theorem]{Corollary}

\theoremstyle{definition}
\newtheorem{definition}[theorem]{Definition}
\newtheorem{remark}[theorem]{Remark}

\newtheorem{assumption}[theorem]{Assumption}

\def \T {\mathbb{T}}
\def \Z {\mathbb{Z}}
\def \R {\mathbb{R}}
\def \P {\mathbb{P}}
\def \N {\mathbb{N}}
\def \E {\mathbb{E}}
\def \dd {\mathrm{d}}
\def \sgn {\mathrm{sign}}
\def \TA {\tilde\tau} 

\def \e {\mathrm e}

\newcommand{\ind}[1]{\mathds{1}_{\{#1\}}}

\renewcommand{\tilde}{\widetilde}

\title{Finite-volume approximation of the invariant measure \\ of a viscous stochastic scalar conservation law}

\thanks{This work is partially supported by the French National Research Agency (ANR) under the programs ANR-19-CE40-0010 (QuAMProcs) and ANR-17-CE40-0030 (EFI)}

\author{S\'ebastien Boyaval}
\address{{\bf S\'ebastien Boyaval}\newline
{\rm \indent LHSV, Ecole des Ponts, CEREMA, EDF R et D, Chatou, France}\newline
{\rm \indent Inria Paris, Matherials}}
\email{\href{mailto:sebastien.boyaval@enpc.fr}{sebastien.boyaval@enpc.fr}}

\author{Sofiane Martel}
\address{{\bf Sofiane Martel}\newline
{\rm \indent Inria Rennes - Bretagne Atlantique, SIMSMART}}
\email{\href{mailto:sofiane.martel@inria.fr}{sofiane.martel@inria.fr}}

\author{Julien Reygner}
\address{{\bf Julien Reygner}\newline
{\rm \indent CERMICS, Ecole des Ponts, Marne-la-Vallée, France}}
\email{\href{mailto:julien.reygner@enpc.fr}{julien.reygner@enpc.fr}}

\keywords{Stochastic conservation laws, Invariant measure, Finite Volume schemes}

\subjclass[2010]{60H15,30R60,65M08,60H35}

\begin{document}

\begin{abstract}
We study the numerical approximation of the invariant measure of a viscous scalar conservation law, one-dimensional and periodic in the space variable, and stochastically forced with a white-in-time but spatially correlated noise. The flux function is assumed to be locally Lipschitz continuous and to have at most polynomial growth. The numerical scheme we employ discretises the SPDE according to a finite-volume method in space, and a split-step backward Euler method in time. As a first result, we prove the well-posedness as well as the existence and uniqueness of an invariant measure for both the semi-discrete and the split-step scheme.
Our main result is then the convergence of the invariant measures of the discrete approximations, as the space and time steps go to zero, towards the invariant measure of the SPDE, with respect to the second-order Wasserstein distance. We investigate rates of convergence theoretically, in the case where the flux function is globally Lipschitz continuous with a small Lipschitz constant, and numerically for the Burgers equation.
\end{abstract}

\maketitle


\section{Introduction}

\subsection{Viscous scalar conservation law with random forcing}\label{sec:intro1}

We consider the following viscous scalar conservation law with stochastic forcing on the one-dimensional torus $\T=\R / \Z$:
\begin{equation}\label{SSCL}
\dd u = - \partial_x A(u) \dd t + \nu \partial_{xx} u \dd t + \sum_{k\geq1} g^k \dd W^k(t) , \qquad x \in \T , \quad t \geq 0,
\end{equation}
where $(W^k)_{k\geq1}$ is a family of independent real Brownian motions and $(g^k)_{k\geq1}$ is a 
family of smooth functions on $\T$. 
The viscosity coefficient $\nu$ is assumed to be positive. Under regularity and polynomial growth assumptions on the \emph{flux} function $A$, Equation~\eqref{SSCL} is well-posed in a strong sense,
and there exists a unique invariant measure for its solution, see Proposition~\ref{SPDEresults} below,
which is proved in the companion paper \cite{MR19}.

In this work, we construct a numerical scheme, based on the finite-volume method, that allows to approximate this invariant measure. We place ourselves in the setting of~\cite{MR19} and first recall our main notations, assumptions and results. 


\subsubsection{Notations}

For any $p \in [1,+\infty]$, we denote by $L^p_0(\T)$ the set of functions $v \in L^p(\T)$ such that
\begin{equation*}
  \int_\T v(x)\dd x = 0,
\end{equation*}
and we write $\|v\|_{L^p_0(\T)}$ for the $L^p$ norm induced on $L^p_0(\T)$. For any integer $m \geq 0$, we denote by $H^m_0(\T)$ the intersection of $L^2_0(\T)$ with the Sobolev space $H^m(\T)$. Combining the Jensen inequality
\begin{equation}\label{eq:holder}
  \forall 1\leq p\leq q\leq+\infty, \qquad \|v\|_{L^p_0(\T)}\leq\|v\|_{L^q_0(\T)},
\end{equation}
with the gradient estimate
\begin{equation}\label{eq:strongerpoincare}
  \|v\|_{L^\infty_0(\T)}\leq\|\partial_xv\|_{L^1_0(\T)},
\end{equation}
we observe that $\|v\|_{H^m_0(\T)} := \|\partial_x^m v\|_{L^2_0(\T)}$ defines a norm on $H^m_0(\T)$, which is associated with the scalar product $\langle v, w\rangle_{H^m_0(\T)}$ and makes $H^m_0(\T)$ a separable Hilbert space.

We denote by $\N$ the set of non-negative integers, and by $\N^*$ the set of positive integers.


\subsubsection{Assumptions on the flux and the noise}

We shall assume that the flux function $A$ and the family of functions $(g^k)_{k \geq 1}$ satisfy the following condition.

\begin{assumption}[On $A$ and $(g^k)_{k \geq 1}$]\label{ass:Ag}
 The function $A : \R \to \R$ is of class $C^2$, its first derivative has at most polynomial growth:
    \begin{equation}\label{PGA}
      \exists \mathsf{C}_A>0 , \quad \exists \mathsf{p}_A \in \N^* , \quad \forall v \in \R , \qquad | A'(v) | \leq \mathsf{C}_A \left( 1+ |v|^{\mathsf{p}_A} \right),
    \end{equation}
    and its second derivative $A''$ is locally Lipschitz continuous on $\R$. Furthermore, for all $k \geq 1$, $g^k \in H^2_0(\T)$ and
  \begin{equation}\label{boundgk}
    \mathsf{D} := \sum_{k\geq1} \left\| g^k \right\|_{H^2_0(\T)}^2 < + \infty .
  \end{equation}
\end{assumption}

The assumptions~\eqref{PGA} and~\eqref{boundgk} will be needed in the arguments contained in this paper while the local Lipschitz continuity of $A''$ is only necessary for Proposition~\ref{SPDEresults}, which is proved in~\cite{MR19}.

The family of Brownian motions $(W^k)_{k \geq 1}$ is defined on a probability space $(\Omega,\mathcal{F},\P)$ equipped with a normal filtration $(\mathcal{F}_t)_{t \geq 0}$ in the sense of~\cite[Section~3.3]{DZ92}. Under Assumption~\ref{ass:Ag}, the series $\sum_k g^k W^k$ has to be understood as an $H^2_0(\T)$-valued Wiener process $W^Q$ with trace class covariance operator $Q : H^2_0(\T) \to H^2_0(\T)$ given by $Qv = \sum_{k\geq1}g^k\langle v,g^k\rangle_{H^2_0(\T)}$, see~\cite{MR19} for details. In the sequel, we shall call $W^Q$ a \emph{$Q$-Wiener process}.

\subsubsection{Main results from~\cite{MR19}}

Given a normed vector space $E$, $\mathcal B(E)$ denotes the Borel $\sigma$-field on $E$, $\mathcal P(E)$ denotes the set of probability measures over $(E,\mathcal B(E))$, and for $p\in[1,+\infty)$, $\mathcal P_p(E)$ denotes the subset of $\mathcal P(E)$ of probability measures with finite $p$-th order moment. The well-posedness of~\eqref{SSCL}, as well as the existence and uniqueness of an invariant measure for its solution, is proved in~\cite[Theorem 1, Theorem 2]{MR19}.
\begin{prop}[Well-posedness and invariant measure for~\eqref{SSCL}]\label{SPDEresults}
 Let $u_0 \in H^2_0(\T)$. Under Assumption~\ref{ass:Ag}, there exists a unique strong solution $(u(t))_{t \geq 0}$ to Equation~\eqref{SSCL} with initial condition $u_0$. That is, an $(\mathcal{F}_t)_{t \geq 0}$-adapted process $(u(t))_{t \geq 0}$ with values in $H^2_0(\T)$ such that, almost surely:
  \begin{enumerate}
    \item the mapping $t \mapsto u(t)$ is continuous from $[0,+\infty)$ to $H^2_0(\T)$;
    \item for all $t \geq 0$, the following equality holds:
    \begin{equation}\label{strongsolution}
    u(t) = u_0 + \int_0^t \left( - \partial_x A \left( u(s) \right) + \nu \partial_{xx} u(s) \right) \dd s + W^Q(t) .
    \end{equation}
  \end{enumerate}
  Furthermore, the process $(u(t))_{t\geq0}$ admits a unique invariant measure $\mu\in\mathcal P(H^2_0(\T))$, and if $v$ is a random variable with distribution $\mu$, then $\E[\|v\|_{H^2_0(\T)}^2]<+\infty$ and for all $p\in[1,+\infty)$, $\E[\|v\|_{L^p_0(\T)}^p]<+\infty$.
\end{prop}

Let us specify that for any $t\geq0$, the notation $u(t)$ shall always refer to an element of the space $H^2_0(\T)$. The scalar values taken by this function are denoted by $u(t,x)$, for $x\in\T$.

\subsection{Space discretisation}

To discretise~\eqref{SSCL} with respect to the space variable, we first fix $N \geq 1$, denote by $\T_N = \Z/N\Z$ the discrete torus, and define the regular mesh $\mathcal T_N$ on $\T$ by
\[ \mathcal T_N := \left\{ \left( x_{i-1} , x_i \right] , i \in \T_N \right\}, \qquad x_i := \frac{i}{N} , \]
where we identify $\T$ with $(0,1]$ and $\T_N$ with $\{1, \ldots, N\}$. Next, we introduce the finite dimensional space
\[ \R^N_0:=\{\mathbf v=(v_1,\dots,v_N)\in\R^N:v_1+\cdots+v_N=0\}, \]
on which we define, for any $p \in [1,+\infty]$, the normalised $\ell^p$ norm 
\begin{equation*}
  \|\mathbf{v}\|_{\ell^p_0(\T_N)} := \left(\frac{1}{N}\sum_{i \in \T_N} |v_i|^p\right)^{1/p} \quad \text{if $p<+\infty$}, \qquad \|\mathbf{v}\|_{\ell^\infty_0(\T_N)} := \max_{i \in \T_N} |v_i|.
\end{equation*}
The \emph{projection operator} $\Pi_N : L^1_0(\T) \to \R^N_0$ is defined by
\begin{equation*}
  \forall i \in \T_N, \qquad (\Pi_Nv)_i = N \int_{x_{i-1}}^{x_i} v(x)\dd x.
\end{equation*}
Notice that by Jensen's inequality, it satisfies the inequality
\begin{equation}\label{eq:isoPiN}
  \|\Pi_N v\|_{\ell^p_0(\T_N)} \leq \|v\|_{L^p_0(\T)},
\end{equation}
for any $p \in [1,+\infty]$.

Applying this operator to both sides of~\eqref{SSCL}, we get, for any $i \in \T_N$, 
\begin{equation*}\label{ACL}
  \dd \left(\Pi_N u(t)\right)_i =-N \left( A \left( u\left(t,x_i\right) \right) - A \left( u\left(t,x_{i-1}\right) \right) \right) \dd t + \nu N \left( \partial_x u \left(t,x_i\right) - \partial_x u\left(t,x_{i-1}\right) \right) \dd t + (\Pi_N W^Q(t))_i.
\end{equation*}
Let us denote by $\mathbf{U}^N(t) = (U^N_i(t))_{i \in \T_N}$ a vector whose purpose is to approximate the vector $\Pi_N u(t)$. The basic idea of finite-volume schemes consists in introducing a \emph{numerical flux} function $\overline A:\R^2\to\R$ such that the value $\overline A(U^N_i(t),U^N_{i+1}(t))$ aims to approximate the flux $A(u(t,x_i))$ of the conserved quantity between the two adjacent cells $(x_{i-1},x_i]$ and $(x_i,x_{i+1}]$. The function $\overline{A}$ satisfies certain properties which are stated in Assumption~\ref{ass:Abar} below. Given such a function, for any $\mathbf{v} \in \R^N_0$ we denote by $\overline{\mathbf{A}}^N(\mathbf{v})$ the vector with coordinates $\overline{A}(v_i,v_{i+1})$, $i \in \T_N$.

We then introduce the first-order forward and backward discrete derivative operators $\mathbf{D}_N^{(1,+)}$ and $\mathbf{D}_N^{(1,-)}$, defined by
\begin{equation*}
  \forall i \in \T_N, \qquad (\mathbf{D}_N^{(1,+)}\mathbf{v})_i := N(v_{i+1}-v_i), \quad (\mathbf{D}_N^{(1,-)}\mathbf{v})_i := N(v_i-v_{i-1}),
\end{equation*}
and the second-order centered discrete derivative operator $\mathbf{D}_N^{(2)}$, defined by
\begin{equation*}
  \mathbf{D}_N^{(2)}\mathbf{v} = \mathbf{D}_N^{(1,-)}\mathbf{D}_N^{(1,+)}\mathbf{v} = \mathbf{D}_N^{(1,+)}\mathbf{D}_N^{(1,-)}\mathbf{v}.
\end{equation*}
In the sequel, we shall sometimes call the quantities $\|\mathbf{D}^{(1,+)}\mathbf{v}\|_{\ell^2_0(\T)}$ and $\|\mathbf{D}^{(2)}\mathbf{v}\|_{\ell^2_0(\T)}$ the $h^1_0(\T_N)$ and $h^2_0(\T_N)$ norms of $\mathbf{v}$.

For all $k \geq 1$, we finally define the vector $\mathbf{g}^k := \Pi_N g^k \in \R^N_0$ and denote by $\mathbf{W}^{Q,N}$ the process $\Pi_N W^Q = \sum_k \mathbf{g}^k W^k$. This is a Wiener process in $\R^N_0$ with covariance
\begin{equation}\label{eq:covWQN}
  \E \left[ W^{Q,N}_i(t) W^{Q,N}_j(t) \right] = t \sum_{k\geq1} g^k_ig^k_j ,
\end{equation}
which is easily seen to be finite under Assumption~\ref{ass:Ag} (see~\eqref{eq:sigma-glob} below). 

These notations allow us to write a semi-discrete finite-volume approximation of~\eqref{SSCL} as the stochastic differential equation (SDE)
\begin{equation}\label{FVS}
  \dd \mathbf{U}^N(t) = - \mathbf{D}_N^{(1,-)} \overline{\mathbf{A}}^N(\mathbf{U}^N(t))\dd t + \nu \mathbf{D}_N^{(2)} \mathbf{U}^N(t) \dd t  + \dd \mathbf{W}^{Q,N}(t).
\end{equation}
We shall sometimes use the notation
\begin{equation*}
  \mathbf{b}(\mathbf{v}) := - \mathbf{D}_N^{(1,-)} \overline{\mathbf{A}}^N(\mathbf{v}) + \nu \mathbf{D}_N^{(2)} \mathbf{v}
\end{equation*}
for the drift of this SDE. Since both this vector field and the noise $\mathbf{W}^{Q,N}$ take their values in $\R^N_0$, we deduce that~\eqref{FVS} is conservative in the sense that if $\mathbf U^N(0)\in\R^N_0$, then for all $t\geq0$, $\mathbf U^N(t)\in\R^N_0$. 

We may now state our assumptions on the numerical flux.
\begin{assumption}[On $\overline{A}$]\label{ass:Abar}
 The function $\overline{A}$ belongs to $C^1(\R^2,\R)$, its first derivatives $\partial_1\bar A$ and $\partial_2\bar A$ are locally Lipschitz continuous on $\R^2$, and it satisfies the following properties.
\begin{itemize}
\item[(i)] Consistency:
\begin{equation}\label{eq:consistency}
\forall v\in\R,\qquad \overline{A}(v,v) = A(v).
\end{equation}
\item[(ii)] Monotonicity:
\begin{equation}\label{E1}
\forall v,w\in\R,\qquad \partial_1 \bar A (v,w) \geq 0 , \quad \partial_2 \bar A (v,w) \leq 0.
\end{equation}
\item[(iii)] Polynomial growth:
\begin{equation}\label{PG}
\exists \mathsf{C}_{\bar A}>0,\quad \exists \mathsf{p}_{\bar A} \in \N^*,\quad 
\forall v,w \in \R,\quad | \partial_1 \overline{A} (v,w) | \leq \mathsf{C}_{\bar A} ( 1 + |v|^{\mathsf{p}_{\bar A}} ) ,\quad | \partial_2 \overline{A} (v,w) | \leq \mathsf{C}_{\bar A} ( 1 + |w|^{\mathsf{p}_{\bar A}} ) .
\end{equation}
\end{itemize}
\end{assumption}

Note in particular that the numerical flux function, and therefore the drift of the SDE~\eqref{FVS},
is not globally Lipschitz continuous. Nevertheless, we prove in Proposition~\ref{SDEwellposed} that~\eqref{FVS} is well-posed under Assumption~\ref{ass:Abar}. The polynomial growth assumption is used in Proposition~\ref{prop:unifestim} to obtain uniform $h^2_0(\T_N)$ estimates on the invariant measure of $(\mathbf{U}^N(t))_{t \geq 0}$. This task will require uniform $\ell^p_0(\T_N)$ moment estimates, for large values of $p$, which will be established in Lemma~\ref{lem:recursivebound} and Proposition~\ref{prop:unifestim}. The fact that the polynomial growth assumption on $\partial_1 \overline{A} (v,w)$ (resp. $\partial_2 \overline{A} (v,w)$) is required to hold uniformly in $w$ (resp. $v$) may seem demanding, it is however well suited to flux-splitting schemes as the Engquist--Osher numerical flux described below.

\begin{remark}[Engquist--Osher numerical flux]
A notable class of numerical fluxes satisfying the monotonicity and polynomial growth conditions (under Assumption~\ref{ass:Ag}) are the flux-splitting schemes~\cite[Example~5.2]{EGH00}, among which a commonly employed example is the \textit{Engquist--Osher numerical flux} \cite{EO80} defined (for $A(0)=0$) by
\[ \forall v,w \in \R,\qquad \bar A_{\mathrm{EO}}(v,w) := \int_0^v [A'(v')]_+\dd v' - \int_0^w [A'(w')]_-\dd w'. \]
\end{remark}

\subsection{Space and time discretisation}

The second stage in constructing a numerical scheme for~\eqref{SSCL} is the time discretisation of the SDE~\eqref{FVS}. Considering a time step $\Delta t>0$ and a positive integer $n$, we introduce the notation 
\begin{equation}\label{eq:DeltaWQN}
  \Delta \mathbf W^{Q,N}_n := \mathbf W^{Q,N}(n\Delta t) - \mathbf W^{Q,N}((n-1)\Delta t).
\end{equation} 

As already noticed in \cite{MSH02}, explicit numerical schemes for SDEs with non-globally Lipschitz continuous coefficients do not preserve in general the long time stability, whereas implicit schemes are more robust. Therefore, since our main focus in this paper is to approximate invariant measures, we follow \cite{MSH02} and propose the following \textit{split-step stochastic backward Euler method}
\begin{equation}\label{SSBE}
\begin{dcases}
\mathbf U^{N,\Delta t}_{n+\frac12} =\mathbf U^{N,\Delta t}_n + \Delta t \mathbf b\left(\mathbf U^{N,\Delta t}_{n+\frac12}\right) , \\
\mathbf U^{N,\Delta t}_{n+1} =\mathbf U^{N,\Delta t}_{n+\frac 1 2} + \Delta\mathbf W^{Q,N}_{n+1} .
\end{dcases}
\end{equation}

The well-posedness of the scheme, \textit{i.e.} the existence and uniqueness of the value $\mathbf U^{N,\Delta t}_{n+\frac12}$ in the first line of~\eqref{SSBE}, is ensured by Proposition~\ref{SSBEwellposed}.

\subsection{Main results}\label{sec:mainresults}

Our first focus is on the long time behaviour of the processes $(\mathbf U^N(t))_{t\geq0}$ and $(\mathbf U^{N,\Delta t}_n)_{n\in\N}$. In this perspective, we state our first result.
\begin{theorem}[Existence and uniqueness of invariant measures for both schemes]\label{IM}
Under Assumptions~\ref{ass:Ag} and~\ref{ass:Abar}, the following two statements hold:
\begin{itemize}
 \item[(i)] for any $N\geq1$, the process $(\mathbf U^N(t))_{t\geq0}$ solution of the SDE~\eqref{FVS} admits a unique invariant measure $\vartheta_N \in \mathcal P(\R^N_0)$;
 \item[(ii)] for any $N\geq1$ and $\Delta t>0$, the sequence $(\mathbf U^{N,\Delta t}_n )_{n\in\N}$ defined by~\eqref{SSBE} admits a unique invariant measure $\vartheta_{N,\Delta t}\in\mathcal P(\R^N_0)$.
\end{itemize}
Moreover, for any $N\geq1$ and $\Delta t>0$, the measures $\vartheta_N$ and $\vartheta_{N,\Delta t}$ belong to $\mathcal P_2(\R^N_0)$.
\end{theorem}
The proofs for these two statements are given separately in Section~\ref{IMschemes}. The structure of the proof is the same as for~\cite[Theorem 2]{MR19} where we derived the existence and uniqueness of an invariant measure for the solution of~\eqref{SSCL} from two important properties: respectively the dissipativity of the $L^2_0(\T)$ norm of the solution, and an $L^1_0(\T)$ contraction property. In Lemma~\ref{L1Cdrift} below, we show that both of these properties are preserved at the discrete level. Therefore, we then prove the existence of an invariant measure with a tightness argument (uniform energy estimates and the Krylov--Bogoliubov theorem) and the uniqueness with a coupling argument. While the proof of existence is rather standard, our analysis of uniqueness depends on arguments which are more specific to the processes $(\mathbf{U}^N(t))_{t \geq 0}$ and $(\mathbf U^{N,\Delta t}_n )_{n\in\N}$. We insist on the fact that both proofs of existence and uniqueness crucially rely on the positivity of the viscosity coefficient $\nu$.

We now address the $\Delta t \to 0$ limit of $\vartheta_{N, \Delta t}$ and the $N \to +\infty$ limit of $\vartheta_N$. Since most of our results follow from $\ell^2_0(\T_N)$ or $L^2_0(\T)$ estimates, it is natural in our setting to work with the following distance on $\mathcal{P}_2(\ell^2_0(\T_N))$ or $\mathcal{P}_2(L^2_0(\T))$.

\begin{definition}[Wasserstein distance]\label{WD}
 Let $(E,\|\cdot\|_E)$ be a normed vector space and let $\alpha,\beta\in\mathcal P_2(E)$. The \emph{quadratic Wasserstein distance} between $\alpha$ and $\beta$ is defined by
 \[ W_2(\alpha,\beta):=\inf_{\pi\in\Pi(\alpha,\beta)} \left( \int_{E\times E} \left\|u-v\right\|_E^2 \dd\pi(u,v) \right)^{1/2} , \]
where $\Pi(\alpha,\beta)$ is the set of probability measures on $E\times E$ with marginals $\alpha$ and $\beta$:
\[ \Pi(\alpha,\beta) := \left\{ \pi\in\mathcal P_2\left(E\times E\right):\forall B\in\mathcal B\left( E \right),\pi(B\times E)=\alpha(B) \text{ and } \pi(E\times B)=\beta(B) \right\}.  \]
\end{definition}

The reader is referred to~\cite[Chapter 6]{Vil08} for further details on the Wasserstein distance, and in particular for the proof that it actually defines a distance on $\mathcal P_2(E)$. From now on, the spaces $\mathcal{P}_2(\ell^2_0(\T_N))$ and $\mathcal{P}_2(L^2_0(\T))$ are systematically endowed with the topology induced by the corresponding distance $W_2$.

As a first step to approximate numerically the measure $\mu$, we 
embed the measures $\vartheta_N$ and $\vartheta_{N,\Delta t}$ into $\mathcal P (L^2_0(\T))$. Let us denote by $\Psi_N:\R^N_0\to L^\infty_0(\T)$ the 
piecewise constant reconstruction operator defined by, for all $\mathbf{v} \in \R^N_0$,
\begin{equation*}
  \forall i \in \T_n, \quad \forall x \in (x_{i-1},x_i], \qquad  \Psi_N\mathbf{v}(x) := v_i.
\end{equation*}
Notice that for any $p \in [1,+\infty]$, 
\begin{equation}\label{eq:isoPsiN}
  \|\Psi_N\mathbf{v}\|_{L^p_0(\T)} = \|\mathbf{v}\|_{\ell^p_0(\T_N)},
\end{equation}
so that Theorem~\ref{IM} implies that the pushforward measures 
\begin{equation}\label{eq:muN}
  \mu_N := \vartheta_N \circ (\Psi_N)^{-1}, \qquad \mu_{N,\Delta t} := \vartheta_{N,\Delta t} \circ (\Psi_N)^{-1},
\end{equation}
belong to $\mathcal P_2(L^2_0(\T))$. Sections~\ref{section:convergence} and~\ref{section:convergence2} are devoted to the proof of our main result.

\begin{theorem}[Convergence of the invariant measures]\label{TCV}
Under Assumptions~\ref{ass:Ag} and~\ref{ass:Abar}, we have 
\begin{equation}\label{cv1}
\lim_{N\to\infty} \mu_N = \mu \qquad\text{in $\mathcal P_2(L^2_0(\T))$,}
\end{equation}
and moreover, for any $N\geq1$,
\begin{equation}\label{cv2}
\lim_{\Delta t\to0}\vartheta_{N,\Delta t}=\vartheta_N\qquad\text{in $\mathcal P_2(\R^N_0)$.}
\end{equation}
In short, we have the following approximation result:
\[ \lim_{N\to\infty} \lim_{\Delta t\to0} \mu_{N,\Delta t}=\mu\qquad\text{in $\mathcal P_2(L^2_0(\T))$.} \]
\end{theorem}

\begin{remark}
In Theorem~\ref{TCV}, $\mu$ is seen as a probability measure of $\mathcal P(L^2_0(\T))$ giving full weight to $H^2_0(\T)$, as opposed to Proposition~\ref{SPDEresults} where $\mu$ was seen as a probability measure of $\mathcal P(H^2_0(\T))$. The fact that both objects coincide follows from~\cite[Lemma~6]{MR19}.
\end{remark}

Let us briefly sketch the lines of our proof of~\eqref{cv1}. First, the positivity of the viscosity coefficient $\nu$ makes~\eqref{SSCL} parabolic, so that energy estimates in $L^2_0(\T)$ and $H^1_0(\T)$ are a natural tool to study this equation. 
In this perspective, we derive uniform (in $N$) discrete $\ell^p_0(\T_N)$, $h^1_0(\T_N)$ and $h^2_0(\T_N)$ bounds on $\vartheta_N$. They imply that the sequence $(\mu_N)_{N \geq 1}$ is relatively compact in $\mathcal{P}_2(L^2_0(\T))$. 
Using the finite-time convergence of the finite-volume scheme $\Psi_N\mathbf{U}^N(t)$ to the solution $u(t)$ of the SPDE~\eqref{SSCL}, we then show that any limit $\mu^*$ of a weakly converging subsequence of $(\mu_N)_{N \geq 1}$ is invariant for~\eqref{SSCL}, which allows to identify all these limits and leads to~\eqref{cv1}. This finite-time convergence result, which is an important step in our argument, is stated in Proposition~\ref{CFT}. It relies in particular on $H^2_0(\T)$ estimates on $u$, so that the framework of \emph{strong} solutions for~\eqref{SSCL} is well-suited to our approach.

The proof of~\eqref{cv2} follows the same approach, and the main finite-time convergence result is stated in Proposition~\ref{CFTSS}.

\begin{remark}[Ergodicity]\label{rm:ergo}
 As the invariant measure $\mu$ of the process $(u(t))_{t\geq0}$ is unique from Proposition~\ref{SPDEresults}, it is ergodic. In particular, a consequence of Birkhoff's ergodic theorem (see for instance~\cite[Theorem~1.2.3]{DZ96}) is that for any $\varphi\in L^1(\mu)$ and for $\mu$-almost every initial condition $u_0\in H^2_0(\T)$, almost surely,
 \[ \lim_{t\to\infty} \frac1t\int_0^t\varphi(u(s))\dd s = \E\left[\varphi(v)\right], \qquad\text{where $v\sim\mu$.} \]
 By virtue of Theorem~\ref{IM}, this property also holds at the discrete level: the sequence $(\mathbf U^{N, \Delta t}_n)_{n\in\N}$ satisfies for any $\varphi\in L^1(\vartheta_{N,\Delta t})$ and for $\vartheta_{N,\Delta t}$-almost every initial condition $\mathbf U^{N, \Delta t}_0\in\R^N_0$, almost surely,
 \[ \lim_{n\to\infty}\frac1n\sum_{l=0}^{n-1}\varphi(\mathbf U^{N, \Delta t}_l) = \E\left[\varphi(\mathbf V^{N, \Delta t})\right],\qquad\text{where $\mathbf V^{N, \Delta t}\sim\vartheta_{N,\Delta t}$.} \]
Thanks to this property, it is possible to approximate numerically expectations of functionals under the invariant measure by averaging in time the simulated process. We used this method to perform the numerical experiments presented in Section~\ref{s:newnum}.
\end{remark}

Complementing the convergence results of Theorem~\ref{TCV} with a quantitative rate in $N$ and $\Delta t$ is a natural question. As far as the convergence of $\mu_N$ to $\mu$ when $N \to +\infty$ is concerned, we sketch in Subsection~\ref{ss:rate} how our arguments may be adapted to yield a strong $L^2_0(\T)$ error estimate of order $1/N$ between $u^N(t)$ and $u(t)$, valid in the long time limit, in the case where the flux function $A$ is globally Lipschitz continuous with a small Lipschitz norm. This implies that $W_2(\mu_N,\mu) \leq C/N$. We show in Section~\ref{s:newnum} that this result is sharp in the case where $A=0$. In this section, we also study numerically the rate of convergence of $\vartheta_{N,\Delta t}$ to $\vartheta_N$ and observe a weak error of order $\Delta t$, even when the flux function is not small and not Lipschitz continuous, which is consistent with theoretical results by Kopec on a related class of split-step schemes~\cite{Kop14}.

\subsection{Review of literature}

Many results are found concerning the numerical approximation \emph{in finite time} of stochastic conservation laws. A particular case of interest is the stochastic Burgers equation which corresponds to the case of the flux function $A(v)=v^2/2$. Finite difference schemes are presented in~\cite{AG06,Hai11} to approximate its solution. When the viscosity coefficient is equal to zero, the SPDE falls into a different framework. Convergence of finite-volume schemes in this \emph{hyperbolic} case have been established both under the kinetic \cite{DV18,DV19,Dot17} and the entropic formulations \cite{BCG16,BCG16b}.

As regards the numerical approximation of the invariant measure of an SPDE, we may start by mentioning~\cite{CHW17} concerning the damped stochastic non-linear Schr\"odinger equation, where a spectral Galerkin method is used for the space discretisation and a modified implicit Euler scheme for the temporal discretisation. Several works by Br\'ehier and coauthors are devoted to the numerical approximation of the invariant measures of semi-linear SPDEs in Hilbert spaces perturbed with white noise~\cite{Bre14,BK16,BG16}, where spectral Galerkin and semi-implicit Euler methods are used. Those results hold under a global Lipschitz assumption on the nonlinearity. In the more recent works~\cite{CGW18,CHS18,Bre20}, non-Lipschitz nonlinearities are considered, but they still need to satisfy a one-sided Lipschitz condition.

In the present work, our assumptions on the flux function do not imply that the non-linear term is globally Lipschitz continuous in $L^2_0(\T)$ nor even one-sided Lipschitz continuous. In particular, the case of the Burgers equation is covered. However, Equation~\eqref{SSCL} satisfies an $L^1_0(\T)$ contraction property \cite[Proposition~5]{MR19} which may be viewed as a one-sided Lipschitz condition in the Banach space $L^1_0(\T)$.

\subsection{Outline of the paper and comments on the presentation}

Throughout the article, we always work under Assumptions~\ref{ass:Ag} and~\ref{ass:Abar}. We will not repeat these assumptions in the statements of our results.

Section~\ref{IMschemes} is dedicated to the proof of the well-posedness and of the existence and uniqueness of an invariant measure for the semi-discrete scheme~\eqref{FVS} and the split-step scheme~\eqref{SSBE}. The proof of Theorem~\ref{TCV} is then detailed in two separate sections. The convergence in space~\eqref{cv1} is proved in Section~\ref{section:convergence} and then, in Section~\ref{section:convergence2}, we prove the convergence with respect to the time step, \textit{i.e.} Equation~\eqref{cv2}. Numerical experiments investigating the rates of convergence in Theorem~\ref{TCV} are presented in Section~\ref{s:newnum}. The proofs of certain results which are not essential to the exposition of our arguments are gathered in Appendix.

In order to emphasise our original contributions, throughout the paper some arguments which are standard to either stochastic calculus or numerical analysis are omitted or merely sketched. A preliminary version of this work, with all proofs detailed, is available as Chapter~3 of~\cite{Mar19}, and references to this document are given whenever necessary (see also the first arXiv version of this paper). More numerical experiments, in particular regarding the turbulent behaviour of the process in its stationary regime, are also reported in Chapter~4 of~\cite{Mar19}.

\section{Semi-discrete and split-step schemes: well-posedness and invariant measure}\label{IMschemes}

Preliminary results are given in Subsection~\ref{sec:notations}. In Subsection~\ref{IMschemes:FVS}, we prove the well-posedness of the semi-discrete scheme $(\mathbf{U}^N(t))_{t \geq 0}$, and after establishing some properties of this process, we prove the existence and uniqueness of an invariant measure $\vartheta_N$ as well as the fact that necessarily, $\vartheta_N\in\mathcal P_2(\R^N_0)$. Similar results for the split-step scheme $(\mathbf{U}^{N,\Delta t}_n)_{n \in \N}$ are obtained in Subsection~\ref{IMschemes:SSBE}, which completes the proof of Theorem~\ref{IM}.

\subsection{Preliminary results}\label{sec:notations}

In this subsection, we state a few preliminary results which will be used throughout the paper.

\subsubsection{Algebraic identities} We define the normalised scalar product on $\R^N$ by
\begin{equation*}
  \langle \mathbf{v},\mathbf{w}\rangle_{\ell^2(\T_N)} = \frac{1}{N}\sum_{i\in \T_N} v_iw_i.
\end{equation*}
When both $\mathbf{v}$ and $\mathbf{w}$ belong to $\R^N_0$, we shall rather denote $\langle \mathbf{v},\mathbf{w}\rangle_{\ell^2_0(\T_N)}$.

The discrete derivative operators $\mathbf{D}^{(1,+)}_N$ and $\mathbf{D}^{(1,-)}_N$ satisfy the summation by parts identity
\begin{equation}\label{eq:sbp}
  \langle \mathbf{D}^{(1,+)}_N\mathbf{v}, \mathbf{w}\rangle_{\ell^2(\T_N)} = - \langle \mathbf{v}, \mathbf{D}^{(1,-)}_N\mathbf{w}\rangle_{\ell^2(\T_N)}.
\end{equation}
We shall also use the variant
\begin{equation}\label{eq:sbp2}
  \langle \mathbf{D}^{(1,+)}_N\mathbf{v}, \mathbf{D}^{(1,+)}_N \mathbf{w}\rangle_{\ell^2_0(\T_N)} = \langle \mathbf{D}^{(1,-)}_N\mathbf{v}, \mathbf{D}^{(1,-)}_N\mathbf{w}\rangle_{\ell^2_0(\T_N)} = -\langle \mathbf{D}^{(2)}_N\mathbf{v}, \mathbf{w}\rangle_{\ell^2(\T_N)}.
\end{equation}

\subsubsection{Discrete inequalities} The discrete Jensen inequality writes
\begin{equation}\label{eq:normorder}
  \forall 1\leq p\leq q \leq +\infty, \qquad \|\mathbf{v}\|_{\ell^p_0(\T_N)}\leq\|\mathbf{v}\|_{\ell^q_0(\T_N)},
\end{equation}
and we shall use the following version of the discrete Poincar\'e inequality:
\begin{equation}\label{eq:poinca-discr}
  \|\mathbf{v}\|_{\ell^2_0(\T_N)} \leq \|\mathbf{D}^{(1,+)}_N\mathbf{v}\|_{\ell^2_0(\T_N)} = \|\mathbf{D}^{(1,-)}_N\mathbf{v}\|_{\ell^2_0(\T_N)},
\end{equation}
which follows from~\eqref{eq:normorder} and 
\begin{equation}\label{eq:gradient-estim-discr}
  \|\mathbf{v}\|_{\ell^\infty_0(\T_N)} \leq \|\mathbf{D}^{(1,+)}_N\mathbf{v}\|_{\ell^1_0(\T_N)},
\end{equation}
which is the discrete version of the gradient estimate~\eqref{eq:strongerpoincare}.

\subsubsection{Properties of $\mathbf{b}$}

For any $z\in\R$, we write $\sgn(z):=\ind{z\geq0}-\ind{z<0}$. By extension, for $\mathbf v\in\R^N_0$, $\bm\sgn(\mathbf v)$ denotes the vector of $\{-1,+1\}^N$ defined by $(\bm\sgn(\mathbf v))_i=\sgn(v_i)$. The discretised drift $\mathbf b$ preserves two important features of Equation~\eqref{SSCL} that we will use repeatedly throughout this paper:
\begin{lemma}[Discrete contraction and dissipativity]\label{L1Cdrift}
  For all $\mathbf v,\mathbf w \in \R^N_0$, the function $\mathbf b$ satisfies
 \begin{itemize}
 \item[(i)] $\langle \bm\sgn(\mathbf v-\mathbf w) , \mathbf b(\mathbf v)-\mathbf b(\mathbf w) \rangle_{\ell^2(\T_N)} \leq 0$ ($\ell^1_0(\T_N)$ contraction);
 \item[(ii)] $\langle \mathbf v,\mathbf b(\mathbf v)\rangle_{\ell^2_0(\T_N)} \leq -\nu\|\mathbf D^{(1,+)}_N\mathbf v\|_{\ell^2_0(\T_N)}^2$ ($\ell^2_0(\T_N)$ dissipativity).
 \end{itemize}
\end{lemma}

The proof of Lemma~\ref{L1Cdrift} relies on the following result, the proof of which is postponed to Appendix~\ref{app}.

\begin{lemma}[Stability]\label{L1}
For any $\mathbf v \in \R^N_0$ and any $q \in 2\N^*$, we have
\[ \langle \mathbf{v}^{q-1}, \mathbf{D}^{(1,-)}\overline{\mathbf{A}}^N(\mathbf{v})\rangle_{\ell^2(\T_N)} \geq 0, \]
where the notation $\mathbf{v}^{q-1}$ refers to the vector with coordinates $(v_1^{q-1}, \ldots, v_N^{q-1})$.
\end{lemma}

We now detail the proof of Lemma~\ref{L1Cdrift}.

\begin{proof}[Proof of Lemma~\ref{L1Cdrift}.] (i) Let $\mathbf v,\mathbf w\in\R^N_0$. From the definition of $\mathbf b$ and~(\ref{eq:sbp}--\ref{eq:sbp2}), we write
  \begin{align*}
    &\langle\bm\sgn(\mathbf v-\mathbf w),\mathbf b(\mathbf v)-\mathbf b(\mathbf w)\rangle_{\ell^2(\T_N)}\\
    &=-\langle\bm\sgn(\mathbf v-\mathbf w),\mathbf{D}^{(1,-)}_N(\overline{\mathbf{A}}^N(\mathbf v)-\overline{\mathbf{A}}^N(\mathbf w))\rangle_{\ell^2(\T_N)} + \nu \langle\bm\sgn(\mathbf v-\mathbf w),\mathbf{D}^{(2)}_N(\mathbf v - \mathbf w)\rangle_{\ell^2(\T_N)}\\
    &= \langle\mathbf{D}^{(1,+)}_N\bm\sgn(\mathbf v-\mathbf w),\overline{\mathbf{A}}^N(\mathbf v)-\overline{\mathbf{A}}^N(\mathbf w)\rangle_{\ell^2(\T_N)} - \nu \langle\mathbf{D}^{(1,+)}_N\bm\sgn(\mathbf v-\mathbf w),\mathbf{D}^{(1,+)}_N(\mathbf v - \mathbf w)\rangle_{\ell^2_0(\T_N)}.
  \end{align*}
Observe that since the function $\sgn:\R\to\R$ is non-decreasing, 
\begin{equation*}
  \langle\mathbf{D}^{(1,+)}_N\bm\sgn(\mathbf v-\mathbf w),\mathbf{D}^{(1,+)}_N(\mathbf v - \mathbf w)\rangle_{\ell^2_0(\T_N)} \geq 0.
\end{equation*}
As for the other term, it follows from the monotonicity property of $\bar A$ that for any $i \in \T_N$,
\begin{equation*}
  \left(\sgn(v_{i+1}-w_{i+1})-\sgn(v_i-w_i)\right)(\overline{A}(v_i,v_{i+1})-\overline{A}(w_i,w_{i+1})) \leq 0.
\end{equation*}
Indeed, let us address for instance the case where $v_{i+1}\geq w_{i+1}$ and $v_i\leq w_i$. Then, on the one hand, we have $\sgn(v_{i+1}-w_{i+1})-\sgn(v_i-w_i)=2$. On the other hand, we have
\begin{align*}
\bar A(v_i,v_{i+1})-\bar A(w_i,w_{i+1})&=\left(\bar A(v_i,v_{i+1})-\bar A(v_i,w_{i+1})\right)+\left(\bar A(v_i,w_{i+1})-\bar A(w_i,w_{i+1})\right)\\
&=\int_{w_{i+1}}^{v_{i+1}}\partial_2\bar A(v_i,z)\dd z-\int_{v_i}^{w_i}\partial_1\bar A(z,w_{i+1})\dd z\leq 0.
\end{align*}
 The case where $v_{i+1}\leq w_{i+1}$ and $v_i\geq w_i$ is treated symmetrically.
  
(ii) Let $\mathbf v\in\R^N_0$. We have
\[ \langle\mathbf v,\mathbf b(\mathbf v)\rangle_{\ell^2_0(\T_N)}=-\langle\mathbf v,\mathbf{D}^{(1,-)}_N\overline{\mathbf{A}}^N(\mathbf v)\rangle_{\ell^2_0(\T_N)} + \nu\langle\mathbf v,\mathbf{D}^{(2)}_N\mathbf v\rangle_{\ell^2_0(\T_N)} . \]
Lemma~\ref{L1} with $q=2$ shows that the first term of the above decomposition is non-positive, and applying~\eqref{eq:sbp2} in the second term yields the result.
\end{proof}

\begin{remark}
The $\ell^2_0(\T_N)$ dissipativity property actually holds for the family of E-fluxes~\cite{MP03}, a larger family than the class of monotone numerical fluxes. The monotonicity assumption~\eqref{E1} seems however necessary as regards the $\ell^1_0(\T_N)$ contraction property.
\end{remark}

\subsubsection{Finiteness of the covariance of $\mathbf{W}^{Q,N}$}

At several places we shall need the estimates
\begin{equation}\label{eq:sigma-glob}
  \max_{i \in \T_N} \sum_{k \geq 1} (g_i^k)^2 \leq \mathsf{D}, \qquad \sum_{k \geq 1} \|\mathbf{g}^k\|_{\ell^2_0(\T_N)}^2 \leq \mathsf{D}, \qquad \sum_{k \geq 1} \|\mathbf{D}^{(1,+)}_N\mathbf{g}^k\|_{\ell^2_0(\T_N)}^2 \leq \mathsf{D}.
\end{equation}
The third estimate in~\eqref{eq:sigma-glob} follows from~\eqref{boundgk} and the fact that for any $k \geq 1$,
\begin{equation*}
  \|\mathbf{D}^{(1,+)}_N\mathbf{g}^k\|_{\ell^2_0(\T_N)}^2 \leq \|g^k\|_{H^1_0(\T)}^2 \leq \|g^k\|_{H^2_0(\T)}^2,
\end{equation*}
where the first inequality can be checked by a direct computation using Jensen's inequality and the second inequality follows from~\eqref{eq:holder} and~\eqref{eq:strongerpoincare}. The first and second estimates in~\eqref{eq:sigma-glob} are then consequences of the third estimate thanks to~\eqref{eq:gradient-estim-discr} and~\eqref{eq:poinca-discr}, respectively. These estimates prove for instance the finiteness of the sum in the right-hand side of~\eqref{eq:covWQN}.

\subsection{The semi-discrete scheme}\label{IMschemes:FVS}

In this subsection, we first show that the SDE~\eqref{FVS} has a unique global solution $(\mathbf{U}^N(t))_{t \geq 0}$. We then give uniform $\ell^p_0(\T_N)$ estimates on this process, which will be used at several places in the sequel of the paper. We finally prove the existence and the uniqueness of an invariant measure $\vartheta_N$ for $(\mathbf{U}^N(t))_{t \geq 0}$.

\subsubsection{Well-posedness of~\eqref{FVS}}

Since the function $\mathbf b$ is locally Lipschitz continuous, it is a standard result that there exists a unique strong solution $(\mathbf U^N(t))_{t\in[0,T^*)}$ to Equation~\eqref{FVS} defined up to a random explosion time $T^*$. That this solution is actually global in time usually follows from a Lyapunov-type condition. In our context, the presence of a viscous term in the SPDE~\eqref{SSCL} allows the use of energy methods based on the dissipation of the squared $L^2_0(\T)$ norm, see~\cite{MR19}. At the level of the SDE~\eqref{FVS}, denoting by $\mathcal{L}_N$ the associated infinitesimal generator, this fact is observed on the simple estimate
\begin{equation}\label{eq:L2H1}
  \mathcal{L}_N \|\mathbf{v}\|_{\ell^2_0(\T_N)}^2 = 2 \langle \mathbf{v}, \mathbf{b}(\mathbf{v})\rangle_{\ell^2_0(\T_N)} + \sum_{k \geq 1} \|\mathbf{g}^k\|_{\ell^2_0(\T_N)}^2 \leq -2\nu\|\mathbf{D}^{(1,+)}_N\mathbf{v}\|_{\ell^2_0(\T_N)}^2 + \mathsf{D} \leq -2\nu\|\mathbf{v}\|_{\ell^2_0(\T_N)}^2 + \mathsf{D},
\end{equation}
which follows from Lemma~\ref{L1Cdrift}, \eqref{eq:sigma-glob} and~\eqref{eq:poinca-discr}. This shows that the squared $\ell^2_0(\T_N)$ norm is a \emph{Lyapunov function} for $\mathcal{L}_N$, and implies the following statement.

\begin{prop}[Well-posedness of~\eqref{FVS}]\label{SDEwellposed}
Let $\mathbf U^N_0$ be an $\R^N_0$-valued, $\mathcal F_0$-measurable random variable. The stochastic differential equation~\eqref{FVS} admits a unique strong solution $(\mathbf U^N(t))_{t\geq0}$ taking values in $\R^N_0$ and with initial condition $\mathbf U^N_0$.
\end{prop}

The proof of Proposition~\ref{SDEwellposed} is omitted, we refer to~\cite[Proposition~3.15]{Mar19} for details.

\subsubsection{Moment estimates} In this paragraph, we prove the following uniform (in $N$) $\ell^p_0(\T_N)$ estimates on the process $\mathbf{U}^N$. For any $\mathbf{v} \in \R^N_0$, we recall the notation $\mathbf{v}^p = (v_1^p, \ldots, v_N^p)$ and take the convention that $\|\mathbf{v}\|^0_{\ell^0_0(\T_N)} = 1$.

\begin{lemma}[Moment estimates on the semi-discrete scheme]\label{lem:recursivebound}
Let $p\in2\N^*$ and let $\mathbf U^N_0$ be an $\mathcal F_0$-measurable random variable such that $\E[\|\mathbf U^N_0\|_{\ell^p_0(\T_N)}^p]<+\infty$. The solution $(\mathbf U^N(t))_{t\geq0}$ of~\eqref{FVS} with initial condition $\mathbf U^N_0$ satisfies the following estimates.
\begin{itemize}
\item[(i)] For all $t\geq0$,
 \begin{equation}\label{eq:rb}
   \begin{split}
     &\E\left[\left\|\mathbf U^N(t)\right\|_{\ell^p_0(\T_N)}^p\right]+\nu p \E\left[\int_0^t\left\langle\mathbf D^{(1,+)}_N\left((\mathbf U^N(s))^{p-1}\right),\mathbf D^{(1,+)}_N\mathbf U^N(s)\right\rangle_{\ell^2_0(\T_N)}\dd s\right]\\
     &\leq\E\left[\left\|\mathbf U^N_0\right\|_{\ell^p_0(\T_N)}^p\right]+\mathsf{D}\frac{p(p-1)}2\E\left[\int_0^t\left\|\mathbf U^N(s)\right\|_{{\ell^{p-2}_0(\T_N)}}^{p-2}\dd s\right].
   \end{split}
 \end{equation}
 \item[(ii)] There exist six positive constants $\mathsf{c}_0^{(p)}$, $\mathsf{c}_1^{(p)}$, $\mathsf{c}_2^{(p)}$, $\mathsf{d}_0^{(p)}$, $\mathsf{d}_1^{(p)}$ and $\mathsf{d}_2^{(p)}$, depending only on $\mathsf{D}$, $\nu$ and $p$ such that we have
 \begin{equation}\label{eq:integralinduction}
\forall t>0 ,\qquad \E \left[ \int_0^t\left\| \mathbf U^N(s)\right\|_{\ell^p_0(\T_N)}^p\dd s\right] \leq \mathsf{c}_0^{(p)} + \mathsf{c}_1^{(p)}\E\left[\left\|\mathbf U^N_0\right\|_{\ell^p_0(\T_N)}^p\right]+\mathsf{c}_2^{(p)}t , 
 \end{equation}
  and
  \begin{equation}\label{eq:supinduction}
 \forall t>0,\qquad \sup_{s\in[0,t]}\E\left[\left\|\mathbf U^N(s)\right\|_{\ell^p_0(\T_N)}^p\right] \leq \mathsf{d}_0^{(p)} + \mathsf{d}^{(p)}_1\E\left[\left\|\mathbf U^N_0\right\|_{\ell^p_0(\T_N)}^p\right]+\mathsf{d}^{(p)}_2t.
 \end{equation}
 \end{itemize}
\end{lemma}

The proof of Lemma~\ref{lem:recursivebound} relies on the following $\ell^p_0(\T_N)$ extension of the Poincar\'e inequality~\eqref{eq:poinca-discr}, the proof of which is postponed to Appendix~\ref{app}.

\begin{lemma}[$\ell^p_0(\T_N)$ Poincar\'e inequality]\label{comp}
 For any $\mathbf v\in\R^N_0$ and $p\in2\N^*$, we have
 \[ \left\langle\mathbf D^{(1,+)}_N\left(\mathbf v^{p-1}\right),\mathbf D^{(1,+)}_N\mathbf v\right\rangle_{\ell^2_0(\T_N)}\geq\frac{4(p-1)}{p^2}\|\mathbf v\|_{\ell^p_0(\T_N)}^p. \]
\end{lemma}

We sketch the lines of the proof of Lemma~\ref{lem:recursivebound} and refer to~\cite[Lemma~3.16]{Mar19} for details.

\begin{proof}[Sketch of the proof of Lemma~\ref{lem:recursivebound}]
  Let us fix $p\in2\N^*$ and apply the It\^o formula to $\|\mathbf{U}^N(t)\|^p_{\ell^p_0(\T_N)}$. We get
  \begin{align*}
    \dd \|\mathbf{U}^N(t)\|^p_{\ell^p_0(\T_N)} &= p\left\langle \mathbf{U}^N(t)^{p-1}, \left(-\mathbf{D}^{(1,-)}_N \overline{\mathbf{A}}^N(\mathbf{U}^N(t)) + \nu \mathbf{D}^{(2)}_N\mathbf{U}^N(t)\right)\dd t + \dd\mathbf{W}^{Q,N}(t)\right\rangle_{\ell^2(\T_N)}\\
    &\quad  + \frac{p(p-1)}{2}\left\langle \mathbf{U}^N(t)^{p-2}, \sum_{k \geq 1}(\mathbf{g}^k)^2\right\rangle_{\ell^2(\T_N)}\dd t.
  \end{align*}
  By Lemma~\ref{L1},
  \begin{equation*}
    \left\langle \mathbf{U}^N(t)^{p-1}, -\mathbf{D}^{(1,-)}_N \overline{\mathbf{A}}^N(\mathbf{U}^N(t))\right\rangle_{\ell^2(\T_N)} \leq 0;
  \end{equation*}
  by~\eqref{eq:sbp2},
  \begin{equation*}
    \left\langle \mathbf{U}^N(t)^{p-1}, \nu \mathbf{D}^{(2)}_N\mathbf{U}^N(t)\right\rangle_{\ell^2(\T_N)} = -\nu \left\langle \mathbf{D}^{(1,+)}_N\left(\mathbf{U}^N(t)^{p-1}\right), \mathbf{D}^{(1,+)}_N\mathbf{U}^N(t)\right\rangle_{\ell^2_0(\T_N)};
  \end{equation*}
  and by~\eqref{eq:sigma-glob},
  \begin{equation*}
    \left\langle \mathbf{U}^N(t)^{p-2}, \sum_{k \geq 1}(\mathbf{g}^k)^2\right\rangle_{\ell^2(\T_N)} \leq \mathsf{D} \|\mathbf{U}^N(t)\|^{p-2}_{\ell^{p-2}_0(\T_N)}.
  \end{equation*}
  Using a localisation argument to remove the stochastic integral when taking the expectation, we get~\eqref{eq:rb}.
  
  We deduce from~\eqref{eq:rb} and Lemma~\ref{comp} that for any $p \in 2\N^*$ and $t \geq 0$,
  \begin{equation*}
    \E\left[\left\|\mathbf U^N(t)\right\|_{\ell^p_0(\T_N)}^p\right] \leq\E\left[\left\|\mathbf U^N_0\right\|_{\ell^p_0(\T_N)}^p\right]+\mathsf{D}\frac{p(p-1)}2\E\left[\int_0^t\left\|\mathbf U^N(s)\right\|_{{\ell^{p-2}_0(\T_N)}}^{p-2}\dd s\right],
  \end{equation*}
  and
  \begin{equation*}
    \E\left[\int_0^t\left\|\mathbf D^{(1,+)}_N\mathbf U^N(s)\right\|^p_{\ell^p_0(\T_N)}\dd s\right]\leq \frac{p}{4\nu(p-1)}\E\left[\left\|\mathbf U^N_0\right\|_{\ell^p_0(\T_N)}^p\right]+\mathsf{D}\frac{p^2}{8\nu}\E\left[\int_0^t\left\|\mathbf U^N(s)\right\|_{{\ell^{p-2}_0(\T_N)}}^{p-2}\dd s\right].
  \end{equation*}
  This readily yields~\eqref{eq:integralinduction} and~\eqref{eq:supinduction} for $p=2$ and, along with the elementary inequality
  \begin{equation*}
    \|\mathbf U^N_0\|_{\ell^{p-2}_0(\T_N)}^{p-2} \leq 1 + \|\mathbf U^N_0\|_{\ell^p_0(\T_N)}^p,
  \end{equation*}
  then serves as the basis for an inductive argument on even values of $p$ to obtain the general form of~\eqref{eq:integralinduction} and~\eqref{eq:supinduction}.
\end{proof}

\subsubsection{Feller property and existence of an invariant measure}

In order to deduce from the Lyapunov condition~\eqref{eq:L2H1} the existence of an invariant measure for $(\mathbf{U}^N(t))_{t \geq 0}$, it is necessary to check that it is a Feller process~\cite[Theorem~4.21]{Hai06}, which in our case is a consequence of the following $\ell^1_0(\T_N)$ contraction property.

\begin{prop}[$\ell^1_0(\T_N)$ contraction for $\mathbf{U}^N$]\label{L1C}
 Two solutions $(\mathbf U^N(t))_{t\geq0}$ and $(\mathbf V^N(t))_{t\geq0}$ of~\eqref{FVS}, driven by the same Wiener process $\mathbf W^{Q,N}$, with possibly different initial conditions, satisfy almost surely
 \[ \forall  0\leq s\leq t, \qquad \left\|\mathbf U^N(t)-\mathbf V^N(t) \right\|_{\ell^1_0(\T_N)} \leq \left\|\mathbf U^N(s) -\mathbf V^N(s) \right\|_{\ell^1_0(\T_N)}.\]
\end{prop}

\begin{proof}
 Since $(\mathbf U^N(t))_{t\geq0}$ and $(\mathbf V^N(t))_{t\geq0}$ are driven by the same Wiener process, then $(\mathbf U^N(t)-\mathbf V^N(t))_{t\geq0}$ is an absolutely continuous process, and 
 \[ \dd(\mathbf U^N(t)-\mathbf V^N(t))=\left(\mathbf b(\mathbf U^N(t))-\mathbf b(\mathbf V^N(t))\right)\dd t . \]
 In particular, we can write for all $t\geq0$,
 \[ \frac\dd{\dd t}\left\|\mathbf U^N(t)-\mathbf V^N(t)\right\|_{\ell^1_0(\T_N)}=\left\langle\bm\sgn\left(\mathbf U^N(t)-\mathbf V^N(t)\right),\mathbf b(\mathbf U^N(t))-\mathbf b(\mathbf V^N(t))\right\rangle_{\ell^2(\T_N)}\leq 0 , \]
 where the inequality comes from Lemma~\ref{L1Cdrift}.(i), and the result follows by integrating in time.
\end{proof}

\begin{corollary}[Feller property for $\mathbf{U}^N$]\label{feller}
The solution $(\mathbf U^N(t))_{t\geq0}$ of Equation~\eqref{FVS} satisfies the Feller property, \emph{i.e.} for any continuous and bounded function $\varphi:\R^N_0\to\R$ and any $t\geq0$, the mapping
\[\mathbf u_0\in\R^N_0\longmapsto\E_{\mathbf u_0}\left[\varphi(\mathbf U^N(t))\right]\in\R \]
is continuous and bounded, where the notation $\E_{\mathbf u_0}$ indicates that $\mathbf U^N(0)=\mathbf u_0$.
\end{corollary}

The Lyapunov condition~\eqref{eq:L2H1} implies that for any initial condition $\mathbf{u}_0 \in \R^N_0$, the family of probability measures $(\frac{1}{t}\int_0^t \mathbb{P}_{\mathbf{u}_0}(\mathbf{U}^N(s) \in \cdot)\dd s)_{t \geq 0}$ is tight (see~\cite[Section~3.2.2]{Mar19} for details), which by the Krylov--Bogoliubov criterion~\cite[Corollary~3.1.2]{DZ96} implies the following result.

\begin{prop}[Existence of an invariant measure for $\mathbf{U}^N$]\label{existenceIMSDE}
The solution $(\mathbf U^N(t))_{t\geq0}$ of~\eqref{FVS} admits an invariant measure $\vartheta_N\in\mathcal P_2(\R^N_0)$. 
\end{prop}
\subsubsection{Uniqueness of the invariant measure}\label{sec:uniqueness}

The proof of uniqueness of the invariant measure $\vartheta_N$ relies on the intermediary Lemmas~\ref{entrance} and~\ref{entranceLB} which are stated below. They follow from standard but a bit lengthy computation, similar instances of which can be found in~\cite{DV15,MR19}. For the sake of legibility of the whole argument, we therefore postpone their proofs to Appendix~\ref{app:uniq}.

\begin{lemma}[Hitting any neighbourhood of $0$ with positive probability]\label{entrance}
 Let $(\mathbf U^N(t))_{t\geq0}$ and $(\mathbf V^N(t))_{t\geq0}$ be two solutions of~\eqref{FVS} driven by the same Wiener process. Then, for all $M>0$ and all $\varepsilon>0$, there exists $t_{\varepsilon,M}>0$ such that
 \[ p_{\varepsilon,M} := \inf \P_{(\mathbf u_0,\mathbf v_0)} \left( \left\|\mathbf U^N(t_{\varepsilon,M}) \right\|_{\ell^1_0(\T_N)} + \left\|\mathbf V^N(t_{\varepsilon,M}) \right\|_{\ell^1_0(\T_N)} \leq \varepsilon \right) > 0, \]
 where the notation $\P_{(\mathbf u_0,\mathbf v_0)}$ indicates that $\mathbf{U}^N(0)=\mathbf{u}_0$ and $\mathbf{V}^N(0)=\mathbf{v}_0$ and the infimum is taken over pairs of initial conditions $(\mathbf u_0,\mathbf v_0)$ such that $\|\mathbf u_0\|_{\ell^2_0(\T_N)}\vee\|\mathbf v_0\|_{\ell^2_0(\T_N)}\leq M$.
\end{lemma}

In the setting of Lemma~\ref{entrance}, we let, for any $M \geq 0$,
\begin{equation}\label{eq:entrancetimedef}
\tau_M := \inf \left\{ t\geq0 : \left\|\mathbf U^N(t) \right\|_{\ell^2_0(\T_N)} \vee \left\|\mathbf V^N(t) \right\|_{\ell^2_0(\T_N)} \leq M \right\} .
\end{equation}

\begin{lemma}[Almost sure entrance in some ball]\label{entranceLB}
 There exists $M>0$ such that for any deterministic initial conditions $\mathbf u_0$, $\mathbf v_0 \in \R^N_0$, $\tau_M < +\infty$ almost surely.
\end{lemma}

We now detail how Lemmas~\ref{entrance} and~\ref{entranceLB} allow to complete the proof of uniqueness of an invariant measure for $\mathbf{U}^N$.

\begin{proof}[Proof of the uniqueness of an invariant measure for $\mathbf{U}^N$.] We start by fixing $\varepsilon>0$, to which we associate the quantities $t_{\varepsilon,M}$ and $p_{\varepsilon,M}$ defined at Lemma~\ref{entrance}, where $M$ has been defined at Lemma~\ref{entranceLB}. Let $(\mathbf U^N(t))_{t\geq0}$ and $(\mathbf V^N(t))_{t\geq0}$ start respectively from arbitrary deterministic initial conditions $\mathbf u_0$ and $\mathbf v_0$ and be driven by the same Wiener process. 
We define the increasing stopping time sequence
 \[ T_1 := \tau_M \]
 \[ T_2 := \inf \left\{ t \geq T_1 + t_{\varepsilon,M} : \left\|\mathbf U^N(t) \right\|_{\ell^2_0(\T_N)} \vee \left\|\mathbf V^N(t) \right\|_{\ell^2_0(\T_N)} \leq M \right\} \]
 \[ T_3 := \inf \left\{ t \geq T_2 + t_{\varepsilon,M} : \left\|\mathbf U^N(t) \right\|_{\ell^2_0(\T_N)} \vee \left\|\mathbf V^N(t) \right\|_{\ell^2_0(\T_N)} \leq M \right\} \]
 \[ \vdots \]
 
By the strong Markov property and Lemma~\ref{entranceLB}, each term of this sequence is finite almost surely. We introduce the event
\begin{equation*}
  E_j = \left\{\left\|\mathbf U^N(T_j+t_{\varepsilon,M}) \right\|_{\ell^1_0(\T_N)} + \left\|\mathbf V^N(T_j+t_{\varepsilon,M}) \right\|_{\ell^1_0(\T_N)} > \varepsilon\right\}
\end{equation*}
and claim that
 \begin{equation*}
 \forall J \in \N^* , \qquad \P \left(\cap_{j=1}^J E_j\right) \leq (1-p_{\varepsilon,M})^J .
 \end{equation*}
 Indeed, it is true for $J=1$ thanks to the strong Markov property and Lemma~\ref{entrance}:
\begin{equation*}
  \P \left( E_1 \right) = \E \left[ \P \left( \left\|\mathbf U^N(\tau_M+t_{\varepsilon,M}) \right\|_{\ell^1_0(\T_N)} + \left\|\mathbf V^N(\tau_M+t_{\varepsilon,M}) \right\|_{\ell^1_0(\T_N)} > \varepsilon   | \mathcal F_{\tau_M} \right) \right] \leq 1-p_{\varepsilon,M} ,
\end{equation*}
and the general case follows by induction. 
Letting $J\to+\infty$, we get
\begin{equation*}
 \P \left( \cap_{j=1}^\infty E_j \right) = \lim_{J\to\infty} \P \left( \cap_{j=1}^J E_j \right)  \leq \lim_{J\to\infty} (1-p_{\varepsilon,M})^J = 0 ,
\end{equation*}
and consequently, almost surely there exists some $t=T_j + t_{\varepsilon,M}$ such that 
\begin{equation*}
  \|\mathbf U^N(t) -\mathbf V^N(t) \|_{\ell^1_0(\T_N)} \leq \|\mathbf U^N(t) \|_{\ell^1_0(\T_N)} + \|\mathbf V^N(t) \|_{\ell^1_0(\T_N)} \leq \varepsilon.
\end{equation*}

Now recall that thanks to Proposition~\ref{L1C}, $\|\mathbf U^N(t)-\mathbf V^N(t)\|_{\ell^1_0(\T_N)}$ is non-increasing in time almost surely. Since $\varepsilon$ has been chosen arbitrarily, we deduce that $\|\mathbf U^N(t)-\mathbf V^N(t)\|_{\ell^1_0(\T_N)}$ converges almost surely to $0$ as $t\to+\infty$ when the initial conditions are deterministic, and since this assertion is true for any pair of initial conditions, it also holds for random and $\mathcal F_0$-measurable initial conditions. Let $\phi:\ell^1_0(\T_N)\to\R$ be a Lipschitz continuous and bounded test function, with Lipschitz constant $\mathsf{L}_\phi$. We have in particular, almost surely,
\begin{equation}\label{testfunction}
\lim_{t\to\infty}\left|\phi(\mathbf U^N(t))-\phi(\mathbf V^N(t))\right|\leq \mathsf{L}_\phi\lim_{t\to\infty}\left\|\mathbf U^N(t)-\mathbf V^N(t)\right\|_{\ell^1_0(\T_N)}=0.
\end{equation}

To conclude the proof, assume that there exist two invariant measures $\nu^{(1)}_N$ and $\nu^{(2)}_N$ for the solution of~\eqref{FVS}, and take random initial conditions $\mathbf U^N_0$ and $\mathbf V^N_0$ with distributions $\nu^{(1)}_N$ and $\nu^{(2)}_N$ respectively. We have for all $t\geq0$,
\[ \left|\E\left[\phi\left(\mathbf U^N_0\right)\right]-\E\left[\phi\left(\mathbf V^N_0\right)\right]\right|=\left|\E\left[\phi\left(\mathbf U^N(t)\right)\right]-\E\left[\phi\left(\mathbf V^N(t)\right)\right]\right|\leq\E\left[\left|\phi\left(\mathbf U^N(t)\right)-\phi\left(\mathbf V^N(t)\right)\right|\right] . \]
Letting $t$ go to $+\infty$, by~\eqref{testfunction} and the dominated convergence theorem, we have
\[ \left|\E\left[\phi\left(\mathbf U^N_0\right)\right]-\E\left[\phi\left(\mathbf V^N_0\right)\right]\right|\leq\lim_{t\to\infty}\E\left[\left|\phi\left(\mathbf U^N(t)\right)-\phi\left(\mathbf V^N(t)\right)\right|\right] =0. \]
As a consequence, $\mathbf U^N_0$ and $\mathbf V^N_0$ have the same distribution, meaning that $\nu^{(1)}_N=\nu^{(2)}_N$.
\end{proof}

\begin{remark}
  This proof shows in addition that for any initial distribution, the law of $\mathbf{U}^N(t)$ converges, when $t \to +\infty$, to the invariant measure $\vartheta_N$.
\end{remark}

\subsection{The split-step scheme}\label{IMschemes:SSBE}

In this subsection, we first show that the implicit equation in~\eqref{SSBE} has a unique solution, which ensures the well-posedness of the sequence $(\mathbf{U}^{N,\Delta t}_n)_{n \in \N}$. We then prove the existence and the uniqueness of an invariant measure $\vartheta_{N,\Delta t}$ for this Markov chain. The general organisation of the arguments is rather close to Subsection~\ref{IMschemes:FVS}, and we only emphasise which points have to be adapted. 

\subsubsection{Well-posedness} The following preliminary result ensures that the scheme~\eqref{SSBE} is well-posed.

\begin{prop}[Well-posedness of~\eqref{SSBE}]\label{SSBEwellposed}
For any $\Delta t>0$ and $\mathbf v\in \R^N_0$, there exists a unique $\mathbf w \in \R^N_0$ such that $\mathbf w =\mathbf v + \Delta t\mathbf b(\mathbf w)$.
\end{prop}

\begin{proof}
\emph{Uniqueness.} It is a straightforward consequence of Lemma~\ref{L1Cdrift}.(i): if $\mathbf w_1$ and $\mathbf w_2$ are two solutions, then
\begin{equation*}
\|\mathbf w_1-\mathbf w_2 \|_{\ell^1_0(\T_N)} = \left\langle\bm\sgn(\mathbf w_1-\mathbf w_2),\mathbf w_1-\mathbf w_2\right\rangle_{\ell^2(\T_N)} = \Delta t \langle \bm\sgn (\mathbf w_1-\mathbf w_2) ,\mathbf b(\mathbf w_1)-\mathbf b(\mathbf w_2) \rangle_{\ell^2(\T_N)} \leq 0 .
\end{equation*}

\emph{Existence.} The mapping $\bm{\mathrm{Id}}-\Delta t\mathbf b:\R^N_0\to\R^N_0$ is continuous. Furthermore, by Lemmas~\ref{L1Cdrift}.(ii) and~\eqref{eq:poinca-discr}, we have for all $\mathbf w\in\R^N_0$,
\begin{equation*}
\frac{\langle (\bm{\mathrm{Id}}-\Delta t\mathbf b)(\mathbf w),\mathbf w\rangle_{\ell^2_0(\T_N)}}{\|\mathbf w\|_{\ell^2_0(\T_N)}} = \|\mathbf w\|_{\ell^2_0(\T_N)}-\Delta t\frac{\langle\mathbf b(\mathbf w),\mathbf w\rangle_{\ell^2_0(\T_N)}}{\|\mathbf w\|_{\ell^2_0(\T_N)}} \geq\|\mathbf w\|_{\ell^2_0(\T_N)}+\nu\Delta t\frac{\|\mathbf D^{(1,+)}_N\mathbf w\|_{\ell^2_0(\T_N)}^2}{\|\mathbf w\|_{\ell^2_0(\T_N)}}\geq(1+\nu\Delta t)\|\mathbf w\|_{\ell^2_0(\T_N)} .
\end{equation*}
Thus, as a consequence of~\cite[Theorem~3.3]{Dei85}, $\bm{\mathrm{Id}}-\Delta t\mathbf b$ is surjective in $\R^N_0$ and, for any $\mathbf v\in\R^N_0$, there exists $\mathbf w\in\R^N_0$ such that $\mathbf w=\mathbf v+\Delta t\mathbf b(\mathbf w)$.
\end{proof}

\subsubsection{Existence of an invariant measure}

We first prove an $\ell^1_0(\T_N)$ contraction property in order to deduce the Feller property for $(\mathbf U^{N, \Delta t}_n)_{n\in\N}$. 

\begin{lemma}[$\ell^1_0(\T_N)$ contraction for $\mathbf{U}^{N,\Delta t}$]\label{L1CSSBE}
 Let $(\mathbf U^{N, \Delta t}_n )_{n\in\N}$ and $(\mathbf V^{N, \Delta t}_n )_{n\in\N}$ be two solutions of~\eqref{SSBE}, constructed with the same sequence of noise increments $(\Delta \mathbf W^{Q,N})_{n \in \N^*}$. Then, almost surely and for any $n\in\N$,
 \[ \left\|\mathbf U^{N, \Delta t}_{n+1} -\mathbf V^{N, \Delta t}_{n+1} \right\|_{\ell^1_0(\T_N)} \leq \left\|\mathbf U^{N, \Delta t}_n -\mathbf V^{N, \Delta t}_n \right\|_{\ell^1_0(\T_N)} . \]
\end{lemma}
\begin{proof}
From Equations~\eqref{SSBE} and Lemma~\ref{L1Cdrift}.(ii), we write
\begin{align*}
&\left\|\mathbf U^{N, \Delta t}_{n+1} -\mathbf V^{N, \Delta t}_{n+1} \right\|_{\ell^1_0(\T_N)}\\
&= \left\|\mathbf U^{N, \Delta t}_{n+\frac 1 2} -\mathbf V^{N, \Delta t}_{n+\frac 1 2} \right\|_{\ell^1_0(\T_N)} \\
&= \left\langle \bm\sgn \left(\mathbf U^{N, \Delta t}_{n+\frac 1 2} -\mathbf V^{N, \Delta t}_{n+\frac 1 2} \right) ,\mathbf U^{N, \Delta t}_{n+\frac 1 2} -\mathbf V^{N, \Delta t}_{n+\frac 1 2} \right\rangle_{\ell^2(\T_N)} \\
&= \left\langle \bm\sgn \left(\mathbf U^{N, \Delta t}_{n+\frac 1 2} -\mathbf V^{N, \Delta t}_{n+\frac 1 2} \right) ,\mathbf U^{N, \Delta t}_n -\mathbf V^{N, \Delta t}_n \right\rangle_{\ell^2(\T_N)} + \Delta t \left\langle \bm\sgn \left(\mathbf U^{N, \Delta t}_{n+\frac 1 2} -\mathbf V^{N, \Delta t}_{n+\frac 1 2} \right) ,\mathbf b \left(\mathbf U^{N, \Delta t}_{n+\frac 1 2} \right) -\mathbf b\left(\mathbf V^{N, \Delta t}_{n+\frac 1 2} \right) \right\rangle_{\ell^2(\T_N)} \\
&\leq \left\langle \bm\sgn \left(\mathbf U^{N, \Delta t}_{n+\frac 1 2} -\mathbf V^{N, \Delta t}_{n+\frac 1 2} \right) ,\mathbf U^{N, \Delta t}_n -\mathbf V^{N, \Delta t}_n \right\rangle_{\ell^2(\T_N)} \\
&\leq \left\|\mathbf U^{N, \Delta t}_n-\mathbf V^{N, \Delta t}_n \right\|_{\ell^1_0(\T_N)} .\qedhere
\end{align*}
\end{proof}
\begin{remark}
     The choice of the split-step backward Euler scheme is essential for the $\ell^1_0(\T_N)$ contraction property to hold. Indeed, consider for instance two sequences $(\mathbf{\tilde U}^{N, \Delta t}_n)_{n\in\N}$ and $(\mathbf{\tilde V}^{N, \Delta t}_n)_{n\in\N}$ built via an explicit Euler method, that is,
     \[ \mathbf{\tilde U}^{N, \Delta t}_{n+1} = \mathbf{\tilde U}^{N, \Delta t}_n+\Delta t\mathbf b\left(\mathbf{\tilde U}^{N, \Delta t}_n\right)+\Delta \mathbf W^{Q,N}_{n+1} \]
     (and naturally, the same construction for $(\mathbf{\tilde V}^{N, \Delta t}_n)_{n\in\N}$), then the expansion of the $\ell^1_0(\T_N)$ distance gives
     \begin{align*}
       \left\| \mathbf{\tilde U}^{N, \Delta t}_{n+1} -\mathbf{\tilde V}^{N, \Delta t}_{n+1} \right\|_{\ell^1_0(\T_N)} &= \bm\sgn \left\langle \left(\mathbf{\tilde U}^{N, \Delta t}_{n+1} -\mathbf{\tilde V}^{N, \Delta t}_{n+1} \right) ,\mathbf{\tilde U}^{N, \Delta t}_n -\mathbf{\tilde V}^{N, \Delta t}_n \right\rangle_{\ell^2(\T_N)}\\
       &\quad + \Delta t\left\langle \bm\sgn \left(\mathbf{\tilde U}^{N, \Delta t}_{n+1} -\mathbf{\tilde V}^{N, \Delta t}_{n+1} \right) ,\mathbf b \left(\mathbf{\tilde U}^{N, \Delta t}_n \right) -\mathbf b\left(\mathbf{\tilde V}^{N, \Delta t}_n \right) \right\rangle_{\ell^2(\T_N)} .
     \end{align*}
     Thus, we would need to control the second term of the right-hand side in the above equation, which is delicate given that $\mathbf b$ is not globally Lipschitz.
    \end{remark}

As for the semi-discrete scheme, Lemma~\ref{L1CSSBE} induces the following property.
\begin{corollary}[Feller property for $\mathbf{U}^{N,\Delta t}$]\label{cor:SSBEfeller}
The solution $(\mathbf U^{N, \Delta t}_n)_{n\in\N}$ of~\eqref{SSBE} has the Feller property.
\end{corollary}

The existence of an invariant measure follows again from the application of the Krylov--Bogoliubov theorem to a Lyapunov condition similar to~\eqref{eq:L2H1}.

\begin{prop}[Existence of an invariant measure for $\mathbf{U}^{N,\Delta t}$]\label{existence}
For all $n \geq 1$, we have
\begin{equation}\label{eq:kb}
\frac 1 n \sum_{l=0}^{n-1} \E \left[ \left\|\mathbf D^{(1,+)}_N\mathbf U^{N,\Delta t}_{l+1} \right\|_{\ell^2_0(\T_N)}^2 \right] \leq \frac{1}{2n\nu\Delta t} \left\| \mathbf u_0 \right\|_{\ell^2_0(\T_N)}^2  + \frac{\mathsf{D}}{2\nu} + \Delta t \mathsf{D} .
\end{equation}
As a consequence, the sequence $(\mathbf U^{N,\Delta t}_n)_{n\in\N}$ admits an invariant measure $\vartheta_{N,\Delta t}\in\mathcal P_2(\R^N_0)$. 
\end{prop}
\begin{proof}
Starting from the first equation in~\eqref{SSBE}, we have
\[ \left\| \mathbf U^{N,\Delta t}_{n+\frac 1 2} - \Delta t \mathbf b\left(\mathbf U^{N,\Delta t}_{n+\frac12}\right) \right\|_{\ell^2_0(\T_N)}^2 = \left\|\mathbf U^{N,\Delta t}_n \right\|_{\ell^2_0(\T_N)}^2 , \]
by expanding the left-hand side, we derive the inequality
\[ \left\|\mathbf U^{N,\Delta t}_{n+\frac 1 2} \right\|_{\ell^2_0(\T_N)}^2 \leq \left\|\mathbf U^{N,\Delta t}_n \right\|_{\ell^2_0(\T_N)}^2 +2\Delta t \left\langle\mathbf b\left(\mathbf U^{N,\Delta t}_{n+\frac 1 2}\right) ,\mathbf U^{N,\Delta t}_{n+\frac 1 2} \right\rangle_{\ell^2_0(\T_N)} . \]
Using Lemma~\ref{L1Cdrift}.(ii), we get
\begin{equation}\label{FB0}
\left\|\mathbf U^{N,\Delta t}_{n+\frac 1 2} \right\|_{\ell^2_0(\T_N)}^2 \leq \left\|\mathbf U^{N,\Delta t}_n \right\|_{\ell^2_0(\T_N)}^2 -2\nu \Delta t \left\|\mathbf D^{(1,+)}_N\mathbf U^{N,\Delta t}_{n+\frac 1 2} \right\|_{\ell^2_0(\T_N)}^2 .
\end{equation}
Now, from the second equation in~\eqref{SSBE}, we have
\begin{equation}\label{SEN}
\left\|\mathbf U^{N,\Delta t}_{n+1} \right\|_{\ell^2_0(\T_N)}^2 = \left\|\mathbf U^{N,\Delta t}_{n+\frac 1 2} \right\|_{\ell^2_0(\T_N)}^2 + 2 \left\langle\mathbf U^{N,\Delta t}_{n+\frac 1 2} , \Delta\mathbf W^{Q,N}_{n+1} \right\rangle_{\ell^2_0(\T_N)} + \left\| \Delta\mathbf W^{Q,N}_{n+1} \right\|_{\ell^2_0(\T_N)}^2 .
\end{equation}
Injecting Inequality~\eqref{FB0} into Equation~\eqref{SEN}, we get
\begin{equation}\label{eq:discretederivative}
\left\|\mathbf U^{N,\Delta t}_{n+1}\right\|_{\ell^2_0(\T_N)}^2-\left\|\mathbf U^{N,\Delta t}_n\right\|_{\ell^2_0(\T_N)}^2\leq-2\nu\Delta t\left\|\mathbf D^{(1,+)}_N\mathbf U^{N,\Delta t}_{n+\frac12}\right\|_{\ell^2_0(\T_N)}^2+2\left\langle\mathbf U^{N,\Delta t}_{n+\frac12},\Delta\mathbf W^{Q,N}_{n+1}\right\rangle_{\ell^2_0(\T_N)}+\left\|\Delta\mathbf W^{Q,N}_{n+1}\right\|_{\ell^2_0(\T_N)}^2.
\end{equation}
By definition of $\mathbf W^{Q,N}$ and from~\eqref{eq:sigma-glob}, we have
\begin{equation}\label{NB}
\E \left[ \left\| \Delta\mathbf W^{Q,N}_{n+1} \right\|_{\ell^2_0(\T_N)}^2 \right] = \Delta t\sum_{k\geq1}\|\mathbf{g}^k\|^2_{\ell^2_0(\T_N)} \leq \mathsf{D}\Delta t .
\end{equation}
On the other hand, the variables $\mathbf U^{N,\Delta t}_{n+\frac 1 2}$ and $\Delta\mathbf W^{Q,N}_{n+1}$ are independent, so that taking the expectation in~\eqref{eq:discretederivative} yields
\[ \E \left[ \left\|\mathbf U^{N,\Delta t}_{n+1} \right\|_{\ell^2_0(\T_N)}^2 \right] - \E \left[ \left\|\mathbf U^{N,\Delta t}_n \right\|_{\ell^2_0(\T_N)}^2 \right] \leq -2\nu \Delta t \E \left[ \left\|\mathbf D^{(1,+)}_N\mathbf U^{N,\Delta t}_{n+\frac 1 2} \right\|_{\ell^2_0(\T_N)}^2 \right] +\mathsf{D}\Delta t , \]
which is valid for any $n\in\N$, so that we get a telescopic sum:
\begin{equation*}
  \begin{aligned}
  \E \left[ \left\|\mathbf U^{N,\Delta t}_n \right\|_{\ell^2_0(\T_N)}^2 \right] - \left\| \mathbf u_0 \right\|_{\ell^2_0(\T_N)}^2 &= \sum_{l=0}^{n-1} \left( \E \left[ \left\|\mathbf U^{N,\Delta t}_{l+1} \right\|_{\ell^2_0(\T_N)}^2 \right] - \E \left[ \left\|\mathbf U^{N,\Delta t}_l \right\|_{\ell^2_0(\T_N)}^2 \right] \right)\\
  & \leq -2\nu \Delta t \sum_{l=0}^{n-1} \E \left[ \left\|\mathbf D^{(1,+)}_N\mathbf U^{N,\Delta t}_{l+\frac 1 2} \right\|_{\ell^2_0(\T_N)}^2 \right] + n \Delta t \mathsf{D} . 
  \end{aligned}
\end{equation*}
Hence,
\begin{equation}\label{almostKB}
2\nu \Delta t \sum_{l=0}^{n-1} \E \left[ \left\|\mathbf D^{(1,+)}_N\mathbf U^{N,\Delta t}_{l+\frac 1 2} \right\|_{\ell^2_0(\T_N)}^2 \right] \leq \left\| \mathbf u_0 \right\|_{\ell^2_0(\T_N)}^2 +n\Delta t \mathsf{D} .
\end{equation}
Besides, since the random variables $\mathbf U^{N,\Delta t}_{l+\frac 1 2}$ and $\Delta\mathbf W^{Q,N}_{l+1}$ are independent,
\begin{equation}\label{a2}
\E \left[ \left\|\mathbf D^{(1,+)}_N\mathbf U^{N,\Delta t}_{l+1} \right\|_{\ell^2_0(\T_N)}^2 \right] = \E \left[ \left\|\mathbf D^{(1,+)}_N\mathbf U^{N,\Delta t}_{l+\frac 1 2} \right\|_{\ell^2_0(\T_N)}^2 \right] + \E \left[ \left\|\mathbf D^{(1,+)}_N \Delta\mathbf W^{Q,N}_{l+1} \right\|_{\ell^2_0(\T_N)}^2 \right] ,
\end{equation}
and by~\eqref{eq:sigma-glob},
\begin{equation*}
 \E \left[ \left\|\mathbf D^{(1,+)}_N \Delta\mathbf W^{Q,N}_{l+1} \right\|_{\ell^2_0(\T_N)}^2 \right] =\Delta t\sum_{k\geq1}\left\|\mathbf D^{(1,+)}_N\mathbf{g}^k\right\|_{\ell^2_0(\T_N)}^2 \leq \Delta t\mathsf{D}.
 \end{equation*}
Injecting this bound into~\eqref{a2}, and~\eqref{a2} into~\eqref{almostKB}, we get~\eqref{eq:kb}.

Since $\|\mathbf D^{(1,+)}_N\cdot\|_{\ell^2_0(\T_N)}$ defines a norm on $\R^N_0$ and since from Corollary~\ref{cor:SSBEfeller}, the sequence $(\mathbf U^{N,\Delta t}_n)_{n\in\N}$ has the Feller property, the existence of an invariant measure $\vartheta_{N,\Delta t}\in\mathcal P_2(\R^N_0)$ now follows from the Krylov--Bogoliubov theorem.
\end{proof}

\subsubsection{Uniqueness of the invariant measure}

Similarly to Subsection~\ref{IMschemes:FVS}, we first state the intermediary Lemmas~\ref{proba} and~\ref{entranceSSBE}.

\begin{lemma}[Hitting any neighbourhood of $0$ with positive probability]\label{proba}
Let $(\mathbf U^{N,\Delta t}_n)_{n\in\N}$ and $(\mathbf V^{N,\Delta t}_n)_{n\in\N}$ be two solutions of~\eqref{SSBE} constructed with the same sequence of noise increments $(\Delta \mathbf W^{Q,N}_n)_{n \in \N^*}$. For any $\varepsilon>0$ and any $M>0$, there exists $n_{\varepsilon,M}\in\N$ such that
\[ p_{\varepsilon,M} := \inf \P_{(\mathbf u_0,\mathbf v_0)} \left( \left\|\mathbf U^{N,\Delta t}_{n_{\varepsilon,M}} \right\|_{\ell^1_0(\T_N)} + \left\|\mathbf V^{N,\Delta t}_{n_{\varepsilon,M}} \right\|_{\ell^1_0(\T_N)} \leq \varepsilon \right) > 0 , \]
where the notation $\P_{(\mathbf u_0,\mathbf v_0)}$ indicates that $\mathbf{U}^{N,\Delta t}_0=\mathbf{u}_0$ and $\mathbf{V}^{N,\Delta t}_0=\mathbf{v}_0$ and the infimum is taken over pairs of initial conditions $(\mathbf u_0,\mathbf v_0)$ such that $\|\mathbf u_0\|_{\ell^2_0(\T_N)}\vee\|\mathbf v_0\|_{\ell^2_0(\T_N)}\leq M$.
\end{lemma}

In the setting of Lemma~\ref{proba}, we now define
\[ \eta_M := \inf \left\{ n \in\N : \left\|\mathbf U^{N,\Delta t}_{n+1} \right\|_{\ell^2_0(\T_N)} \vee \left\| \mathbf V^{N,\Delta t}_{n+1} \right\|_{\ell^2_0(\T_N)} \leq M \right\}. \]
The following lemma is the time-discrete version of Lemma~\ref{entranceLB}. The proof is omitted as it is very similar to its time-continuous counterpart.

\begin{lemma}[Almost sure entrance in some ball]\label{entranceSSBE}
There exists $M>0$ such that for any initial conditions $\mathbf u_0$, $\mathbf v_0\in\R^N_0$ for the sequences $(\mathbf U^{N,\Delta t}_n)_{n\in\N}$ and $(\mathbf V^{N,\Delta t}_n)_{n\in\N}$, $\eta_M<+\infty$ almost surely.
\end{lemma}

The proof of Lemma~\ref{proba} is quite different from the proof of Lemma~\ref{entrance}, therefore it is detailed in Appendix~\ref{app:uniq}. On the contrary, Lemma~\ref{entranceSSBE} essentially follows from the same arguments as Lemma~\ref{entranceLB} and therefore we omit its proof. Finally, given Lemmas~\ref{L1CSSBE},~\ref{proba} and~\ref{entranceSSBE}, the proof of the uniqueness of the invariant measure $\vartheta_{N,\Delta t}$ of the split-step scheme is an obvious adaptation of the proof for the semi-discrete scheme.

\subsubsection{Conclusion}

In order to conclude Section~\ref{IMschemes}, let us summarise the main arguments forming the proof of Theorem~\ref{IM}. Well-posedness of both approximations $(\mathbf U^N(t))_{t\geq0}$ and $(\mathbf U^{N,\Delta t}_n)_{n\geq0}$ are stated respectively in Propositions~\ref{SDEwellposed} and~\ref{SSBEwellposed}. Then, from the $\ell^1_0(\T_N)$-contraction property (Propositions~\ref{L1C} and~\ref{L1CSSBE}), it is inferred that both processes are Feller (Corollaries~\ref{feller} and~\ref{cor:SSBEfeller}). The Feller property combined with the Lyapunov-like estimates~\eqref{eq:L2H1} and~\eqref{eq:kb} then lead to the existence of an invariant measure (Propositions~\ref{existenceIMSDE} and~\ref{existence}).

The proof of uniqueness consists in a coupling argument that relies on two preliminary steps. The first one is showing that two coupled solutions hit any neighbourhood of $0$ with positive probability (Lemmas~\ref{entrance} and~\ref{proba}). The second one is showing that there exists a ball which both coupled processes attain almost surely, in finite time, whatever their initial conditions (Lemmas~\ref{entranceLB} and~\ref{entranceSSBE}). The final argument establishing uniqueness of the invariant measure from those two steps is displayed only for the continuous time case at the end of Section~\ref{sec:uniqueness}.

Observe that the moment estimates for $p>2$ from Lemma~\ref{lem:recursivebound} have not been used yet. However, their role is crucial in the sequel. They will allow us to establish the regularity estimates uniformly in $N$ which will be at the core of our convergence argument.

\section{Convergence of invariant measures: semi-discrete scheme towards SPDE}\label{section:convergence}

This section is dedicated to the proof of the first statement in Theorem~\ref{TCV}, namely the convergence of $\mu_N$ to $\mu$. The general sketch of the proof is detailed in Subsection~\ref{ss:sketchofproof}. Subsections~\ref{ss:pftight} and~\ref{ss:CFT} contain the proofs of the main arguments. A discussion on the rate of convergence associated with Theorem~\ref{TCV} is then provided in Subsection~\ref{ss:rate}.

\subsection{General sketch of the proof}\label{ss:sketchofproof}

The first ingredient of the proof is the following series of uniform estimates on $\vartheta_N$.

\begin{prop}[Uniform $\ell^p_0(\T_N)$, $h^1_0(\T_N)$ and $h^2_0(\T_N)$ estimates on $\vartheta_N$]\label{prop:unifestim}
  For any $N\geq1$, let $\mathbf{V}^N$ be a random variable in $\R^N_0$ with distribution $\vartheta_N$. For all $p \in [1,+\infty)$, there exists a constant $\mathsf{C}^{0,p} \in [0,+\infty)$ such that
  \begin{equation*}
   \sup_{N\geq1} \E\left[\|\mathbf{V}^N\|^p_{\ell^p_0(\T_N)}\right] \leq \mathsf{C}^{0,p}.
  \end{equation*}
  Besides, there exist constants $\mathsf{C}^{1,2}, \mathsf{C}^{2,2} \in [0,+\infty)$ such that
  \begin{equation*}
    \sup_{N\geq1}\E\left[\|\mathbf{D}^{(1,+)}_N\mathbf{V}^N\|^2_{\ell^2_0(\T_N)}\right] \leq \mathsf{C}^{1,2}, \qquad \sup_{N\geq1}\E\left[\|\mathbf{D}^{(2)}_N\mathbf{V}^N\|^2_{\ell^2_0(\T_N)}\right] \leq \mathsf{C}^{2,2}.
  \end{equation*}
\end{prop}

\begin{remark}[Notation for constants]
  Throughout this section and the next one, we use the superscript indices $m,p$ in the notation for constants in order to refer to bounds over Sobolev $W^{m,p}$ norms.
\end{remark}

For all $N \geq 1$, we recall the definition~\eqref{eq:muN} of the measure $\mu_N \in \mathcal{P}_2(L^2_0(\T))$. The discrete regularity estimates from Proposition~\ref{prop:unifestim} actually ensure the continuous spatial regularity of random variables drawn from any limiting distribution of $(\mu_N)_{N\geq1}$, as stated in the following result.

\begin{corollary}[Relative compactness and $L^p_0(\T)$, $H^1_0(\T)$ and $H^2_0(\T)$ estimates on $\mu_N$]\label{cor:relcom}
  The sequence $(\mu_N)_{N \geq 1}$ is relatively compact in $\mathcal{P}_2(L^2_0(\T))$, and any subsequential limit $\mu^*$ has the property that if $v^*$ is a random variable in $L^2_0(\T)$ with distribution $\mu^*$, then for all $p \in [1,+\infty)$,
  \begin{equation*}
    \E\left[\|v^*\|^p_{L^p_0(\T)}\right] \leq \mathsf{C}^{0,p},
  \end{equation*}
  and
  \begin{equation*}
    \E\left[\|v^*\|^2_{H^1_0(\T)}\right] \leq \mathsf{C}^{1,2}, \qquad \E\left[\|v^*\|^2_{H^2_0(\T)}\right] \leq \mathsf{C}^{2,2}.
  \end{equation*}
\end{corollary}

As a consequence of Corollary~\ref{cor:relcom}, in order to prove that $\mu_N$ converges in $\mathcal{P}_2(L^2_0(\T))$ to the unique invariant measure $\mu$ of~\eqref{SSCL}, it suffices to show that any subsequential limit $\mu^*$ of this sequence coincides with $\mu$. To this aim, we let $\mu^*$ be such a limit, and for convenience we still denote by $(\mu_N)_{N \geq 1}$ the extracted subsequence which converges to $\mu^*$. We then apply the Skorohod representation theorem to the subsequence $(\mu_N)_{N \geq 1}$ to construct, on the same probability space, a sequence of $\mathcal{F}_0$-measurable random variables $\mathbf{U}^N_0 \in \R_0^N$ and a random variable $u^*_0 \in L^2_0(\T)$ such that:
\begin{itemize}
  \item for all $N \geq 1$, $\mathbf{U}^N_0 \sim \vartheta_N$ (or equivalently, $\Psi_N\mathbf{U}^N_0 \sim \mu_N$),
  \item $u^*_0 \sim \mu^*$,
  \item the sequence $\Psi_N \mathbf{U}^N_0$ converges almost surely to $u^*_0$ in $L^2_0(\T)$. 
\end{itemize}
Notice that by Corollary~\ref{cor:relcom}, $u_0^* \in H^2_0(\T)$, almost surely, which thus allows to take this random variable as an initial condition for~\eqref{SSCL}. Up to taking the product of this probability space with another probability space on which a $Q$-Wiener process $(W^Q(t))_{t \geq 0}$ is defined, which will thus be independent from $\mathcal{F}_0$, one may then consider the solution $(u^*(t))_{t\geq0}$ to~\eqref{SSCL}, with initial condition $u^*_0$ and driven by $(W^Q(t))_{t \geq 0}$. We will still use $\P$ and $\E$ to denote respectively the probability and the expectation on this product probability space. For all $N \geq 1$, we let $(\mathbf{U}^N(t))_{t \geq 0}$ be the solution to~\eqref{FVS} with initial condition $\mathbf{U}^N_0 \sim \mathcal{\nu}_N$ and driven by the Wiener process $\mathbf{W}^{Q,N} = \Pi_N W^Q$. We finally denote by $u^N(t) = \Psi_N \mathbf{U}^N(t)$ the piecewise constant reconstruction of $\mathbf{U}^N(t)$ on $\T$.

\begin{prop}[Finite-time convergence of $u^N(t)$]\label{CFT}
  In the setting introduced above, for all $t \geq 0$,
  \begin{equation*}
    \lim_{N \to +\infty} \E\left[\|u^N(t)-u^*(t)\|^2_{L^2_0(\T)}\right] = 0.
  \end{equation*}
\end{prop}

Proposition~\ref{CFT} implies in particular that the law of $u^N(t)$ converges in $\mathcal{P}_2(L^2_0(\T))$ to the law of $u^*(t)$. But since $u^N(t) = \Psi_N \mathbf{U}^N(t)$, and the process $(\mathbf{U}^N(t))_{t \geq 0}$ is stationary, its law does not depend on $t$. Therefore, the law of $u^*(t)$ does not depend on $t$ either; in other words, $\mu^*$ is an invariant measure for~\eqref{SSCL}. By the uniqueness result of Proposition~\ref{SPDEresults}, we deduce that $\mu^*=\mu$ and the proof of~\eqref{cv1} is completed.

\begin{remark}
This argument shows that any subsequential limit $\mu^*$ of $(\mu_N)_{N \geq 1}$ is an invariant measure for~\eqref{SSCL}, therefore it provides an alternative proof for the existence part in \cite[Theorem 2]{MR19}. The uniqueness part remains crucial to identify all subsequential limits and obtain the convergence of the sequence $(\mu_N)_{N \geq 1}$.
\end{remark}

\subsection{Proofs of Proposition~\ref{prop:unifestim} and Corollary~\ref{cor:relcom}}\label{ss:pftight}

We first detail the proof of Proposition~\ref{prop:unifestim}.

\begin{proof}[Proof of Proposition~\ref{prop:unifestim}]
  The proof is divided in 4 steps. The $\ell^p_0(\T_N)$ estimate is derived in Step~1. An intermediary result is stated in Step~2, and the $h^1_0(\T_N)$ and $h^2_0(\T_N)$ estimates are respectively derived in Steps~3 and~4.
  
  \emph{Step~1: $\ell^p_0(\T_N)$ estimate.} The argument is inspired from~\cite[Proposition~4.24]{Hai06} and relies on the finite-time uniform estimates from Lemma~\ref{lem:recursivebound}. For any $M>0$, $p\in2\N^*$ and $t>0$, we first write
  \begin{align*}
    \E\left[\left\|\mathbf V^N\right\|_{\ell^p_0(\T_N)}^p\wedge M\right]&=\frac1t\int_0^t\int_{\R^N_0}\E_{\mathbf u_0}\left[\left\|\mathbf U^N(s)\right\|_{\ell^p_0(\T_N)}^p\wedge M\right]\dd\vartheta_N(\mathbf u_0)\dd s\\
    &=\int_{\R^N_0}\frac1t\int_0^t\E_{\mathbf u_0}\left[\left\|\mathbf U^N(s)\right\|_{\ell^p_0(\T_N)}^p\wedge M\right]\dd s\dd\vartheta_N(\mathbf u_0)\\
    &\leq\int_{\R^N_0}\left(\frac1t\int_0^t\E_{\mathbf u_0}\left[\left\|\mathbf U^N(s)\right\|_{\ell^p_0(\T_N)}^p\right]\dd s\right)\wedge M\dd\vartheta_N(\mathbf u_0)\\
    &\leq\int_{\R^N_0}\left(\frac1t\mathsf{c}_0^{(p)}+\mathsf{c}_1^{(p)}\frac{\|\mathbf u_0\|_{\ell^p_0(\T_N)}^p}t+\mathsf{c}_2^{(p)}\right)\wedge M\dd\vartheta_N(\mathbf u_0),
  \end{align*}
  where we have used~\eqref{eq:integralinduction} in Lemma~\ref{lem:recursivebound} at the last line. Letting $t\to+\infty$, we get from the dominated convergence theorem,
  \begin{equation*}
    \E\left[\left\|\mathbf V^N\right\|_{\ell^p_0(\T_N)}^p\wedge M\right]\leq\int_{\R^N_0}\lim_{t\to\infty}\left(\frac1t\mathsf{c}_0^{(p)}+\mathsf{c}_1^{(p)}\frac{\|\mathbf u_0\|_{\ell^p_0(\T_N)}^p}t+\mathsf{c}_2^{(p)}\right)\wedge M\dd\vartheta_N(\mathbf u_0)=\mathsf{c}_2^{(p)}\wedge M\leq \mathsf{c}_2^{(p)} =: \mathsf{C}^{0,p},
  \end{equation*}
  and the result for $p\in2\N^*$ follows by letting $M\to+\infty$ and using the monotone convergence theorem. This result extends readily to the general case $p\in[1,+\infty)$ by using the Jensen inequality.
  
  \emph{Step~2: intermediary estimate.} Let $p \in 2\N^*$. Taking $\mathbf{U}^N_0 = \mathbf{V}^N$ in~\eqref{eq:rb}, and using the result of Step~1, we obtain by stationarity
  \begin{equation}\label{eq:statrec}
    \nu p \E\left[\left\langle\mathbf D^{(1,+)}_N\left((\mathbf V^N)^{p-1}\right),\mathbf D^{(1,+)}_N\mathbf V^N\right\rangle_{\ell^2_0(\T_N)}\right] \leq \mathsf{D}\frac{p(p-1)}2\E\left[\left\|\mathbf V^N\right\|_{{\ell^{p-2}_0(\T_N)}}^{p-2}\right] \leq \mathsf{D}\frac{p(p-1)}2 \mathsf{C}^{0,p-2},
  \end{equation}
  with the convention that $\mathsf{C}^{0,0}=1$.
  
  \emph{Step~3: $h^1_0(\T_N)$ estimate.} Take $p=2$ in~\eqref{eq:statrec} to get
  \begin{equation*}
    \E\left[\|\mathbf D^{(1,+)}_N\mathbf V^N\|_{\ell^2_0(\T_N)}^2\right] \leq \frac{2\mathsf{D}}{\nu} =: \mathsf{C}^{1,2}.
  \end{equation*}
  
  \emph{Step~4: $h^2_0(\T_N)$ estimate.} Let $(\mathbf U^N(t))_{t\geq0}$ be the solution of~\eqref{FVS} with initial distribution $\vartheta_N$. By It\^o's formula, for all $t\geq0$,
\begin{equation}\label{itoH1}
  \begin{aligned}
    \left\|\mathbf D^{(1,+)}_N\mathbf U^N(t)\right\|_{\ell^2_0(\T_N)}^2 &=\left\|\mathbf D^{(1,+)}_N\mathbf U^N_0\right\|_{\ell^2_0(\T_N)}^2 +2\int_0^t\left\langle\mathbf D^{(1,+)}_N\mathbf U^N(s),\mathbf D^{(1,+)}_N \mathbf b(\mathbf U^N(s))\right\rangle_{\ell^2_0(\T_N)}\dd s\\
    &\quad +2\int_0^t\left\langle\mathbf D^{(1,+)}_N\mathbf U^N(s),\dd \left(\mathbf D^{(1,+)}_N\mathbf W^{Q,N}\right)(s)\right\rangle_{\ell^2_0(\T_N)} +t\sum_{k\geq1}\left\|\mathbf D^{(1,+)}_N\mathbf{g}^k\right\|_{\ell^2_0(\T_N)}^2.
  \end{aligned}
\end{equation}
The third term of the right-hand side is a martingale since
\begin{equation}\label{martingale}
 \sum_{k\geq1}\E\left[\int_0^t\left\langle\mathbf D^{(1,+)}_N\mathbf U^N(s), \mathbf D^{(1,+)}_N\mathbf{g}^k\right\rangle_{\ell^2_0(\T_N)}^2\dd s\right] \leq t\left(\sum_{k\geq1}\left\|\mathbf D^{(1,+)}_N\mathbf{g}^k\right\|_{\ell^2_0(\T_N)}^2\right)\E\left[\left\|\mathbf D^{(1,+)}_N\mathbf V^N\right\|_{\ell^2_0(\T_N)}^2\right] \leq t\mathsf{D}\mathsf{C}^{1,2}<+\infty,
\end{equation}
where we have used the stationarity of $(\mathbf U^N(t))_{t\geq0}$, Inequality~\eqref{eq:sigma-glob}, and the $h^1_0(\T_N)$ estimate from Step~3. Thus, taking the expectation in~\eqref{itoH1} and using the stationarity and Inequality~\eqref{eq:sigma-glob} again, we get
\begin{equation*}
  -2\E\left[ \left\langle\mathbf D^{(1,+)}_N\mathbf V^N,\mathbf D^{(1,+)}_N\mathbf b(\mathbf V^N)\right\rangle_{\ell^2_0(\T_N)} \right] = \sum_{k\geq1}\left\| \mathbf D^{(1,+)}_N\mathbf{g}^k\right\|_{\ell^2_0(\T_N)}^2 \leq \mathsf{D}.\]

For any $\mathbf{v} \in \R^N_0$, we have, by~\eqref{eq:sbp2},
\begin{align*}
  \left\langle\mathbf D^{(1,+)}_N\mathbf v,\mathbf D^{(1,+)}_N\mathbf b(\mathbf v)\right\rangle_{\ell^2_0(\T_N)} &= -\left\langle\mathbf D^{(2)}_N\mathbf v,\mathbf b(\mathbf v)\right\rangle_{\ell^2_0(\T_N)}\\
  &= -\left\langle\mathbf D^{(2)}_N\mathbf v,\left(-\mathbf{D}^{(1,-)}_N \overline{\mathbf{A}}^N(\mathbf{v}) + \nu \mathbf D^{(2)}_N\mathbf v\right)\right\rangle_{\ell^2_0(\T_N)}\\
  &= \left\langle\mathbf D^{(2)}_N\mathbf v,\mathbf{D}^{(1,-)}_N \overline{\mathbf{A}}^N(\mathbf{v})\right\rangle_{\ell^2_0(\T_N)} - \nu\|\mathbf D^{(2)}_N\mathbf v\|^2_{\ell^2_0(\T_N)},
\end{align*}
so that 
\begin{align*}
  2\nu \E\left[\|\mathbf D^{(2)}_N\mathbf V^N\|^2_{\ell^2_0(\T_N)}\right] &\leq 2 \E\left[\left\langle\mathbf D^{(2)}_N\mathbf V^N,\mathbf{D}^{(1,-)}_N \overline{\mathbf{A}}^N(\mathbf{V}^N)\right\rangle_{\ell^2_0(\T_N)}\right] + \mathsf{D}\\
  &\leq 2 \sqrt{\E\left[\|\mathbf D^{(2)}_N\mathbf V^N\|_{\ell^2_0(\T_N)}^2\right]}\sqrt{\E\left[\|\mathbf{D}^{(1,-)}_N \overline{\mathbf{A}}^N(\mathbf{V}^N)\|_{\ell^2_0(\T_N)}^2\right]} + \mathsf{D},
\end{align*}
thanks to the Cauchy--Schwarz inequality. We now write 
\begin{equation}\label{fluxbound}
  \begin{aligned}
    &\E\left[\|\mathbf{D}^{(1,-)}_N \overline{\mathbf{A}}^N(\mathbf{V}^N)\|_{\ell^2_0(\T_N)}^2\right]\\
    &=\E \left[ N \sum_{i \in \T_N} \left( \bar A(V^N_i,V^N_{i+1}) - \bar A(V^N_{i-1},V^N_i) \right)^2 \right] \\
    &= \E \left[ N \sum_{i \in \T_N} \left( \int_{V^N_i}^{V^N_{i+1}} \partial_2 \bar A(V^N_i,z) \dd z + \int_{V^N_{i-1}}^{V^N_i} \partial_1 \bar A(z,V^N_i) \dd z \right)^2 \right]  \\
    &\leq 2 \E \left[ N \sum_{i \in \T_N} (V^N_{i+1}-V^N_i) \int_{V^N_i}^{V^N_{i+1}} \partial_2 \bar A(V^N_i,z)^2 \dd z \right]+ 2 \E \left[ N \sum_{i \in \T_N}(V^N_i-V^N_{i-1}) \int_{V^N_{i-1}}^{V^N_i} \partial_1 \bar A(z,V^N_i)^2 \dd z \right]\\
    &\leq 4 \mathsf{C}_{\bar A}^2 \E \left[ N \sum_{i \in \T_N} (V^N_i-V^N_{i-1}) \int_{V^N_{i-1}}^{V^N_i} \left( 1+|z|^{\mathsf{p}_{\bar A}} \right)^2 \dd z \right] \\
    &\leq 8 \mathsf{C}_{\bar A}^2 \left( \E \left[ N \sum_{i \in \T_N} (V^N_i-V^N_{i-1})^2 \right] + \E \left[ N \sum_{i \in \T_N} (V^N_i-V^N_{i-1}) \int_{V^N_{i-1}}^{V^N_i} |z|^{2\mathsf{p}_{\bar A}} \dd z \right] \right)  \\
    &=8 \mathsf{C}_{\bar A}^2 \left( \E\left[\left\|\mathbf D^{(1,+)}_N\mathbf V^N\right\|_{\ell^2_0(\T_N)}^2\right]+\frac1{2\mathsf{p}_{\bar A}+1}\E\left[\left\langle\mathbf D^{(1,+)}_N\left((\mathbf V^N)^{2\mathsf{p}_{\bar A}+1}\right),\mathbf D^{(1,+)}_N\mathbf V^N\right\rangle_{\ell^2_0(\T_N)}\right]\right) \\
    &\leq8 \mathsf{C}_{\bar A}^2 \left(\mathsf{C}^{1,2}+\frac{\mathsf{D}}{2\nu}\mathsf{C}^{0,2\mathsf{p}_{\bar A}} \right), 
  \end{aligned}
\end{equation}
where we have used~\eqref{E1} at the third line, the Jensen inequality at the fourth line, \eqref{PG} at the fifth line and Step~3 and the intermediary estimate~\eqref{eq:statrec} with $p=2\mathsf{p}_{\bar A}+2$ at the last line.

We therefore get
\begin{equation*}
2\nu \E\left[\|\mathbf D^{(2)}_N\mathbf V^N\|^2_{\ell^2_0(\T_N)}\right] \leq 2\sqrt{\E \left[ \left\|\mathbf D^{(2)}_N\mathbf V^N\right\|_{\ell^2_0(\T_N)}^2 \right]}\sqrt{4\mathsf{C}_{\bar A}^2\frac{\mathsf{D}}\nu\left(1+\mathsf{C}^{0,2\mathsf{p}_{\bar A}}\right)}+\mathsf{D}.
\end{equation*}
Applying Young's inequality on the right-hand side, we get
\begin{equation*}
2\nu \E\left[\|\mathbf D^{(2)}_N\mathbf V^N\|^2_{\ell^2_0(\T_N)}\right] \leq\nu\E\left[\left\|\mathbf D^{(2)}_N\mathbf V^N\right\|_{\ell^2_0(\T_N)}^2\right]+4 \mathsf{C}_{\bar A}^2\frac{\mathsf{D}}{\nu^2}\left(1+\mathsf{C}^{0,2\mathsf{p}_{\bar A}}\right)+\mathsf{D},
\end{equation*}
which rewrites
\[\E\left[\left\|\mathbf D^{(2)}_N\mathbf V^N\right\|_{\ell^2_0(\T_N)}^2\right]\leq4 \mathsf{C}_{\bar A}^2\frac{\mathsf{D}}{\nu^3}\left(1+\mathsf{C}^{0,2\mathsf{p}_{\bar A}}\right)+\frac{\mathsf{D}}\nu =: \mathsf{C}^{2,2},\]
and yields the claimed estimate.
\end{proof}

In order to prepare the proof of Corollary~\ref{cor:relcom}, we first introduce the interpolation operators $\Psi^{(1)}_N : \R^N_0 \to H^1_0(\T)$ and $\Psi^{(2)}_N : \R^N_0 \to H^2_0(\T)$ defined by the condition that for any $\mathbf{v} \in \R^N_0$, the function $\Psi^{(1)}\mathbf{v}$ (resp. $\Psi^{(2)}\mathbf{v}$) is linear (resp. quadratic) on each cell of the mesh $\mathcal{T}_N$, and takes the value $v_i$ at the interface $x_i$. It follows from elementary computations that, for any $\mathbf v\in\R^N_0$,
\begin{equation}\label{eq:Psi12}
  \|\Psi^{(1)}_N\mathbf{v}\|^2_{L^2_0(\T)} = \|\mathbf{D}^{(1,+)}_N\mathbf{v}\|^2_{\ell^2_0(\T_N)}, \qquad \|\Psi^{(2)}_N\mathbf{v}\|^2_{L^2_0(\T)} = \|\mathbf{D}^{(2)}_N\mathbf{v}\|^2_{\ell^2_0(\T_N)},
\end{equation}
and
\begin{align}
  &\left\|\Psi^{(1)}_N\mathbf v-\Psi_N\mathbf v\right\|_{L^2_0(\T)}^2=\frac1{3N^2}\left\|\mathbf D^{(1,+)}_N\mathbf v\right\|_{\ell^2_0(\T_N)}^2,\label{eq:Psi10}\\
  &\left\|\Psi^{(2)}_N\mathbf v-\Psi_N\mathbf v\right\|_{L^2_0(\T)}^2\leq\frac3{20N^4}\left\|\mathbf D^{(2)}_N\mathbf v\right\|_{\ell^2_0(\T_N)}^2+\frac1{2N^2}\left\|\mathbf D^{(1,+)}_N\mathbf v\right\|_{\ell^2_0(\T_N)}^2,\label{eq:Psi20}
\end{align}
see~\cite[Lemma~3.30]{Mar19} for details.

\begin{proof}[Proof of Corollary~\ref{cor:relcom}]
  For $m=1,2$, we denote by $\mu^{(m)}_N = \vartheta_N \circ (\Psi^{(m)}_N)^{-1}$ the pushforward measure of $\vartheta_N$ by the interpolation operator $\Psi^{(m)}_N$.  For all $N \geq 1$, we shall also denote by $\mathbf{V}^N$ a random variable in $\R^N_0$ with distribution $\vartheta_N$.
  
  \emph{Step~1: tightness on $L^2_0(\T)$.} By~\eqref{eq:Psi12}, the uniform $h^1_0(\T_N)$ estimate of Proposition~\ref{prop:unifestim} rewrites under the form
  \begin{equation*}
    \sup_{N \geq 1}\E\left[\|\Psi^{(1)}_N\mathbf{V}^N\|^2_{H^1_0(\T)}\right] \leq \mathsf{C}^{1,2}.
  \end{equation*}
  Since bounded sets of $H^1_0(\T)$ are compact in $L^2_0(\T)$, this uniform moment estimate shows that the sequence $(\mu^{(1)}_N)_{N \geq 1}$ is tight on $L^2_0(\T)$. Therefore, by the Prokhorov theorem~\cite[Theorem~5.1]{Bil99}, it is relatively compact in $\mathcal{P}(L^2_0(\T))$. 
  
  It is then easy to deduce from~\eqref{eq:Psi10} and~\eqref{eq:Psi20}, combined with the uniform $h^1_0(\T_N)$ and $h^2_0(\T_N)$ estimates from Proposition~\ref{prop:unifestim}, that for any increasing sequence of integers $(N_j)_{j \geq 1}$ and probability measure $\mu^*$ on $L^2_0(\T)$, the following three statements are equivalent:
  \begin{itemize}
    \item $\mu_{N_j}$ converges weakly to $\mu^*$;
    \item $\mu^{(1)}_{N_j}$ converges weakly to $\mu^*$;
    \item $\mu^{(2)}_{N_j}$ converges weakly to $\mu^*$;
  \end{itemize}
  so that all three sequences $(\mu_N)_{N \geq 1}$, $(\mu^{(1)}_N)_{N \geq 1}$ and $(\mu^{(2)}_N)_{N \geq 1}$ are relatively compact on $L^2_0(\T)$, with the same converging subsequences.

  \emph{Step~2: moment estimates.} Let $(N_j)_{j \geq 1}$ and $\mu^*$ be such that $\mu_{N_j}$ converges weakly to $\mu^*$. For any $p \in [1,+\infty)$, the function $v \mapsto \|v\|^p_{L^p_0(\T)}$ is lower semicontinuous on $L^2_0(\T)$. As a consequence, by the Portmanteau theorem, if $v^*$ is a random variable in $L^2_0(\T)$ with distribution $\mu^*$, then
  \begin{equation*}
    \E\left[\|v^*\|^p_{L^p_0(\T)}\right] \leq \liminf_{j \to +\infty} \E\left[\|\Psi_N \mathbf{V}^N\|^p_{L^p_0(\T)}\right] \leq \mathsf{C}^{0,p},
  \end{equation*}
  where we have used the isometry property~\eqref{eq:isoPsiN} and the uniform $\ell^p_0(\T_N)$ estimate from Proposition~\ref{prop:unifestim}. By Step~1, the sequences $\mu^{(1)}_{N_j}$ and $\mu^{(2)}_{N_j}$ also converge weakly to $\mu^*$. Besides, we recall that for any $s \geq 0$, the Sobolev norm $\|\cdot\|_{H^s_0(\T)}$ writes
  \begin{equation*}
    \|v\|^2_{H^s_0(\T)} = \sum_{m \geq 1} (-\lambda_m)^s \langle v, e_m\rangle^2_{L^2_0(\T)}
  \end{equation*}
  for nonpositive numbers $\lambda_m$ and $L^2_0(\T)$ orthonormal Fourier modes $e_m$, see for instance~\cite[Section~2.1.1]{MR19}. As a consequence, it follows from the Fatou lemma that the function $v \mapsto \|v\|^2_{H^s_0(\T)}$ is lower semicontinuous on $L^2_0(\T)$. Therefore, using now the isometry property~\eqref{eq:Psi12} and the uniform $h^1_0(\T_N)$ and $h^2_0(\T_N)$ estimates from Proposition~\ref{prop:unifestim}, we get
  \begin{equation*}
    \E\left[\|v^*\|^2_{H^1_0(\T)}\right] \leq \mathsf{C}^{1,2}, \qquad \E\left[\|v^*\|^2_{H^2_0(\T)}\right] \leq \mathsf{C}^{2,2}.
  \end{equation*}
  
 \emph{Step~3: relative compactness in $\mathcal{P}_2(L^2_0(\T))$.} It remains to prove that the weak convergence of $\mu_{N_j}$ to $\mu^*$ can be strengthened to the Wasserstein topology. To this aim it is sufficient to check that the sequence $\|\mathbf{V}^N\|^2_{\ell^2_0(\T_N)}$ is uniformly integrable~\cite[Definition 6.8.(iii)]{Vil08}. For any $\epsilon > 0$, let us set $p = 2(1+\epsilon)$ and use~\eqref{eq:normorder} together with Proposition~\ref{prop:unifestim} to get, for any $N \geq 1$,
  \begin{equation*}
    \E\left[\left(\|\mathbf{V}^N\|^2_{\ell^2_0(\T_N)}\right)^{1+\epsilon}\right] \leq \E\left[\|\mathbf{V}^N\|^p_{\ell^p_0(\T_N)}\right] \leq \mathsf{C}^{0,p}.
  \end{equation*}
  By~\cite[Eq.~(3.18), p.~31]{Bil99}, this estimate ensures the required uniform integrability.
\end{proof}

\subsection{Proof of Proposition~\ref{CFT}}\label{ss:CFT} 

The first step in the proof of Proposition~\ref{CFT} is the decomposition of the error
\begin{equation}\label{eq:decomperror}
  \|u^N(t)-u^*(t)\|^2_{L^2_0(\T)} \leq 2\left(\|\Psi_N\Pi_N u^*(t)-u^*(t)\|^2_{L^2_0(\T)} + \|u^N(t)-\Psi_N\Pi_N u^*(t)\|^2_{L^2_0(\T)}\right).
\end{equation}
On the one hand, for all $t \geq 0$, an elementary computation yields the estimate
\begin{equation*}
  \| \Psi_N\Pi_N u^*(t)-u^*(t)\|^2_{L^2_0(\T)} \leq \frac{1}{N^2}\|u^*(t)\|^2_{H^1_0(\T)},
\end{equation*}
while on the other hand, by~\eqref{eq:isoPsiN} we have
\begin{equation*}
  \|u^N(t)-\Psi_N\Pi_N u^*(t)\|^2_{L^2_0(\T)} = \|\Psi_N\mathbf{U}^N(t)-\Psi_N\Pi_N u^*(t)\|^2_{L^2_0(\T)} = \|\mathbf{U}^N(t)-\Pi_N u^*(t)\|^2_{\ell^2_0(\T_N)}.
\end{equation*}

In order to estimate both terms, we rely on the following finite-time estimates on $u^*$.

\begin{lemma}[Finite-time $L^p_0(\T)$, $H^1_0(\T)$ and $H^2_0(\T)$ bounds on $u^*$]\label{lem:FTBus}
  Under the assumptions of Proposition~\ref{CFT}, for all $t > 0$ there exist constants $\mathsf{C}^{*,0,p}_t, \mathsf{C}^{*,1,2}_t, \mathsf{C}^{*,2,2}_t \in [0,+\infty)$, such that
  \begin{equation*}
    \sup_{s \in [0,t]} \E\left[\|u^*(s)\|^p_{L^p_0(\T)}\right] \leq \mathsf{C}^{*,0,p}_t, \qquad \sup_{s \in [0,t]} \E\left[\|u^*(s)\|^2_{H^1_0(\T)}\right] \leq \mathsf{C}^{*,1,2}_t, \qquad \E\left[\int_0^t \|u^*(s)\|^2_{H^2_0(\T)}\dd s\right] \leq \mathsf{C}^{*,2,2}_t.
  \end{equation*}
\end{lemma}

The proof of Lemma~\ref{lem:FTBus} is postponed to Appendix~\ref{app:estim}. From the $H^1_0(\T)$ bound, we already get the control
\begin{equation}\label{eq:PiNus}
  \E\left[\| \Psi_N \Pi_N u^*(t)-u^*(t)\|^2_{L^2_0(\T)}\right] \leq \frac{\mathsf{C}^{*,1,2}_t}{N^2},
\end{equation}
on the expectation of the first term in the right-hand side of~\eqref{eq:decomperror}. 

In order to estimate the second term, we use the fact that both functions $A$ and $\overline{A}$ are locally Lipschitz continuous, respectively on $\R$ and $\R^2$. In this purpose, we fix $M \geq 0$ and introduce the stopping time
\begin{equation}\label{eq:taunm}
  \tau^N_{(M)} := \inf\left\{t \geq 0: \|\mathbf{U}^N(t)\|_{\ell^\infty_0(\T_N)} \geq M \text{ or } \|u^*(t)\|_{L^\infty_0(\T)} \geq M\right\}.
\end{equation}
We also denote by $\mathsf{L}_{(M)}$ a Lipschitz constant of $A$ on $[-M,M]$ and of $\overline{A}$ on $[-M,M]^2$. We now write
\begin{equation*}
  \E\left[\|\mathbf{U}^N(t)-\Pi_N u^*(t)\|^2_{\ell^2_0(\T_N)}\right] = \E\left[\|\mathbf{U}^N(t)-\Pi_N u^*(t)\|^2_{\ell^2_0(\T_N)}\ind{t \leq \tau^N_{(M)}}\right] + \E\left[\|\mathbf{U}^N(t)-\Pi_N u^*(t)\|^2_{\ell^2_0(\T_N)}\ind{t > \tau^N_{(M)}}\right].
\end{equation*}

The terms of the right-hand side are respectively estimated in Lemmas~\ref{lem:ordre1} and~\ref{lem:tauM}.

\begin{lemma}[Finite-time convergence in the Lipschitz case]\label{lem:ordre1}
  Under the assumptions of Proposition~\ref{CFT}, for all $t>0$ and $\delta_1, \delta_2 > 0$ such that
  \begin{equation}\label{eq:conddelta}
    \delta_1\nu + \delta_2\mathsf{L}_{(M)} \leq 2 \nu,
  \end{equation}
  there exists a constant $\mathfrak{C}(t,\mathsf{L}_{(M)},\delta_1,\delta_2)$ such that for any $N \geq 1$,
  \begin{equation*}
    \E\left[\|\mathbf{U}^N(t)-\Pi_N u^*(t)\|^2_{\ell^2_0(\T_N)}\ind{t \leq \tau^N_{(M)}}\right] \leq \e^{\gamma_{(M)} t}\E\left[\|\mathbf{U}^N_0-\Pi_N u^*_0\|^2_{\ell^2_0(\T_N)}\right] + \frac{\mathfrak{C}(t,\mathsf{L}_{(M)},\delta_1,\delta_2)}{N^2},
  \end{equation*}
  with
  \begin{equation*}
    \gamma_{(M)} = -2\nu + \delta_1\nu + \left(\delta_2 + \frac{3}{\delta_2}\right)\mathsf{L}_{(M)}.
  \end{equation*}
\end{lemma}
\begin{proof}
  For all $t \geq 0$, we define the vector $\mathbf{e}^N(t) \in \R^N_0$ by $\mathbf{e}^N(t) = \mathbf{U}^N(t)-\Pi_Nu^*(t)$. We have
  \begin{equation}\label{eq:dedt}
    \frac{\mathrm{d}}{\mathrm{d}t}\mathbf{e}^N(t) = \mathbf{D}_N^{(1,-)}\mathbf{f}^N(t),
  \end{equation}
  where, for all $i \in \T_N$,
  \begin{equation*}
    f^N_i(t) = -\left(\overline{A}(U^N_i(t),U^N_{i+1}(t))-A(u^*(t,x_i))\right) + \nu\left((\mathbf{D}^{(1,+)}_N\mathbf{U}^N(t))_i - \partial_x u^*(t,x_i)\right).   
  \end{equation*}
  Using~\eqref{eq:sbp}, we deduce that
  \begin{equation*}
    \frac{\mathrm{d}}{\mathrm{d}t}\|\mathbf{e}^N(t)\|^2_{\ell^2_0(\T_N)} = 2\langle \mathbf{e}^N(t),\mathbf{D}_N^{(1,-)}\mathbf{f}^N(t)\rangle_{\ell^2_0(\T_N)} = -2\langle \mathbf{D}_N^{(1,+)}\mathbf{e}^N(t),\mathbf{f}^N(t)\rangle_{\ell^2(\T_N)},
  \end{equation*}
  and now proceed to control the right-hand side of this identity. In this purpose we introduce the zeroth- and first-order errors $\mathbf{r}^{N,(0)}(t)$ and $\mathbf{r}^{N,(1)}(t)$ defined by
  \begin{equation*}
    r^{N,(0)}_i(t) = u^*(t,x_i) - (\Pi_Nu^*(t))_i, \qquad r^{N,(1)}_i(t) = \partial_x u^*(t,x_i) - (\mathbf{D}^{(1,+)}_N\Pi_Nu^*(t))_i.
  \end{equation*}
  
  On the one hand,
  \begin{equation*}
    (\mathbf{D}^{(1,+)}_N\mathbf{U}^N(t))_i - \partial_x u^*(t,x_i) = (\mathbf{D}^{(1,+)}_N\mathbf{e}^N(t))_i - r^{N,(1)}_i(t),
  \end{equation*}
  on the other hand, for all $t \leq \tau^N_{(M)}$, we get
  \begin{align*}
    \left|\overline{A}(U^N_i(t),U^N_{i+1}(t)) - A(u^*(t,x_i))\right| &\leq \left|\overline{A}(U^N_i(t),U^N_{i+1}(t)) - \overline{A}(U^N_i(t),U^N_i(t))\right| + \left|A(U^N_i(t))- A(u^*(t,x_i))\right|\\
  &\leq \mathsf{L}_{(M)}\left(\left|U^N_{i+1}(t)-U^N_i(t)\right| + \left|U^N_i(t)-u^*(t,x_i)\right|\right)\\
  &\leq \mathsf{L}_{(M)}\left(\left|U^N_{i+1}(t)-U^N_i(t)\right| + |e^N_i(t)| + \left|r^{N,(0)}_i(t)\right|\right).
  \end{align*}
  
  By the Young and Jensen inequalities, we thus deduce that, for any $\delta_1, \delta_2 > 0$,
  \begin{align*}
    -2\langle \mathbf{D}_N^{(1,+)}\mathbf{e}^N(t),\mathbf{f}^N(t)\rangle_{\ell^2(\T_N)} &\leq -2\nu\|\mathbf{D}^{(1,+)}_N\mathbf{e}^N(t)\|^2_{\ell^2_0(\T_N)} + 2\nu\langle \mathbf{D}^{(1,+)}_N\mathbf{e}^N(t), \mathbf{r}^{N,(1)}(t)\rangle_{\ell^2(\T_N)}\\
    &\quad + \frac{2\mathsf{L}_{(M)}}{N}\sum_{i \in \T_N} N|e^N_{i+1}(t)-e^N_i(t)|\left(\left|U^N_{i+1}(t)-U^N_i(t)\right| + |e^N_i(t)| + \left|r^{N,(0)}_i(t)\right|\right)\\
    &\leq \left(-2\nu + \delta_1\nu + \delta_2\mathsf{L}_{(M)}\right)\|\mathbf{D}^{(1,+)}_N\mathbf{e}^N(t)\|^2_{\ell^2_0(\T_N)} + \frac{\nu}{\delta_1} \|\mathbf{r}^{N,(1)}(t)\|_{\ell^2(\T_N)}^2\\
    &\quad + \frac{3\mathsf{L}_{(M)}}{\delta_2}\left(\frac{\|\mathbf{D}^{(1,+)}_N\mathbf{U}^N(t)\|_{\ell^2_0(\T_N)}^2}{N^2} + \|\mathbf{e}^N(t)\|_{\ell^2_0(\T_N)}^2 + \|\mathbf{r}^{N,(0)}(t)\|^2_{\ell^2(\T_N)}\right),
  \end{align*}
  and if $\delta_1$ and $\delta_2$ are small enough for the inequality~\eqref{eq:conddelta} to hold, then the discrete Poincar\'e inequality~\eqref{eq:poinca-discr} finally yields
  \begin{equation*}
    \frac{\mathrm{d}}{\mathrm{d}t}\|\mathbf{e}^N(t)\|^2_{\ell^2_0(\T_N)} \leq \gamma_{(M)} \|\mathbf{e}^N(t)\|^2_{\ell^2_0(\T_N)} + \mathfrak{c}^N(t),
  \end{equation*}
  with
  \begin{equation*}
    \mathfrak{c}^N(t) := \frac{\nu}{\delta_1} \|\mathbf{r}^{N,(1)}(t)\|_{\ell^2(\T_N)}^2 + \frac{3\mathsf{L}_{(M)}}{\delta_2}\left(\frac{\|\mathbf{D}^{(1,+)}_N\mathbf{U}^N(t)\|_{\ell^2_0(\T_N)}^2}{N^2} + \|\mathbf{r}^{N,(0)}(t)\|^2_{\ell^2(\T_N)}\right).
  \end{equation*}
  
  We deduce from Gr\"onwall's lemma that, for all $t \leq \tau^N_{(M)}$, 
  \begin{equation*}
    \|\mathbf{e}^N(t)\|^2_{\ell^2_0(\T_N)} \leq \e^{\gamma_{(M)} t}\|\mathbf{e}^N(0)\|^2_{\ell^2_0(\T_N)} + \int_0^t \e^{\gamma_{(M)}(t-s)}\mathfrak{c}^N(s)\dd s,
  \end{equation*}
  and therefore, for all $t \geq 0$,
  \begin{equation*}
    \E\left[\|\mathbf{e}^N(t)\|^2_{\ell^2_0(\T_N)}\ind{t \leq \tau^N_{(M)}}\right] \leq \e^{\gamma_{(M)} t}\E\left[\|\mathbf{e}^N(0)\|^2_{\ell^2_0(\T_N)}\right] + \int_0^t \e^{\gamma_{(M)}(t-s)}\E\left[\mathfrak{c}^N(s)\right]\dd s.
  \end{equation*}
  For all $s \in [0,t]$, the expectation of $\mathfrak{c}^N(s)$ rewrites
  \begin{equation*}
    \E\left[\mathfrak{c}^N(s)\right] = \frac{\nu}{\delta_1} \E\left[\|\mathbf{r}^{N,(1)}(s)\|_{\ell^2(\T_N)}^2\right] + \frac{3\mathsf{L}_{(M)}}{\delta_2}\left(\frac{\E\left[\|\mathbf{D}^{(1,+)}_N\mathbf{U}^N(s)\|_{\ell^2_0(\T_N)}^2\right]}{N^2} + \E\left[\|\mathbf{r}^{N,(0)}(s)\|^2_{\ell^2(\T_N)}\right]\right).
  \end{equation*}
  On the one hand, the stationarity of $\mathbf{U}^N$ and Proposition~\ref{prop:unifestim} yield
  \begin{equation*}
    \E\left[\|\mathbf{D}^{(1,+)}_N\mathbf{U}^N(s)\|_{\ell^2_0(\T_N)}^2\right] \leq \mathsf{C}^{1,2}.
  \end{equation*}
  On the other hand, it easily follows from the Taylor formula that
  \begin{equation*}
    \E\left[\|\mathbf{r}^{N,(0)}(s)\|^2_{\ell^2(\T_N)}\right] \leq \frac{1}{N^2}\E\left[\|u^*(s)\|^2_{H^1_0(\T)}\right], \qquad \E\left[\|\mathbf{r}^{N,(1)}(s)\|^2_{\ell^2(\T_N)}\right] \leq \frac{4}{N^2}\E\left[\|u^*(s)\|^2_{H^2_0(\T)}\right],
  \end{equation*}
  so that we get
  \begin{equation*}
    \int_0^t \e^{\gamma_{(M)}(t-s)}\E\left[\mathfrak{c}^N(s)\right]\dd s \leq \frac{\mathfrak{C}(t,\mathsf{L}_{(M)},\delta_1,\delta_2)}{N^2},
  \end{equation*}
  with
  \begin{equation*}
    \mathfrak{C}(t,\mathsf{L}_{(M)},\delta_1,\delta_2) := \int_0^t \e^{\gamma_{(M)}(t-s)}\left(\frac{4\nu}{\delta_1}\E\left[\|u^*(s)\|^2_{H^2_0(\T)}\right] + \frac{3\mathsf{L}_{(M)}}{\delta_2}\left(\mathsf{C}^{1,2} + \E\left[\|u^*(s)\|^2_{H^1_0(\T)}\right]\right)\right)\dd s,
  \end{equation*}
  which is finite by Lemma~\ref{lem:FTBus}.
\end{proof}

\begin{lemma}[Uniform control over $\tau^N_{(M)}$]\label{lem:tauM}
  Under the assumptions of Proposition~\ref{CFT}, for all $t>0$ we have
  \begin{equation*}
    \lim_{M \to +\infty} \limsup_{N \to +\infty} \E\left[\|\mathbf{U}^N(t)-\Pi_N u^*(t)\|^2_{\ell^2_0(\T_N)}\ind{t > \tau^N_{(M)}}\right]=0,
  \end{equation*}
  where $\tau^N_{(M)}$ was defined at Equation~\eqref{eq:taunm}.
\end{lemma}

The proof of Lemma~\ref{lem:tauM} relies on the following result, the proof of which is postponed to Appendix~\ref{app:estim}.

\begin{lemma}[Finite-time uniform $h^1_0(\T_N)$ bound on $\mathbf{U}^N$]\label{lem:FTBUN}
  Under the assumptions of Proposition~\ref{CFT}, for all $t > 0$ there exists a constant $\mathsf{S}^{1,2}_t \in [0,+\infty)$, such that
  \begin{equation*}
    \sup_{N\geq1}\E\left[\sup_{s \in [0,t]} \|\mathbf{D}^{(1,+)}_N\mathbf{U}^N(s)\|^2_{\ell^2_0(\T_N)}\right] \leq \mathsf{S}^{1,2}_t.
  \end{equation*}
\end{lemma}

\begin{proof}[Proof of Lemma~\ref{lem:tauM}]
  By the Cauchy--Schwarz, triangle and Jensen inequalities, we first write
  \begin{align*}
    &\E\left[\|\mathbf{U}^N(t)-\Pi_N u^*(t)\|^2_{\ell^2_0(\T_N)}\ind{t > \tau^N_{(M)}}\right]\\
    & \leq \sqrt{\E\left[\|\mathbf{U}^N(t)-\Pi_N u^*(t)\|^4_{\ell^2_0(\T_N)}\right]}\sqrt{\P\left(t > \tau^N_{(M)}\right)}\\
    & \leq \sqrt{8\left(\E\left[\|\mathbf{U}^N(t)\|^4_{\ell^2_0(\T_N)}\right]+\E\left[\|\Pi_N u^*(t)\|^4_{\ell^2_0(\T_N)}\right]\right)}\sqrt{\P\left(\sup_{s \in [0,t]}\|\mathbf{U}^N(s)\|_{\ell^\infty_0(\T_N)} \geq M \text{ or } \sup_{s \in [0,t]}\|u^*(s)\|_{L^\infty_0(\T)} \geq M\right)}.
  \end{align*}
  On the one hand, we deduce from~\eqref{eq:normorder}, \eqref{eq:isoPiN}, Proposition~\ref{prop:unifestim} and Lemma~\ref{lem:FTBus} that
  \begin{equation*}
    \E\left[\|\mathbf{U}^N(t)\|^4_{\ell^2_0(\T_N)}\right] \leq \mathsf{C}^{0,4}, \qquad \E\left[\|\Pi_N u^*(t)\|^4_{\ell^2_0(\T_N)}\right] \leq \mathsf{C}^{*,0,4}_t.
  \end{equation*}
  On the other hand, we get from the union bound and the Markov inequality that
  \begin{align*}
    &\P\left(\sup_{s \in [0,t]}\|\mathbf{U}^N(s)\|_{\ell^\infty_0(\T_N)} \geq M \text{ or } \sup_{s \in [0,t]}\|u^*(s)\|_{L^\infty_0(\T)} \geq M\right)\\
    &\leq \P\left(\sup_{s \in [0,t]}\|\mathbf{U}^N(s)\|_{\ell^\infty_0(\T_N)} \geq M\right) + \P\left(\sup_{s \in [0,t]}\|u^*(s)\|_{L^\infty_0(\T)} \geq M\right)\\
    &\leq \frac{1}{M}\E\left[\sup_{s \in [0,t]}\|\mathbf{U}^N(s)\|_{\ell^\infty_0(\T_N)}\right] + \P\left(\sup_{s \in [0,t]}\|u^*(s)\|_{L^\infty_0(\T)} \geq M\right).
  \end{align*}
  Using~\eqref{eq:gradient-estim-discr}, \eqref{eq:normorder}, the Jensen inequality, and Lemma~\ref{lem:FTBUN}, we get
  \begin{equation*}
    \E\left[\sup_{s \in [0,t]}\|\mathbf{U}^N(s)\|_{\ell^\infty_0(\T_N)}\right] \leq \sqrt{\E\left[\sup_{s \in [0,t]}\|\mathbf{D}^{(1,+)}_N\mathbf{U}^N(s)\|^2_{\ell^2_0(\T_N)}\right]} \leq \sqrt{\mathsf{S}^{1,2}_t},
  \end{equation*} 
  while by~\eqref{eq:holder} and~\eqref{eq:strongerpoincare}, 
  \begin{equation*}
    \P\left(\sup_{s \in [0,t]}\|u^*(s)\|_{L^\infty_0(\T)} \geq M\right) \leq \P\left(\sup_{s \in [0,t]}\|u^*(s)\|_{H^2_0(\T)} \geq M\right),
  \end{equation*}
  and since $u^*$ is a continuous $H^2_0(\T)$-valued process, the right-hand side (which does not depend on $N$) converges to $0$ when $M \to +\infty$. This completes the proof.
\end{proof}

We are now ready to complete the proof of Proposition~\ref{CFT}.

\begin{proof}[Proof of Proposition~\ref{CFT}]
  Combining~\eqref{eq:decomperror} with~\eqref{eq:PiNus} and the results of Lemmas~\ref{lem:ordre1} and~\ref{lem:tauM}, we see that in order to complete the proof of Proposition~\ref{CFT}, it only remains to check that
  \begin{equation*}
    \lim_{N \to +\infty} \E\left[\|\mathbf{U}^N_0-\Pi_N u^*_0\|^2_{\ell^2_0(\T_N)}\right] = 0.
  \end{equation*}
  Using the obvious identity $\mathbf{v} = \Pi_N\Psi_N\mathbf{v}$ and~\eqref{eq:isoPiN}, we first write
  \begin{equation*}
    \|\mathbf{U}^N_0-\Pi_N u^*_0\|^2_{\ell^2_0(\T_N)} = \|\Pi_N\Psi_N\mathbf{U}^N_0-\Pi_N u^*_0\|^2_{\ell^2_0(\T_N)} \leq \|\Psi_N\mathbf{U}^N_0-u^*_0\|^2_{L^2_0(\T)}.
  \end{equation*}
  By the construction of $\mathbf{U}^N_0$ and $u^*_0$, we know that $\|\Psi_N\mathbf{U}^N_0-u^*_0\|^2_{L^2_0(\T)}$ converges to $0$ almost surely. It remains to check that the expectation of this random variable also converges to $0$, which easily follows from the moment estimates of Proposition~\ref{prop:unifestim} and Corollary~\ref{cor:relcom} by uniform integrability.
\end{proof}

\begin{remark}
  The finite-time convergence result of Proposition~\ref{CFT} shows in particular that the semi-discrete scheme $(u^N(t))_{t \geq 0}$ converges, in a certain sense, to a solution $(u^*(t))_{t \geq 0}$ to the SPDE~\eqref{SSCL}. The proof of this proposition could actually be adapted to make this statement more general, without the assumption that the process $(u^N(t))_{t \geq 0}$ be stationary for example. We will not elaborate in this direction as the purpose of the present work is to focus on invariant \emph{measures} rather than on finite-time \emph{trajectorial} approximations.
\end{remark}

\subsection{Discussion of the rate of convergence}\label{ss:rate}

The estimates~\eqref{eq:decomperror} and~\eqref{eq:PiNus}, together with the result of Lemma~\ref{lem:ordre1}, display error terms of order $1/N^2$ for the estimation of $\|u^N(t)-u^*(t)\|^2_{\ell^2_0(\T)}$, but the localisation procedure by the stopping time $\tau^N_{(M)}$, which is used to handle the fact that $A$ and $\overline{A}$ are not globally Lipschitz continuous, generally prevents this computation from providing a global rate of convergence.

However, under the assumption that both $A$ and $\overline{A}$ are globally $\mathsf{L}$-Lipschitz continuous, for some $\mathsf{L} \in [0,+\infty)$, a strong error estimate can be derived. Indeed, in this case, let us define on a same probability space:
\begin{itemize}
  \item $(\mathbf{U}^N_0)_{N \geq 1}$ a sequence of $\mathcal{F}_0$-measurable random vectors such that $\mathbf{U}^N_0 \sim \vartheta_N$;
  \item $u_0$ a $\mathcal{F}_0$-measurable random variable in $H^2_0(\T)$ with distribution $\mu$;
  \item $W^Q$ a $Q$-Wiener process;
\end{itemize}
and denote by $(\mathbf{U}^N(t))_{t \geq 0}$ and $(u(t))_{t \geq 0}$ the respective solutions to~\eqref{FVS} and~\eqref{SSCL}, with respective initial conditions $\mathbf{U}^N_0$ and $u_0$, and respectively driven by $W^Q$ and $\mathbf{W}^{Q,N} := \Pi_NW^Q$. Notice that, in contrast with the setting of the previous subsections, the sequence $\Psi_N\mathbf{U}^N_0$ is not assumed to converge almost surely to $u_0$; on the other hand, we now know that both processes $(\mathbf{U}^N(t))_{t \geq 0}$ and $(u(t))_{t \geq 0}$ are stationary.

As a consequence, \eqref{eq:decomperror}, \eqref{eq:PiNus} and Corollary~\ref{cor:relcom} already yield
\begin{equation*}
  \E\left[\|u^N(t)-u^*(t)\|^2_{\ell^2_0(\T_N)}\right] \leq 2\left(\frac{\mathsf{C}^{1,2}}{N^2} + \E\left[\|\mathbf{U}^N(t)-\Pi_N u^*(t)\|^2_{\ell^2_0(\T_N)}\right]\right).
\end{equation*}
One may then apply Lemma~\ref{lem:ordre1} with $\mathsf{L}_{(M)} = \mathsf{L}$ to obtain
\begin{equation*}
  \E\left[\|\mathbf{U}^N(t)-\Pi_N u^*(t)\|^2_{\ell^2_0(\T_N)}\ind{t \leq \tau^N_{(M)}}\right] \leq \e^{\gamma t}\E\left[\|\mathbf{U}^N_0-\Pi_N u^*_0\|^2_{\ell^2_0(\T_N)}\right] + \frac{\mathfrak{C}(t,\mathsf{L},\delta_1,\delta_2)}{N^2},
\end{equation*}
with $\gamma=-2\nu + \delta_1\nu + (\delta_2 + \frac{3}{\delta_2})\mathsf{L}$. Since the right-hand side does not depend on $M$, we may take the $M \to +\infty$ limit first and thus obtain
\begin{equation*}
  \E\left[\|\mathbf{U}^N(t)-\Pi_N u^*(t)\|^2_{\ell^2_0(\T_N)}\right] \leq \e^{\gamma t}\E\left[\|\mathbf{U}^N_0-\Pi_N u^*_0\|^2_{\ell^2_0(\T_N)}\right] + \frac{\mathfrak{C}(t,\mathsf{L},\delta_1,\delta_2)}{N^2}.
\end{equation*}
Besides, by stationarity of $(u(t))_{t \geq 0}$ and Corollary~\ref{cor:relcom}, the constant $\mathfrak{C}(t,\mathsf{L},\delta_1,\delta_2)$ from Lemma~\ref{lem:ordre1} is bounded from above by
\begin{equation*}
  \mathfrak{C}'(t,\mathsf{L},\delta_1,\delta_2) = \int_0^t \e^{\gamma(t-s)}\left(\frac{4\nu}{\delta_1}\mathsf{C}^{2,2} + \frac{6\mathsf{L}}{\delta_2}\mathsf{C}^{1,2}\right)\dd s = \left(\frac{4\nu}{\delta_1}\mathsf{C}^{2,2} + \frac{6\mathsf{L}}{\delta_2}\mathsf{C}^{1,2}\right)\int_0^t \e^{\gamma s}\dd s.
\end{equation*}
We therefore deduce the finite-time estimate
\begin{equation}\label{eq:rate1}
  \E\left[\|u^N(t)-u^*(t)\|^2_{L^2_0(\T_N)}\right] \leq 2\left(\e^{\gamma t}\E\left[\|\mathbf{U}^N_0-\Pi_N u^*_0\|^2_{\ell^2_0(\T_N)}\right] + \frac{\mathsf{C}^{1,2}+\mathfrak{C}'(t,\mathsf{L},\delta_1,\delta_2)}{N^2}\right),
\end{equation}
with a squared error term of order $1/N^2$. It is remarkable that this is the same order as for \emph{deterministic} conservation laws (see for instance~\cite[Theorem~17.1]{EGH00}), so that, in our setting, the noise is sufficiently smooth in space not to deteriorate the strong error. This is in contrast with classical results for SPDEs with space-time white noise~\cite{Gyo98}.

When the constant $\mathsf{L}$ is small enough, then it is possible to choose $\delta_1$ and $\delta_2$ which satisfy~\eqref{eq:conddelta} and such that $\gamma < 0$. In this `perturbative' case (which presents similar features to the one briefly discussed in the introduction of~\cite{BreDeb18}), one may take the $t \to +\infty$ limit in~\eqref{eq:rate1} and obtain the long time estimate
\begin{equation*}
  \limsup_{t \to +\infty} \E\left[\|u^N(t)-u^*(t)\|^2_{L^2_0(\T_N)}\right] \leq \frac{\mathfrak{C}''}{N^2}, \qquad \mathfrak{C}'' = 2\left(\mathsf{C}^{1,2} + \frac{1}{-\gamma}\left(\frac{4\nu}{\delta_1}\mathsf{C}^{2,2} + \frac{6\mathsf{L}}{\delta_2}\mathsf{C}^{1,2}\right)\right).
\end{equation*}
Since for any $t \geq 0$, $u^N(t) \sim \mu_N$ and $u(t) \sim \mu$, we have
\begin{equation*}
  W_2(\mu_N,\mu) \leq \sqrt{\E\left[\|u^N(t)-u^*(t)\|^2_{L^2_0(\T_N)}\right]},
\end{equation*}
therefore the estimate above yields the quantitative bound
\begin{equation*}
  W_2(\mu_N,\mu) \leq \frac{\sqrt{\mathfrak{C}''}}{N}.
\end{equation*}
We shall check in Section~\ref{s:newnum} that the order $1/N$ for the Wasserstein distance is sharp in the case where $A=0$.
  
\section{Convergence of invariant measures: split-step scheme towards semi-discrete scheme}\label{section:convergence2}

This section is dedicated to the proof of the second statement in Theorem~\ref{TCV}, namely the convergence of $\vartheta_{N, \Delta t}$ to $\vartheta_N$ when $\Delta t \to 0$. We follow the same outline as in Section~\ref{section:convergence}, and we only emphasise the differences in the arguments. The main results of tightness and finite-time convergence are stated in Subsection~\ref{ss:sketchofproofSS} and proved in Subsection~\ref{ss:pftightSS} and~\ref{ss:CFTSS}, respectively. Rates of convergence are then discussed in Subsection~\ref{ss:rateSS}.

Throughout this section, we arbitrarily fix a maximal time step $\Delta t_{\mathrm{max}}>0$ and we always take $\Delta t \in (0,\Delta t_{\mathrm{max}}]$.

\subsection{Statement of the main arguments}\label{ss:sketchofproofSS}

The next two results play the respective roles of Proposition~\ref{prop:unifestim} and Corollary~\ref{cor:relcom}.

\begin{prop}[Uniform $\ell^4_0(\T_N)$ and $h^1_0(\T_N)$ estimates on $\vartheta_{N, \Delta t}$]\label{prop:unifestimSS}
  For all $N\geq1$ and all $\Delta t\in(0,\Delta t_{\mathrm{max}}]$, let $\mathbf{V}^{N,\Delta t}$ be a random variable in $\R^N_0$ with distribution $\vartheta_{N,\Delta t}$, and let $\mathbf{V}_{\frac12}^{N,\Delta t}$ be defined by
  \begin{equation*}
    \mathbf{V}_{\frac12}^{N,\Delta t}=\mathbf{V}^{N,\Delta t}+\Delta t\mathbf b\left(\mathbf{V}_{\frac12}^{N,\Delta t}\right).
  \end{equation*}
  There exist constants $\mathsf{C}^{\Delta, 0, 4}, \mathsf{C}^{\Delta, 1, 2}, \mathsf{C}_{\frac12}^{\Delta, 1, 2} \in [0,+\infty)$, which only depend on $\nu$, $\mathsf{D}$ and $\Delta t_\mathrm{max}$, such that 
  \begin{equation*}
    \sup_{N\geq1}\sup_{\Delta t\in(0,\Delta t_{\mathrm{max}}]}\E\left[\|\mathbf{V}^{N,\Delta t}\|^4_{\ell^4_0(\T_N)}\right] \leq \mathsf{C}^{\Delta, 0, 4},
    \end{equation*}
    \begin{equation*}
     \sup_{N\geq1}\sup_{\Delta t\in(0,\Delta t_{\mathrm{max}}]}\E\left[\|\mathbf{D}^{(1,+)}_N\mathbf{V}^{N,\Delta t}\|^2_{\ell^2_0(\T_N)}\right] \leq \mathsf{C}^{\Delta, 1, 2}, \qquad \sup_{N\geq1}\sup_{\Delta t\in(0,\Delta t_{\mathrm{max}}]} \E\left[\|\mathbf{D}^{(1,+)}_N\mathbf{V}_{\frac12}^{N,\Delta t}\|^2_{\ell^2_0(\T_N)}\right] \leq \mathsf{C}_{\frac12}^{\Delta, 1, 2}.
  \end{equation*}
\end{prop}

\begin{remark}
Observe that unlike in Proposition~\ref{prop:unifestim}, no $h^2_0(\T_N)$ estimate is established here, and as a consequence, $\ell^p_0(\T_N)$ estimates for $p>4$, which were used in Proposition~\ref{prop:unifestim} to control the polynomial growth of the derivatives of the numerical flux, are not used. Indeed, the $h^2_0(\T_N)$ estimate was necessary to interpret any limit of the space-discretised approximations of the invariant measure $\mu$ as a starting distribution for an $H^2_0(\T)$-valued Markov process. Here, the aim is to approximate the invariant measure of an $\R^N_0$-valued Markov process, so that any probability measure on $\R^N_0$ can be regarded as an initial distribution for this process. The content of Proposition~\ref{prop:unifestimSS} will be sufficient to control the terms that will arise from the decomposition of the approximation error.
\end{remark}

\begin{corollary}[Relative compactness and $\ell^4_0(\T_N)$ and $h^1_0(\T_N)$ estimates on $\nu^*_N$]\label{cor:relcomSS}
  The family of probability measures $\{\vartheta_{N, \Delta t}, \Delta t\in(0,\Delta t_{\mathrm{max}}] \}$ is relatively compact in $\mathcal{P}_2(\R^N_0)$, and any subsequential limit $\vartheta_N^*$ (when $\Delta t \to 0$) has the property that if $\mathbf{V}^{N,*}$ is a random variable in $\R^N_0$ with distribution $\vartheta_N^*$, then 
  \begin{equation*}
    \E\left[\|\mathbf{V}^{N,*}\|^4_{\ell^4_0(\T_N)}\right] \leq \mathsf{C}^{\Delta, 0, 4}, \qquad \E\left[\|\mathbf{D}^{(1,+)}_N\mathbf{V}^{N,*}\|^2_{\ell^2_0(\T_N)}\right] \leq \mathsf{C}^{\Delta, 1, 2}.
  \end{equation*}
\end{corollary}

The proof of Proposition~\ref{prop:unifestimSS} is detailed in Subsection~\ref{ss:pftightSS}. Corollary~\ref{cor:relcomSS} is straightforward.

We now let $(\Delta t_j)_{j \geq 1} \subset (0, \Delta t_\mathrm{max}]$ be a sequence of time steps decreasing to $0$ and such that $\vartheta_{N, \Delta t_j}$ converges to some probability measure $\vartheta_N^*$ in $\mathcal{P}_2(\R^N_0)$. We use the Skorohod representation theorem to construct, on the same probability space, a sequence of $\mathcal{F}_0$-measurable random vectors $(\mathbf{U}^{N,\Delta t_j}_0)_{j \geq 1}$ and a random vector $\mathbf{U}^{N,*}_0$ such that:
\begin{itemize}
  \item for all $j \geq 1$, $\mathbf{U}^{N,\Delta t_j}_0 \sim \vartheta_{N,\Delta t_j}$, 
  \item $\mathbf{U}^{N,*}_0 \sim \vartheta_N^*$,
  \item $\mathbf{U}^{N,\Delta t_j}_0$ converges almost surely to $\mathbf{U}^{N,*}_0$.
\end{itemize}
Up to enlarging this probability space with the same product procedure as in Section~\ref{ss:sketchofproof}, we let $\mathbf{W}^{Q,N}$ be a Wiener process with covariance given by~\eqref{eq:covWQN}, we define the sequence $(\Delta \mathbf{W}^{Q,N}_n)_{n \in \N^*}$ by~\eqref{eq:DeltaWQN}, and we denote by $(\mathbf{U}^{N, \Delta t_j}_n)_{n \in \N}$ and $(\mathbf{U}^{N,*}(t))_{t \geq 0}$ the respective solutions to~\eqref{SSBE} and~\eqref{FVS} with initial conditions $\mathbf{U}^{N,\Delta t_j}_0$ and $\mathbf{U}^{N,*}_0$, and noises $(\Delta \mathbf{W}^{Q,N}_n)_{n \in \N^*}$ and $\mathbf{W}^{Q,N}$. We finally define the continuous-time piecewise constant interpolation of the split-step scheme $(\overline{\mathbf{U}}^{N, \Delta t_j}(t))_{t \geq 0}$ by
\begin{equation*}
  \forall n \in \N, \quad \forall t \in [n\Delta t_j, (n+1)\Delta t_j), \qquad \overline{\mathbf{U}}^{N, \Delta t_j}(t) = \mathbf{U}^{N, \Delta t_j}_n.
\end{equation*}
Notice that for any $t \geq 0$, $\overline{\mathbf{U}}^{N, \Delta t_j}(t) \sim \vartheta_{N,\Delta t_j}$.

The finite-time convergence statement reads as follows. It is proved in Subsection~\ref{ss:CFTSS}.

\begin{prop}[Finite-time convergence of $\overline{\mathbf{U}}^{N, \Delta t_j}(t)$]\label{CFTSS}
  In the setting described above, for all $t \geq 0$,
  \begin{equation*}
    \lim_{j \to +\infty} \E\left[\|\overline{\mathbf{U}}^{N, \Delta t_j}(t)-\mathbf{U}^{N,*}(t)\|^2_{\ell^2_0(\T_N)}\right] = 0.
  \end{equation*}
\end{prop}

We deduce that the measure $\vartheta_N^*$ is invariant for the SDE~\eqref{FVS}, which by Theorem~\ref{IM} completes the proof of~\eqref{cv2} in Theorem~\ref{TCV}.

\subsection{Proof of Proposition~\ref{prop:unifestimSS}}\label{ss:pftightSS}

  We first prove the $h^1_0(\T_N)$ estimates. Set $\mathbf{u}_0 = \mathbf{V}^{N, \Delta t}$ in the proof of Proposition~\ref{existence}. Since Theorem~\ref{IM} asserts that $\vartheta_{N, \Delta t} \in \mathcal{P}_2(\R^N_0)$, one may take the expectation and the $n \to +\infty$ limit in~\eqref{almostKB} to get
  \begin{equation*}
    \E\left[\|\mathbf{D}^{(1,+)}_N\mathbf{V}_{\frac12}^{N,\Delta t}\|^2_{\ell^2_0(\T_N)}\right] \leq \frac{\mathsf{D}}{2\nu} =: \mathsf{C}_{\frac12}^{\Delta, 1, 2}.
  \end{equation*}
  The same operations in~\eqref{eq:kb} yield
  \begin{equation*}
    \E\left[\|\mathbf{D}^{(1,+)}_N\mathbf{V}^{N,\Delta t}\|^2_{\ell^2_0(\T_N)}\right] \leq \mathsf{D}\left(\frac{1}{2\nu}+\Delta t_\mathrm{max}\right) =: \mathsf{C}^{\Delta, 1, 2}.
  \end{equation*}
  
  We now focus on the $\ell^4_0(\T)$ estimate. Let $(\mathbf U^{N,\Delta t}_n)_{n\in\N}=(U^{N,\Delta t}_{1,n},\dots,U^{N,\Delta t}_{N,n})_{n\in\N}$ be a solution of~\eqref{SSBE} with a deterministic initial condition $\mathbf u_0$. By convexity of the function $v\mapsto v^4$, for any $\alpha,\beta\in\R$, we have $(\alpha-\beta)^4\geq\alpha^4-4\alpha^3\beta$. In particular, for any $i \in \T_N$, taking $\alpha=U^{N,\Delta t}_{i,n+\frac12}$ and $\beta=\Delta tb_i(\mathbf U^{N,\Delta t}_{n+\frac12})$, we have
\[ (U^{N,\Delta t}_{i,n})^4 = \left( U^{N,\Delta t}_{i,n+\frac 12} - \Delta t b_i \left(\mathbf U^{N,\Delta t}_{n+\frac 1 2} \right) \right)^4 \geq (U^{N,\Delta t}_{i,n+\frac 12})^4 - 4 (U^{N,\Delta t}_{i,n+\frac 12})^3 \Delta t b_i \left(\mathbf U^{N,\Delta t}_{n+\frac 1 2} \right) . \]
Hence, expanding the function $\mathbf{b}$ and summing over $i$, we get
\begin{equation*}
\left\|\mathbf U^{N,\Delta t}_n \right\|_{\ell^4_0(\T_N)}^4 \geq \left\|\mathbf U^{N,\Delta t}_{n+\frac 12} \right\|_{\ell^4_0(\T_N)}^4 + 4\Delta t \left\langle (\mathbf{U}^{N,\Delta t}_{i,n+\frac 12})^3, \mathbf{D}^{(1,-)}_N \overline{\mathbf{A}}^N(\mathbf U^{N,\Delta t}_{n+\frac 1 2})\right\rangle_{\ell^2(\T_N)} - 4 \nu \Delta t \left\langle (\mathbf{U}^{N,\Delta t}_{i,n+\frac 12})^3, \mathbf{D}^{(2)}_N \mathbf U^{N,\Delta t}_{n+\frac 1 2}\right\rangle_{\ell^2(\T_N)}.
\end{equation*}
We know thanks to Lemma~\ref{L1} that the second term of the right-hand side is non-negative. Using~\eqref{eq:sbp2} in the third term, we get
\[ \left\|\mathbf U^{N,\Delta t}_n \right\|_{\ell^4_0(\T_N)}^4 \geq \left\|\mathbf U^{N,\Delta t}_{n+\frac 12} \right\|_{\ell^4_0(\T_N)}^4 + 4 \nu \Delta t \left\langle\mathbf D^{(1,+)}_N(\mathbf U^{N,\Delta t})^3_{n+\frac12},\mathbf D^{(1,+)}_N\mathbf U^{N,\Delta t}_{n+\frac12}\right\rangle_{\ell^2_0(\T_N)}.\]
From Lemma~\ref{comp}, we get
\begin{equation}\label{eq:firstside}
\left\|\mathbf U^{N,\Delta t}_n \right\|_{\ell^4_0(\T_N)}^4 \geq \left\|\mathbf U^{N,\Delta t}_{n+\frac 12} \right\|_{\ell^4_0(\T_N)}^4 + 3 \nu\Delta t \left\|\mathbf U^{N,\Delta t}_{n+\frac 12} \right\|_{\ell^4_0(\T_N)}^4.
\end{equation}

On the other hand, let us look at the second step of the scheme~\eqref{SSBE}. By the construction of the split-step scheme, the random variables $U^{N,\Delta t}_{i,n+\frac12}$ and $\Delta W^{Q,N}_{i,n+1}$ are independent. Since $\Delta W^{Q,N}_{i,n+1}\sim\mathcal N(0,\Delta t\sum_{k\geq1}(g^k_i)^2)$ and $\sum_{k\geq1}(g^k_i)^2\leq \mathsf{D}$ by~\eqref{eq:sigma-glob}, we write
\begin{align}
\E \left[ \left\|\mathbf U^{N,\Delta t}_{n+1} \right\|_{\ell^4_0(\T_N)}^4 \right] &= \E \left[ \left\|\mathbf U^{N,\Delta t}_{n+\frac 12} + \Delta\mathbf W^{Q,N}_{n+1} \right\|_{\ell^4_0(\T_N)}^4 \right] \nonumber \\
&= \E \left[ \left\|\mathbf U^{N,\Delta t}_{n+\frac 12} \right\|_{\ell^4_0(\T_N)}^4 \right] + \frac6N \E \left[ \sum_{i \in \T_N} (U^{N,\Delta t}_{i,n+\frac 12})^2 \left( \Delta W^{Q,N}_{i,n+1} \right)^2 \right] + \E \left[ \left\| \Delta\mathbf W^{Q,N}_{n+1} \right\|_{\ell^4_0(\T_N)}^4 \right] \nonumber \\
&\leq \E \left[ \left\|\mathbf U^{N,\Delta t}_{n+\frac 12} \right\|_{\ell^4_0(\T_N)}^4 \right] + 6 \mathsf{D} \Delta t \E \left[ \left\|\mathbf U^{N,\Delta t}_{n+\frac 12} \right\|_{\ell^2_0(\T_N)}^2 \right] + 3 \mathsf{D}^2 \Delta t^2. \label{otherside}
\end{align}
Combining Inequalities~\eqref{eq:firstside} and~\eqref{otherside}, we get
\[ \E\left[\left\|\mathbf U^{N,\Delta t}_n\right\|_{\ell^4_0(\T_N)}^4\right] \geq \E\left[\left\|\mathbf U^{N,\Delta t}_{n+1}\right\|_{\ell^4_0(\T_N)}^4\right]-6\mathsf{D}\Delta t\E\left[\left\|\mathbf U^{N,\Delta t}_{n+\frac12}\right\|_{\ell^2_0(\T_N)}^2\right]-3\mathsf{D}^2\Delta t^2 + 3\nu\Delta t\E\left[\left\|\mathbf U^{N,\Delta t}_{n+\frac12}\right\|_{\ell^4_0(\T_N)}^4\right] \]
from which we get the telescoping sum
\[ 3\nu\Delta t\sum_{l=0}^{n-1}\E\left[\left\|\mathbf U^{N,\Delta t}_{l+\frac12}\right\|_{\ell^4_0(\T_N)}^4\right]\leq\|\mathbf u_0\|_{\ell^4_0(\T_N)}^4-\E\left[\left\|\mathbf U^{N,\Delta t}_n\right\|_{\ell^4_0(\T_N)}^4\right]+6\mathsf{D}\Delta t\sum_{l=0}^{n-1}\E\left[\left\|\mathbf U^{N,\Delta t}_{l+\frac12}\right\|_{\ell^2_0(\T_N)}^2\right]+3n\mathsf{D}^2\Delta t^2 .\]
Thus,
\begin{equation}\label{eq:thus}
\frac1n\sum_{l=0}^{n-1}\E\left[\left\|\mathbf U^{N,\Delta t}_{l+\frac12}\right\|_{\ell^4_0(\T_N)}^4\right]\leq\frac1{3\nu\Delta tn}\|\mathbf u_0\|_{\ell^4_0(\T_N)}^4+\frac{2\mathsf{D}}{\nu n}\sum_{l=0}^{n-1}\E\left[\left\|\mathbf U^{N,\Delta t}_{l+\frac12}\right\|_{\ell^2_0(\T_N)}^2\right]+\frac{\mathsf{D}^2\Delta t}\nu.
\end{equation}
Recall that from~\eqref{eq:poinca-discr} and Equation~\eqref{almostKB}, we have
\begin{equation}\label{eq:ergodichalf}
\frac1n\sum_{l=0}^{n-1}\E\left[\left\|\mathbf U^{N,\Delta t}_{l+\frac12}\right\|_{\ell^2_0(\T_N)}^2\right]\leq\frac1n\sum_{l=0}^{n-1}\E\left[\left\|\mathbf D^{(1)}_N\mathbf U^{N,\Delta t}_{l+\frac12}\right\|_{\ell^2_0(\T_N)}^2\right]\leq\frac{\|\mathbf u_0\|_{\ell^2_0(\T_N)}^2}{2\nu n\Delta t}+\frac{\mathsf{D}}{2\nu} .
\end{equation}
Injecting~\eqref{eq:ergodichalf} into~\eqref{eq:thus}, we get
\[ \frac1n\sum_{l=0}^{n-1}\E\left[\left\|\mathbf U^{N,\Delta t}_{l+\frac12}\right\|_{\ell^4_0(\T_N)}^4\right]\leq\frac1{3\nu\Delta tn}\|\mathbf u_0\|_{\ell^4_0(\T_N)}^4+\frac{2\mathsf{D}}\nu\left(\frac{\|\mathbf u_0\|_{\ell^2_0(\T_N)}^2}{2\nu n\Delta t}+\frac{\mathsf{D}}{2\nu}\right)+\frac{\mathsf{D}^2\Delta t}\nu .\]

Using now the same arguments as for derivation of the $\ell^p_0(\T_N)$ estimates in Proposition~\ref{prop:unifestim}, we get
\[ \E\left[\left\| \mathbf V^{N,\Delta t}_{\frac12}\right\|_{\ell^4_0(\T_N)}^4\right]\leq\frac{\mathsf{D}^2}{\nu^2}+\frac{\mathsf{D}^2\Delta t}\nu=\frac{\mathsf{D}^2}\nu \left(\frac1\nu+\Delta t\right) .\]
To conclude, we use Inequality~\eqref{otherside} once again:
\begin{align*}
\E\left[\left\| \mathbf V^{N,\Delta t}\right\|_{\ell^4_0(\T_N)}^4\right]&\leq\E\left[\left\| \mathbf V^{N,\Delta t}_{\frac12}\right\|_{\ell^4_0(\T_N)}^4\right]+6\mathsf{D}\Delta t\E\left[\left\| \mathbf V^{N,\Delta t}_{\frac12}\right\|_{\ell^2_0(\T_N)}^2\right]+3\mathsf{D}^2\Delta t^2\\
&\leq \frac{\mathsf{D}^2}\nu \left(\frac1\nu+\Delta t\right)+\frac{3\mathsf{D}^2\Delta t}\nu+3\mathsf{D}^2\Delta t^2\\
&\leq \mathsf{D}^2 \left( \frac1\nu+3\Delta t_{\mathrm{max}}\right)\left( \frac1\nu+\Delta t_{\mathrm{max}}\right) =: \mathsf{C}^{\Delta, 0, 4}. 
\end{align*}

\subsection{Proof of Proposition~\ref{CFTSS}}\label{ss:CFTSS}

Similarly to the semi-discrete scheme, we use a localisation argument in order to use the local Lipschitz continuity of the function $\mathbf{b}$. For any $M \geq 0$ and $j \geq 1$, we therefore introduce the stopping time
\begin{equation*}
  \rho^j_{(M)} = \inf\left\{t \geq 0: \|\overline{\mathbf{U}}^{N,\Delta t_j}(t)\|_{\ell^2_0(\T_N)} \geq M \text{ or } \|\mathbf{U}^{N,*}(t)\|_{\ell^2_0(\T_N)} \geq M\right\},
\end{equation*}
and write, for all $t \geq 0$,
\begin{align*}
  \E\left[\|\overline{\mathbf{U}}^{N, \Delta t_j}(t)-\mathbf{U}^{N,*}(t)\|^2_{\ell^2_0(\T_N)}\right] &= \E\left[\|\overline{\mathbf{U}}^{N, \Delta t_j}(t)-\mathbf{U}^{N,*}(t)\|^2_{\ell^2_0(\T_N)}\ind{t \leq \rho^j_{(M)}}\right]\\
  &\quad + \E\left[\|\overline{\mathbf{U}}^{N, \Delta t_j}(t)-\mathbf{U}^{N,*}(t)\|^2_{\ell^2_0(\T_N)}\ind{t > \rho^j_{(M)}}\right].
\end{align*}

The terms in the right-hand side are respectively estimated in Lemmas~\ref{lem:cvLipSS} and~\ref{lem:rhoM}, from which the conclusion of the proof of Proposition~\ref{CFTSS} is straightforward.

In the next statement, we respectively denote by $\mathsf{C}_{(M)}$ and $\mathsf{L}_{(M)}$ a bound and a Lipschitz constant (with respect to the $\ell^2_0(\T_N)$ norm) of $\mathbf{b}$ on the ball $\{\|\cdot\|_{\ell^2_0(\T)} \leq M\}$.

\begin{lemma}[Finite-time convergence in the Lipschitz case]\label{lem:cvLipSS}
  Under the assumptions of Proposition~\ref{CFTSS}, for all $t>0$ and $\delta \in (0,1]$, there exists a constant $\mathfrak{C}^\Delta(t,\mathsf{L}_{(M)}, \mathsf{C}_{(M)}, \delta, \Delta t_\mathrm{max})$ such that for any $j \geq 1$,
  \begin{equation*}
    \E\left[\|\overline{\mathbf{U}}^{N, \Delta t_j}(t)-\mathbf{U}^{N,*}(t)\|^2_{\ell^2_0(\T_N)}\ind{t \leq \rho^j_{(M)}}\right] \leq 2\e^{2\mathsf{L}_{(M)}t}\left(\E\left[\left\|\mathbf{U}^{N, \Delta t_j}_0 - \mathbf{U}^{N,*}_0\right\|^2_{\ell^2_0(\T_N)}\right] + \Delta t_j^{1-\delta} \mathfrak{C}^\Delta(t,\mathsf{L}_{(M)}, \mathsf{C}_{(M)}, \delta, \Delta t_\mathrm{max})\right).
  \end{equation*}
\end{lemma}
\begin{proof}
  Let $t \geq 0$ and $j \geq 1$. We introduce the notation $n^j_t = \lfloor \frac{t}{\Delta t_j}\rfloor$ and first use~\eqref{SSBE} to write
  \begin{equation*}
    \overline{\mathbf{U}}^{N, \Delta t_j}(t) = \mathbf{U}^{N, \Delta t_j}_{n^j_t} = \mathbf{U}^{N, \Delta t_j}_0 + \sum_{l=0}^{n_t^j-1} \left(\mathbf{U}^{N, \Delta t_j}_{l+1}-\mathbf{U}^{N, \Delta t_j}_l\right) = \mathbf{U}^{N, \Delta t_j}_0 + \Delta t_j \sum_{l=0}^{n_t^j-1} \mathbf{b}\left(\mathbf{U}^{N, \Delta t_j}_{l+\frac12}\right) + \mathbf{W}^{Q,N}(n_t^j \Delta t_j),
  \end{equation*}
  so that, by~\eqref{FVS},
  \begin{equation*}
    \overline{\mathbf{U}}^{N, \Delta t_j}(t) - \mathbf{U}^{N,*}(t) = \mathbf{U}^{N, \Delta t_j}_0 - \mathbf{U}^{N,*}_0 + \sum_{l=0}^{n_t^j-1} \int_{l\Delta t_j}^{(l+1)\Delta t_j} \left(\mathbf{b}\left(\mathbf{U}^{N, \Delta t_j}_{l+\frac12}\right) - \mathbf{b}\left(\mathbf{U}^{N,*}(s)\right)\right)\dd s - \mathbf{R}^{N,\Delta t_j}(t),
  \end{equation*}
  with 
  \begin{equation*}
    \mathbf{R}^{N,\Delta t_j}(t) = \int_{n^j_t \Delta t_j}^t \mathbf{b}\left(\mathbf{U}^{N,*}(s)\right)\dd s + \mathbf{W}^{Q,N}(t)-\mathbf{W}^{Q,N}(n^j_t\Delta t_j).
  \end{equation*}
  
  We now assume that $t \leq \rho^j_{(M)}$. Since, by~\eqref{FB0}, we have $\|\mathbf{U}^{N, \Delta t_j}_{l+\frac12}\|_{\ell^2_0(\T_N)} \leq \|\mathbf{U}^{N, \Delta t_j}_l\|_{\ell^2_0(\T_N)}$, then for any $l \leq n^j_t - 1$ and $s \in [l\Delta t_j, (l+1)\Delta t_j]$,
  \begin{align*}
    \left\|\mathbf{b}\left(\mathbf{U}^{N, \Delta t_j}_{l+\frac12}\right) - \mathbf{b}\left(\mathbf{U}^{N,*}(s)\right)\right\|_{\ell^2_0(\T_N)} &\leq \mathsf{L}_{(M)} \left\|\mathbf{U}^{N, \Delta t_j}_{l+\frac12} - \mathbf{U}^{N,*}(s)\right\|_{\ell^2_0(\T_N)}\\
    &\leq \mathsf{L}_{(M)} \left(\left\|\mathbf{U}^{N, \Delta t_j}_{l+\frac12} - \mathbf{U}^{N, \Delta t_j}_l\right\|_{\ell^2_0(\T_N)} + \left\|\overline{\mathbf{U}}^{N, \Delta t_j}(s) - \mathbf{U}^{N,*}(s)\right\|_{\ell^2_0(\T_N)}\right)\\
    &\leq \mathsf{L}_{(M)} \left(\Delta t_j \mathsf{C}_{(M)} + \left\|\overline{\mathbf{U}}^{N, \Delta t_j}(s) - \mathbf{U}^{N,*}(s)\right\|_{\ell^2_0(\T_N)}\right).
  \end{align*}
  Likewise, rewriting
  \begin{equation*}
    \mathbf{R}^{N,\Delta t_j}(t) = \int_{n^j_t \Delta t_j}^t \left(\mathbf{b}\left(\mathbf{U}^{N,*}(s)\right)-\mathbf{b}\left(\overline{\mathbf{U}}^{N, \Delta t_j}(s)\right)\right)\dd s + \left(t- n_t \Delta t_j\right)\mathbf{b}\left(\mathbf{U}^{N,\Delta t_j}_{n_t^j}\right) + \mathbf{W}^{Q,N}(t)-\mathbf{W}^{Q,N}(n^j_t\Delta t_j)
  \end{equation*}
  we get
  \begin{equation*}
    \left\|\mathbf{R}^{N,\Delta t_j}(t)\right\|_{\ell^2_0(\T_N)} \leq \mathsf{L}_{(M)}\int_{n^j_t \Delta t_j}^t \left\|\overline{\mathbf{U}}^{N, \Delta t_j}(s) - \mathbf{U}^{N,*}(s)\right\|_{\ell^2_0(\T_N)}\dd s + \Delta t_j \mathsf{C}_{(M)} + \left\|\mathbf{W}^{Q,N}(t)-\mathbf{W}^{Q,N}(n^j_t\Delta t_j)\right\|_{\ell^2_0(\T_N)}.
  \end{equation*} 
  
  We deduce that
  \begin{equation*}
    \left\|\overline{\mathbf{U}}^{N, \Delta t_j}(t) - \mathbf{U}^{N,*}(t)\right\|_{\ell^2_0(\T_N)} \leq \left\|\mathbf{U}^{N, \Delta t_j}_0 - \mathbf{U}^{N,*}_0\right\|_{\ell^2_0(\T_N)} + \mathsf{L}_{(M)}\int_0^t \left\|\overline{\mathbf{U}}^{N, \Delta t_j}(s) - \mathbf{U}^{N,*}(s)\right\|_{\ell^2_0(\T_N)} \dd s + \mathfrak{c}^{\Delta,j}(t),
  \end{equation*}
  with
  \begin{equation*}
    \mathfrak{c}^{\Delta,j}(t) = \Delta t_j \mathsf{C}_{(M)}\left(\mathsf{L}_{(M)} t + 1\right) + \max_{l=0, \dots, n_t^j} \sup_{s \in [l \Delta t_j, (l+1)\Delta t_j)} \left\|\mathbf{W}^{Q,N}(s)-\mathbf{W}^{Q,N}(l\Delta t_j)\right\|_{\ell^2_0(\T_N)}.
  \end{equation*}
  Therefore, by Gr\"onwall's lemma,
  \begin{equation*}
    \left\|\overline{\mathbf{U}}^{N, \Delta t_j}(t) - \mathbf{U}^{N,*}(t)\right\|_{\ell^2_0(\T_N)} \leq \e^{\mathsf{L}_{(M)}t}\left(\left\|\mathbf{U}^{N, \Delta t_j}_0 - \mathbf{U}^{N,*}_0\right\|_{\ell^2_0(\T_N)} + \mathfrak{c}^{\Delta,j}(t)\right),
  \end{equation*}
  and then by Jensen's inequality,
  \begin{equation*}
    \E\left[\left\|\overline{\mathbf{U}}^{N, \Delta t_j}(t) - \mathbf{U}^{N,*}(t)\right\|_{\ell^2_0(\T_N)}^2\ind{t \leq \rho^j_{(M)}}\right] \leq 2\e^{2\mathsf{L}_{(M)}t}\left(\E\left[\left\|\mathbf{U}^{N, \Delta t_j}_0 - \mathbf{U}^{N,*}_0\right\|^2_{\ell^2_0(\T_N)}\right] + \E\left[\mathfrak{c}^{\Delta,j}(t)^2\right]\right).
  \end{equation*}
  
  It remains to estimate 
  \begin{equation*}
    \E\left[\mathfrak{c}^{\Delta,j}(t)^2\right] \leq 2\left(\Delta t_j^2 \mathsf{C}_{(M)}^2\left(\mathsf{L}_{(M)} t + 1\right)^2 + \E\left[\max_{l=0, \dots, n_t^j} \sup_{s \in [l \Delta t_j, (l+1)\Delta t_j)} \left\|\mathbf{W}^{Q,N}(s)-\mathbf{W}^{Q,N}(l\Delta t_j)\right\|^2_{\ell^2_0(\T_N)}\right]\right).
  \end{equation*} 
  By the Markov property and scaling invariance for the Wiener process $\mathbf{W}^{Q,N}$, the random variables
  \begin{equation*}
    G_l := \sup_{s \in [l \Delta t_j, (l+1)\Delta t_j)} \left\|\mathbf{W}^{Q,N}(s)-\mathbf{W}^{Q,N}(l\Delta t_j)\right\|^2_{\ell^2_0(\T_N)}, \qquad l \geq 0, 
  \end{equation*}
  are independent and identically distributed, with, for any $p \in [1,+\infty)$,
  \begin{equation*}
    \E[G_l^p] = \Delta t_j^p \mathsf{g}_p,
  \end{equation*}
  for some constant $\mathsf{g}_p$ which depends on $p$ and the vectors $\mathbf{g}^k$ but not on $\Delta t_j$. As a consequence, we deduce from Jensen's inequality that
  \begin{equation*}
    \E\left[\max_{l=0, \dots, n_t^j} G_l\right] \leq \E\left[\max_{l=0, \dots, n_t^j} G_l^p\right]^{1/p} \leq \E\left[\sum_{l=0}^{n_t^j} G_l^p\right]^{1/p} = (n_t^j+1)^{1/p} \Delta t_j \mathsf{g}_p^{1/p} \leq (t+\Delta t_\mathrm{max}) \Delta t_j^{1-1/p} \mathsf{g}_p^{1/p},
  \end{equation*}
  which yields the claimed estimate by taking $p = 1/\delta$.
\end{proof}

\begin{lemma}[Uniform control over $\rho^j_{(M)}$]\label{lem:rhoM}
  Under the assumptions of Proposition~\ref{CFTSS}, for any $t \geq 0$,
  \begin{equation*}
    \lim_{M \to +\infty} \limsup_{j \to +\infty} \E\left[\|\overline{\mathbf{U}}^{N, \Delta t_j}(t)-\mathbf{U}^{N,*}(t)\|^2_{\ell^2_0(\T_N)}\ind{t > \rho^j_{(M)}}\right] = 0.
  \end{equation*}
\end{lemma}

The proof of Lemma~\ref{lem:rhoM} follows the same outline as Lemma~\ref{lem:tauM}. Namely, it combines $\ell^4_0(\T_N)$ bounds over $\overline{\mathbf{U}}^{N, \Delta t_j}(t)$ and $\mathbf{U}^{N,*}(t)$ uniformly in $j$, and the fact that by virtue of the Markov property, the probability of the event $\{ t > \rho^j_{(M)} \}$ goes to $0$ as $M\to+\infty$. The $\ell^4_0(\T_N)$ bounds follow from Proposition~\ref{prop:unifestimSS} and~\eqref{eq:supinduction}, in which $\E[\|\mathbf{U}^{N,*}_0\|^4_{\ell^4_0(\T_N)}]$ is bounded from above by $\mathsf{C}^{\Delta,0,4}$ thanks to Corollary~\ref{cor:relcomSS}. The control of the term $\P ( t > \rho^j_{(M)})$ relies on the uniform (in $j$) bound on $\E[\sup_{s \in [0,t]} \|\overline{\mathbf{U}}^{N,\Delta t_j}(t)\|_{\ell^2_0(\T_N)}^2]$ stated in the next lemma, the proof of which is postponed to Appendix~\ref{app:estim}.

\begin{lemma}[Finite-time uniform $\ell^2_0(\T_N)$ bound on $\overline{\mathbf{U}}^{N, \Delta t_j}$]\label{FTbound}
  Under the assumptions of Proposition~\ref{CFTSS}, for all $t \geq 0$ there exists a constant $\mathsf{S}^{\Delta, 0,2}_t$ such that 
  \begin{equation*}
    \sup_{j\geq1} \E \left[ \sup_{s \in [0,t]} \left\|\overline{\mathbf{U}}^{N, \Delta t_j}(t) \right\|_{\ell^2_0(\T_N)}^2 \right] \leq \mathsf{S}^{\Delta, 0,2}_t.
  \end{equation*}
\end{lemma}

In order to complete the proof of Proposition~\ref{CFTSS}, it only remains to check that 
\begin{equation*}
  \lim_{j \to +\infty} \E\left[\left\|\mathbf{U}^{N, \Delta t_j}_0 - \mathbf{U}^{N,*}_0\right\|^2_{\ell^2_0(\T_N)}\right] = 0,
\end{equation*}
which follows from the $\ell^4_0(\T_N)$ estimates from Proposition~\ref{prop:unifestim} and Corollary~\ref{cor:relcomSS} by uniform integrability. A more detailed proof of this result is contained in~\cite[Section~3.4]{Mar19}.

\subsection{Discussion of the rate of convergence}\label{ss:rateSS}

Unlike Lemma~\ref{lem:ordre1}, our proof of Lemma~\ref{lem:cvLipSS} is not designed to yield a strong error estimate valid in the long time limit, because it does not directly exploit the decomposition of $\mathbf{b}$ into dissipative and contractive parts. Let us however point out the fact that all constants involved in the uniform in $\Delta t$ estimates in this section also turn out to be uniform in $N$, which is not crucial for our purpose since we establish results for a fixed value of $N$, but might indicate that both limits $\Delta t \to 0$ and $N \to +\infty$, and associated rates of convergence, could be studied simultaneously.

As far as weak error estimates are concerned, one might expect in the Lipschitz case a \emph{weak} error of order $\Delta t$ between $\vartheta_{N, \Delta t}$ and $\vartheta_N$, as is the case for explicit Euler schemes~\cite{Tal90}. For gradient SDEs with a non globally Lipschitz continuous drift, the weak backward error analysis of split-step schemes also shows order $\Delta t$~\cite{Kop14}. For the SDE~\eqref{FVS}, we also observe order $\Delta t$ on the numerical simulations of Section~\ref{s:newnum}, even with a non small (with respect to $\nu$) and non Lipschitz continuous flux function $A$.

\section{Orders of convergence: analytic case and numerical study}\label{s:newnum}

\subsection{Analytic case}\label{ss:analytic}

In this subsection, we consider the case where the flux function $A$ vanishes, so that~\eqref{SSCL} is the stochastic heat equation
\begin{equation}\label{eq:SSCL-lin}
  \dd u(t) = \nu \partial_{xx} u(t) \dd t + \sum_{k \geq 1} g^k \dd W^k(t),
\end{equation}
and the SDE~\eqref{FVS} writes
\begin{equation}\label{eq:FVS-lin}
  \dd \mathbf{U}^N(t) = \nu \mathbf{D}^{(2)}_N \mathbf{U}^N(t) \dd t + \sum_{k \geq 1} \mathbf{g}^k \dd W^k(t).
\end{equation}

\subsubsection{Computation of $\mu$ and $\vartheta_N$} Since the drift of the SDE~\eqref{eq:FVS-lin} is linear, the process $(\mathbf{U}^N(t))_{t \geq 0}$ is Gaussian, and its stationary distribution $\vartheta_N$ is the centered Gaussian measure on $\R^N_0$ with covariance matrix $\mathbf{K}_N$ solution to the Lyapunov equation
\begin{equation*}
  \nu\mathbf{K}_N\mathbf{D}^{(2)}_N + (\nu\mathbf{K}_N\mathbf{D}^{(2)}_N)^\top + \mathbf{Q}_N = 0, \qquad \mathbf{Q}_N := \sum_{k \geq 1} \mathbf{g}^k {\mathbf{g}^k}^\top.
\end{equation*}
The solution to this equation admits the explicit expression
\begin{equation*}
  \mathbf{K}_N = \int_0^{+\infty} \e^{t \nu\mathbf{D}^{(2)}_N} \mathbf{Q}_N \e^{t \nu\mathbf{D}^{(2)}_N} \dd t = \sum_{k \geq 1} \int_0^{+\infty} \bm\rho_N^k(t)\bm\rho_N^k(t)^\top \dd t,
\end{equation*}
where $\bm\rho_N^k(t) = \e^{t \nu\mathbf{D}^{(2)}_N} \mathbf{g}^k$ is the solution to the discrete heat equation
\begin{equation*}
  \frac{\dd}{\dd t} \bm\rho_N^k(t) = \nu\mathbf{D}^{(2)}_N\bm\rho_N^k(t), \qquad \bm\rho_N^k(0) = \mathbf{g}^k.
\end{equation*}
Using the duality relation
\begin{equation*}
  \langle v, \Psi_N \mathbf{w}\rangle_{L^2_0(\T_N)} = \langle \Pi_N v, \mathbf{w}\rangle_{\ell^2_0(\T_N)}, \qquad v \in L^2_0(\T), \quad \mathbf{w} \in \R^N_0,
\end{equation*}
we deduce that the pushforward measure $\mu_N$ is the centered Gaussian measure on $L^2_0(\T)$ with covariance operator $K_N : L^2_0(\T) \to L^2_0(\T)$ defined by, for any $v,w \in L^2_0(\T)$,
\begin{align*}
  \langle v, K_N w\rangle_{L^2_0(\T)} &= \sum_{k \geq 1} \int_0^{+\infty} \langle \bm\rho_N^k(t), \Pi_N v\rangle_{\ell^2_0(\T_N)} \langle \bm\rho_N^k(t), \Pi_N w\rangle_{\ell^2_0(\T_N)} \dd t\\
  &= \sum_{k \geq 1} \int_0^{+\infty} \langle \Psi_N\bm\rho_N^k(t), v\rangle_{L^2_0(\T)} \langle \Psi_N\bm\rho_N^k(t), w\rangle_{L^2_0(\T)} \dd t.
\end{align*}

A similar computation at the infinite-dimensional level of~\eqref{eq:SSCL-lin} shows that $\mu$ is the centered Gaussian measure on $L^2_0(\T)$ with covariance operator $K : L^2_0(\T) \to L^2_0(\T)$ defined by, for any $v,w \in L^2_0(\T)$,
\begin{equation*}
  \langle v, K w\rangle_{L^2_0(\T)} = \sum_{k \geq 1} \int_0^{+\infty} \langle r^k(t), v\rangle_{L^2_0(\T)} \langle r^k(t), w\rangle_{L^2_0(\T)} \dd t,
\end{equation*}
where $r^k(t)$ is the solution to the heat equation
\begin{equation*}
  \partial_t r^k(t,x) = \nu\partial_{xx} r^k(t,x), \qquad r^k(0,x) = g^k(x).
\end{equation*}

\subsubsection{Computation of $W_2(\mu_N,\mu)$} In order to compute explicitly $W_2(\mu_N,\mu)$, we now assume that $g^k = 0$ for $k \geq 2$, and take
\begin{equation*}
  g^1(x) = g(x) := \sqrt{2}\sin\left(2\pi m_0 x\right),
\end{equation*}
for some $m_0 \in \N^*$. The main advantage of this choice lies in the spectral identities
\begin{equation*}
  \partial_{xx} g = -\lambda g, \qquad \lambda = (2\pi m_0)^2,
\end{equation*}
and, with $\mathbf{g} = \Pi_N g$,
\begin{equation*}
  \mathbf{D}^{(2)}_N \mathbf{g} = -\lambda_N \mathbf{g}, \qquad \lambda_N = 2N^2\left(1-\cos\left(\frac{2\pi m_0}{N}\right)\right).
\end{equation*}
These identities allow to compute the time integrals appearing in the operators $K$ and $K_N$ and yield, for any $v,w \in L^2_0(\T)$,
\begin{equation*}
  \langle v, Kw\rangle_N = \frac{1}{2\nu\lambda}\langle g, v\rangle_{L^2_0(\T)}\langle g, w\rangle_{L^2_0(\T)}, \qquad \langle v, K_Nw\rangle_N = \frac{1}{2\nu\lambda_N}\langle \Psi_N\mathbf{g}, v\rangle_{L^2_0(\T)}\langle \Psi_N\mathbf{g}, w\rangle_{L^2_0(\T)}.
\end{equation*}
As a consequence, both operators $K$ and $K_N$ have rank $1$, and it follows from standard results on finite-dimensional Gaussian vectors~\cite{DL82} that the optimal coupling between $\mu$ and $\mu_N$ in Definition~\ref{WD} is given by the law of the pair of random variables $(u,u^N)$ defined by
\begin{equation*}
  u = \frac{Z}{\sqrt{2\nu\lambda}}g, \qquad u^N = \varepsilon_N\frac{Z}{\sqrt{2\nu\lambda_N}}\Psi_N\mathbf{g}, \qquad \varepsilon_N := \sgn\left(\langle g, \Psi_N \mathbf{g}\rangle_{L^2_0(\T)}\right),
\end{equation*}
where $Z$ is a standard, one-dimensional Gaussian variable. The Wasserstein distance $W_2(\mu_N,\mu)$ then writes
\begin{equation*}
  W_2(\mu_N,\mu) = \left\|\frac{g}{\sqrt{2\nu\lambda}} - \frac{\varepsilon_N\Psi_N\mathbf{g}}{\sqrt{2\nu\lambda_N}}\right\|_{L^2_0(\T)}.
\end{equation*}
For $N$ large enough, $\varepsilon_N = 1$ and $\lambda_N = \lambda + \mathrm{O}(1/N^2)$, so that
\begin{equation*}
  N W_2(\mu_N,\mu) \sim \frac{N}{\sqrt{2\nu\lambda}}\left\|g-\Psi_N\mathbf{g}\right\|_{L^2_0(\T)} \to \frac{1}{\sqrt{24\nu\lambda}}\|g\|_{H^1_0(\T)},
\end{equation*}
which confirms that the rate $1/N$ derived in Subsection~\ref{ss:rate} is sharp in this case. 

\subsubsection{Time discretisation} The split-step scheme associated with~\eqref{eq:FVS-lin} rewrites
\begin{equation*}
  \mathbf{U}^{N, \Delta t}_{n+1} = \left(\mathbf{I}-\nu\Delta t \mathbf{D}^{(2)}_N\right)^{-1}\mathbf{U}^{N, \Delta t}_n + \Delta \mathbf{W}^{Q,N}_{n+1},
\end{equation*}
which shows that its invariant measure is the centered Gaussian measure on $\R^N_0$ with covariance matrix $\mathbf{K}_{N, \Delta t}$ given by
\begin{equation*}
  \mathbf{K}_{N, \Delta t} = \Delta t \sum_{n=0}^{+\infty} \left(\left(\mathbf{I}-\nu\Delta t \mathbf{D}^{(2)}_N\right)^{-1}\right)^n\mathbf{Q}_N\left(\left(\mathbf{I}-\nu\Delta t \mathbf{D}^{(2)}_N\right)^{-1}\right)^n.
\end{equation*}
With the one-dimensional noise introduced above, this expression reduces to
\begin{equation*}
  \mathbf{K}_{N, \Delta t} = \Delta t \frac{(1+\nu\Delta t \lambda_N)^2}{(1+\nu\Delta t \lambda_N)^2-1}\mathbf{g}\mathbf{g}^\top.
\end{equation*}
It follows that
\begin{equation*}
  W_2(\vartheta_{N,\Delta t},\vartheta_N) = \left|\sqrt{\frac{1}{2\nu\lambda_N}}-\sqrt{\Delta t \frac{(1+\nu\Delta t \lambda_N)^2}{(1+\nu\Delta t \lambda_N)^2-1}}\right| \|\mathbf{g}\|_{\ell^2_0(\T_N)} = \Delta t(1 + \mathrm{O}(\Delta t))\sqrt{\frac{\nu \lambda_N}{2}}\|\mathbf{g}\|_{\ell^2_0(\T_N)},
\end{equation*}
which shows the order $\Delta t$ for the Wasserstein distance between $\vartheta_{N,\Delta t}$ and $\vartheta_N$, uniformly in $N$ since we recall that $\lambda_N = \lambda + \mathrm{O}(1/N^2)$.

\subsection{Numerical experiments}\label{ss:exp}

In this subsection, we study numerically the \emph{weak error} between $\vartheta_{N,\Delta t}$ and $\vartheta_N$, as a function of $\Delta t$. More precisely, we take as a test function
\begin{equation*}
  \Phi(\mathbf{v}) = \exp\left(-\|\mathbf{v}\|^2_{\ell^2_0(\T_N)}\right), \quad \mathbf{v} \in \R^N_0,
\end{equation*}
and estimate
\begin{equation*}
  \mathtt{err}_N(\Delta t) := \left|\mathbb{E}\left[\Phi(\mathbf{V}^{N,\Delta t})\right]-\mathbb{E}\left[\Phi(\mathbf{V}^N)\right]\right|, \qquad \mathbf{V}^{N, \Delta t} \sim \vartheta_{N, \Delta t}, \quad \mathbf{V}^N \sim \vartheta_N.
\end{equation*}
Notice that, since $\Phi$ is globally Lipschitz continuous, the quantity $\mathtt{err}_N(\Delta t)$ is controlled by $W_2(\vartheta_{N, \Delta t}, \vartheta_N)$ and is therefore at most of order $\Delta t$ in the case $\alpha=0$.

We work with the Burgers equation, for which $A(v) = \alpha v^2/2$, with a strength parameter $\alpha \geq 0$. We keep the one-dimensional noise introduced in Subsection~\ref{ss:analytic} and take the values $m_0=1$ in order to minimise spatial oscillations. The typical behaviour of a solution strongly depends on the relative orders of magnitude of the non-linear term and the viscous term. Whether the equation is viscous-driven or flux-driven, the noise-induced spatial oscillations are either dissipated or non-linearly transported and transformed to oscillations with higher frequency. Therefore, the exact order of convergence derived in Subsection~\ref{ss:analytic} holds in a situation which is far from representative of what ought to be expected from the Burgers equation.

In the case $\alpha=0$, under the invariant measure $\mu$, the functions $\partial_{xx} u$ and $u \partial_x u$ have respective orders of magnitude $\nu^{-1/2}$ and $\nu^{-1}$ in the $L^2_0(\T)$ norm. Therefore, we shall study the four identified regimes of the solution associated with the different ranges where, \emph{ceteris paribus}, the value of $\alpha$ lies:
\begin{itemize}
  \item the linear regime $\alpha=0$;
  \item the viscous regime $|\alpha u \partial_x u| \ll |\nu \partial_{xx} u|$, that is to say $\alpha \ll \nu^{3/2}$;
  \item the equilibrated regime $|\alpha u \partial_x u| \simeq |\nu \partial_{xx} u|$, that is to say $\alpha \simeq \nu^{3/2}$;
  \item the inviscid regime $|\alpha u \partial_x u| \gg |\nu \partial_{xx} u|$, that is to say $\alpha \gg \nu^{3/2}$.
\end{itemize}
For numerical experiments we will take $\nu=0.1$ and the different regimes correspond to $\alpha=0$, $\alpha=0.01\nu^{3/2}$, $\alpha=\nu^{3/2}$ and $\alpha=100\nu^{3/2}$.

\subsubsection{Analytic results for $\alpha=0$} In the Gaussian case $\alpha=0$, the expectations involved in the definition of $\mathtt{err}_N(\Delta t)$ are analytic and write, with the notation of Subsection~\ref{ss:analytic},
\begin{equation*}
  \begin{array}{rll}
    \displaystyle\E[\Phi(\mathbf{V}^{N,\Delta t})] &\displaystyle= \sqrt{\frac{1}{1+2\kappa_{N,\Delta t}}}, &\displaystyle\kappa_{N, \Delta t} = \frac{\Delta t(1+\nu\Delta t \lambda_N)^2}{(1+\nu\Delta t \lambda_N)^2-1}\|\mathbf{g}\|^2_{\ell^2_0(\T_N)},\\[3mm]
    \displaystyle\E[\Phi(\mathbf{V}^N)] &\displaystyle= \sqrt{\frac{1}{1+2\kappa_N}}, &\displaystyle\kappa_N := \frac{\|\mathbf{g}\|^2_{\ell^2_0(\T_N)}}{2\nu\lambda_N},
  \end{array}
\end{equation*}
with $\|\mathbf{g}\|^2_{\ell^2_0(\T_N)} = (\sin(\frac{\pi m_0}{N})/\frac{\pi m_0}{N})^2$ as soon as $N>2m_0$. We deduce from these expressions that $\mathtt{err}_N(\Delta t)$ is of order $\Delta t$, uniformly in $N$.

\subsubsection{Ergodic approximation of $\E[\Phi(\mathbf{V}^{N,\Delta t})]$}\label{sss:ergo} For $N \geq 1$ and $\Delta t > 0$, let us denote
\begin{equation*}
  I^{N,\Delta t} := \E[\Phi(\mathbf{V}^{N,\Delta t})].
\end{equation*}
From Remark~\ref{rm:ergo} and the Central Limit Theorem for Markov chains, we may expect that there exist $\Sigma^{N,\Delta t}$ such that for $T$ large enough, the random variable 
\begin{equation*}
  \widehat{I}^{N,\Delta t}_T = \frac{1}{n}\sum_{l=0}^{n-1} \Phi(\mathbf{U}^{N,\Delta t}_l), \qquad T = n\Delta t,
\end{equation*}
has the Gaussian distribution
\begin{equation*}
  \widehat{I}^{N,\Delta t}_T \sim \mathcal{N}\left(I^{N,\Delta t}, \frac{(\Sigma^{N,\Delta t})^2}{T}\right).
\end{equation*}
The evolution of the empirical average $\widehat{I}^{N,\Delta t}_t$, $t \in [0,T]$, along a single trajectory of the split-step scheme is plotted on Figure~\ref{fig:empirical_average} for the four regimes of $\alpha$. The value of $I^{N,\Delta t}$ seems to be the same for all regimes $\alpha=0$, $\alpha \ll \nu^{3/2}$ and $\alpha \simeq \nu^{3/2}$, and to be significantly larger for $\alpha \gg \nu^{3/2}$.

\begin{center}
  \begin{figure}[ht]
    \includegraphics[width=.5\textwidth]{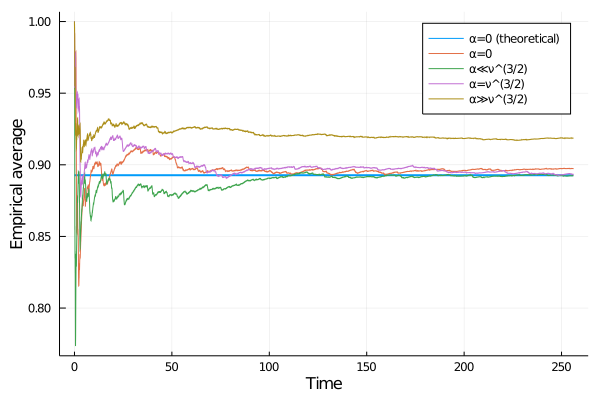}
    \caption{Evolution of $\widehat{I}^{N,\Delta t}_t$ for $t \in [0,T]$ and the four regimes of $\alpha$. In the case $\alpha=0$, the theoretical value of $I^{N,\Delta t}$ is superposed as a horizontal line. Here, $N=32$, $\Delta t=2^{-10}$ and $T=256$.}
    \label{fig:empirical_average}
  \end{figure}
\end{center}

In~§~\ref{sss:weakdt}, Monte Carlo confidence intervals for $I^{N,\Delta t}$ are obtained by fixing a time horizon $T \gg 1$ and estimating the parameters $I^{N,\Delta t}$ and $\Sigma^{N,\Delta t}$ from a sample of $M \gg 1$ independent realisations $\widehat{I}^{N,\Delta t,(1)}_T, \ldots, \widehat{I}^{N,\Delta t,(M)}_T$.

\subsubsection{Weak error between $\vartheta_{N,\Delta t}$ and $\vartheta_N$}\label{sss:weakdt} Let $N=32$, $\Delta t_\mathrm{min} = 2^{-8}$ and $\Delta t_\mathrm{max} = 2^{-1}$. Our purpose is to plot the evolution of $\mathtt{err}_N(\Delta t)$ for $\Delta t \in [\Delta t_\mathrm{min}, \Delta t_\mathrm{max}]$. To this aim, we fix $\overline{\Delta t} = 2^{-10} \ll \Delta t_\mathrm{min}$, approximate 
\begin{equation*}
  \mathtt{err}_N(\Delta t) \simeq \left|I^{N,\Delta t}-I^{N,\overline{\Delta t}}\right|,
\end{equation*}
and compute both terms in the right-hand side thanks to Monte Carlo simulations as is described in §~\ref{sss:ergo}. 

The resulting error curves are plotted on Figure~\ref{fig:err_dt} for the four regimes of $\alpha$. The weak error is observed to be of order $\Delta t$, uniformly in $\alpha$.

\begin{center}
  \begin{figure}[ht]
    \includegraphics[width=.5\textwidth]{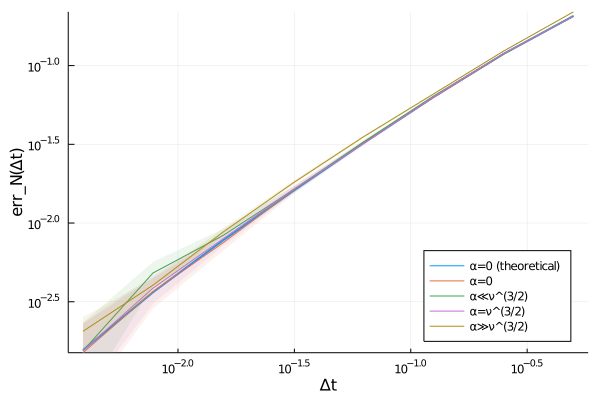}
    \caption{Evolution of $\mathtt{err}_N(\Delta t)$ for $\Delta t \in [\Delta t_\mathrm{min}, \Delta t_\mathrm{max}]$ and the four regimes of $\alpha$, with associated confidence intervals. In the case $\alpha=0$, the theoretical value of $\mathtt{err}_N(\Delta t)$ is superposed. Here, the final time horizon is $T=256$ and the number of copies for the Monte Carlo estimation is $M=200$.}
    \label{fig:err_dt}
  \end{figure}
\end{center}

\begin{remark}\label{rk:ci}
  The error curves of Figure~\ref{fig:err_dt} are plotted with the same final time horizon $T$, and the same number of Monte Carlo realisations $M$, for all values of $\Delta t$. As it turns out that the asymptotic variance $\Sigma^{N,\Delta t}$ is approximately uniform in $N$ and $\Delta t$, this results in the estimator of $\mathtt{err}_N(\Delta t)$ having the same variance for all $\Delta t$. This is the reason why, in log-log coordinates, confidence intervals appear to be larger for smaller values of $\Delta t$.
\end{remark}

\appendix

\section{Proofs of auxiliary inequalities}\label{app}

\begin{proof}[Proof of Lemma~\ref{L1}] Let $\mathbf v \in \R^N_0$ and $q \in 2\N^*$. By~\eqref{eq:sbp} we have
\begin{equation*}
  \langle \mathbf{v}^{q-1}, \mathbf{D}^{(1,-)}\overline{\mathbf{A}}^N(\mathbf{v})\rangle_{\ell^2(\T_N)} = -\langle \mathbf{D}^{(1,+)}\mathbf{v}^{q-1}, \overline{\mathbf{A}}^N(\mathbf{v})\rangle_{\ell^2(\T_N)} = -\sum_{i \in \T_N} (v_{i+1}^{q-1}-v_i^{q-1})\overline{A}(v_i,v_{i+1}).
\end{equation*}
For any $i \in \T_N$, using~\eqref{E1} and~\eqref{eq:consistency}, we get
\begin{align*}
  (v_{i+1}^{q-1}-v_i^{q-1})\overline{A}(v_i,v_{i+1}) &= \int_{v_i^{q-1}}^{v_{i+1}^{q-1}} \overline{A}(v_i,v_{i+1})\dd z\leq \int_{v_i^{q-1}}^{v_{i+1}^{q-1}} \overline{A}\left( z^{1/(q-1)} , z^{1/(q-1)} \right)\dd z=\mathcal{A}_q\left(v_{i+1}^{q-1}\right)-\mathcal{A}_q\left(v_i^{q-1}\right),
\end{align*}
where $\mathcal{A}_q$ denotes a function defined on $\R$ such that $\mathcal{A}_q' (z) = A( z^{1/(q-1)} )$. Since the sum over $i \in \T_N$ of all terms in the right-hand side vanish, the proof is completed.
\end{proof}

\begin{proof}[Proof of Lemma~\ref{comp}] 
For $\mathbf v\in\R^N_0$ and $p\in2\N^*$, we first write
\begin{align*}
  \left\langle\mathbf D^{(1,+)}_N(\mathbf v^{p-1}),\mathbf D^{(1,+)}_N\mathbf v\right\rangle_{\ell^2_0(\T_N)}&=N\sum_{i \in \T_N}\left(v^{p-1}_{i+1}-v^{p-1}_i\right)\left(v_{i+1}-v_i\right)\\
  &=N(p-1)\sum_{i \in \T_N}\left(v_{i+1}-v_i\right)\int_{v_i}^{v_{i+1}}|z|^{p-2}\dd z\\
  &=N(p-1)\sum_{i \in \T_N}\left(v_{i+1}-v_i\right)\int_{v_i}^{v_{i+1}}\left(|z|^{p/2-1}\right)^2\dd z.
 \end{align*}
 For each $i \in \T_N$, the summand in the right-hand side is a symmetric and nonnegative function of $v_i$ and $v_{i+1}$, therefore whatever the sign of $v_{i+1}-v_i$, Jensen's inequality yields
 \begin{equation*}
   \left(v_{i+1}-v_i\right)\int_{v_i}^{v_{i+1}}\left(|z|^{p/2-1}\right)^2\dd z \geq \left(\int_{v_i}^{v_{i+1}}|z|^{p/2-1}\dd z\right)^2.
 \end{equation*}
 As a consequence, we get
 \begin{align*}
  \left\langle\mathbf D^{(1,+)}_N(\mathbf v^{p-1}),\mathbf D^{(1,+)}_N\mathbf v\right\rangle_{\ell^2_0(\T_N)}&\geq N(p-1)\sum_{i \in \T_N}\left(\int_{v_i}^{v_{i+1}}|z|^{p/2-1}\dd z\right)^2\\
  &=\frac{4N(p-1)}{p^2}\sum_{i \in \T_N}\left(\int_{v_i}^{v_{i+1}}\frac\dd{\dd z}\left(\sgn(z)|z|^{p/2}\right)\dd z\right)^2\\
  &=\frac{4N(p-1)}{p^2}\sum_{i \in \T_N}\left(\sgn\left(v_{i+1}\right)\left|v_{i+1}\right|^{p/2}-\sgn\left(v_i\right)\left|v_i\right|^{p/2}\right)^2\\
  &=\frac{4(p-1)}{p^2} \|\mathbf{D}^{(1,+)}\mathbf{w}\|^2_{\ell^2_0(\T_N)},
 \end{align*}
 where $\mathbf{w}$ is the vector with coordinates $w_i = \sgn\left(v_i\right)\left|v_i\right|^{p/2}$, $i \in \T_N$. This vector does not necessarily belong to $\R^N_0$ so we cannot apply the Poincar\'e inequality~\eqref{eq:poinca-discr} direclty. Let us however notice that since $\mathbf{v} \in \R^N_0$, there exist two indices $i_-, i_+ \in \T_N$ such that $v_{i_-} \leq 0 \leq v_{i_+}$, so that $w_{i_-} \leq 0 \leq w_{i_+}$. As a consequence,
 \begin{align*}
   \left\|\mathbf w\right\|^2_{\ell^2_0(\T_N)}&=\frac1N\sum_{w_i\geq0}|w_i|^2+\frac1N\sum_{w_i<0}|w_i|^2 \\
  &\leq\frac1N\sum_{w_i\geq0}|w_i-w_{i_-}|^2+\frac1N\sum_{w_i<0}|w_i-w_{i_+}|^2\\
  &\leq\frac1N\sum_{w_i\geq0}N\sum_{j \in \T_N}|w_{j+1}-w_j|^2+\frac1N\sum_{w_i<0}N\sum_{j \in \T_N}|w_{j+1}-w_j|^2\\
  &=N\sum_{j\in\T_N}\left|w_{j+1}-w_j\right|^2\\
  &=\left\|\mathbf D^{(1,+)}_N\mathbf w\right\|_{\ell^2_0(\T_N)}^2.
 \end{align*}
 Injecting this estimate in the identity above, we deduce that
 \begin{equation*}
   \left\langle\mathbf D^{(1,+)}_N(\mathbf v^{p-1}),\mathbf D^{(1,+)}_N\mathbf v\right\rangle_{\ell^2_0(\T_N)} \geq \frac{4(p-1)}{p^2} \left\|\mathbf w\right\|^2_{\ell^2_0(\T_N)} = \frac{4(p-1)}{p^2} \left\|\mathbf v\right\|^p_{\ell^p_0(\T_N)},
 \end{equation*}
 which completes the proof.
\end{proof}

\section{Proofs of intermediary results for the uniqueness of invariant measures}\label{app:uniq}

\subsection{The semi-discrete scheme}

\begin{proof}[Proof of Lemma~\ref{entrance}]
We recall that $\mathbf b: \R^N_0 \to \R^N_0$ is locally Lipschitz continuous (for every norm over $\R^N_0$). Let $M>0$ and $\varepsilon>0$. Let us also fix the deterministic values $\mathbf u_0, \mathbf v_0 \in \R^N_0$ satisfying $\|\mathbf u_0\|_{\ell^2_0(\T_N)}\vee\|\mathbf v_0\|_{\ell^2_0(\T_N)}\leq M$, along with the following constants:
\[ t_{\varepsilon,M} :=-\frac1{2\nu}\log\frac{\varepsilon^2}{16M^2} ; \]
\[ \mathsf{L}_{M+\varepsilon} := \text{ Lipschitz constant of $\mathbf{b}$ over the ball } \left\{ \left\|\mathbf \cdot\right\|_{\ell^1_0(\T_N)}\leq M+\varepsilon\right\} ; \]
\[ \delta_\varepsilon := \frac \varepsilon 4 \e^{-\mathsf{L}_{M+\varepsilon}t_{\varepsilon,M}} . \]
Let $(\mathbf U^N(t))_{t\geq0}$ and $(\mathbf V^N(t))_{t\geq0}$ denote two solutions of~\eqref{FVS} with the initial conditions $\mathbf u_0$ and $\mathbf v_0$. We introduce the stopping times
\[ \TA^{\mathbf U^N} := \inf \left\{ t\geq0 : \|\mathbf U^N(t) \|_{\ell^1_0(\T_N)} \geq M+\varepsilon \right\}, \qquad \TA^{\mathbf V^N} := \inf \left\{ t\geq0 : \|\mathbf V^N(t) \|_{\ell^1_0(\T_N)} \geq M+\varepsilon \right\} . \]
Furthermore, we denote by $(\mathbf u^N(t))_{t\geq0}$ and $(\mathbf v^N(t))_{t\geq0}$ the noiseless counterparts of $(\mathbf U^N(t))_{t\geq0}$ and $(\mathbf V^N(t))_{t\geq0}$:
\[ \frac{\dd}{\dd t} \mathbf u^N(t) =\mathbf b\left(\mathbf u^N(t)\right) , \qquad \frac{\dd}{\dd t} \mathbf v^N(t) =\mathbf b\left(\mathbf v^N(t)\right) , \]
with respective initial conditions $\mathbf u_0$ and $\mathbf v_0$.

By Lemma~\ref{L1Cdrift}.(ii) and~\eqref{eq:poinca-discr}, we have
\begin{equation*}
\frac{\dd}{\dd t} \left( \left\| \mathbf u^N(t) \right\|_{\ell^2_0(\T_N)}^2 + \left\| \mathbf v^N(t) \right\|_{\ell^2_0(\T_N)}^2 \right) \leq -2\nu \left( \left\| \mathbf u^N(t) \right\|_{\ell^2_0(\T_N)}^2 + \left\| \mathbf v^N(t) \right\|_{\ell^2_0(\T_N)}^2 \right) ,
\end{equation*}
so that Gr\"onwall's lemma yields the upper bound
\[ \left\| \mathbf u^N(t) \right\|_{\ell^2_0(\T_N)}^2 + \left\| \mathbf v^N(t) \right\|_{\ell^2_0(\T_N)}^2 \leq \left( \left\| \mathbf u_0 \right\|_{\ell^2_0(\T_N)}^2 + \left\| \mathbf v_0 \right\|_{\ell^2_0(\T_N)}^2 \right) \e^{-2\nu t} , \]
meaning that for $t\geq t_{\varepsilon,M}$, we have
\[ \left\| \mathbf u^N(t) \right\|_{\ell^2_0(\T_N)}^2 + \left\| \mathbf v^N(t) \right\|_{\ell^2_0(\T_N)}^2 \leq \frac{\varepsilon^2}8 , \]
and consequently, by~\eqref{eq:normorder},
\[ \left\| \mathbf u^N(t) \right\|_{\ell^1_0(\T_N)} + \left\| \mathbf v^N(t) \right\|_{\ell^1_0(\T_N)} \leq \left\| \mathbf u^N(t) \right\|_{\ell^2_0(\T_N)} + \left\| \mathbf v^N(t) \right\|_{\ell^2_0(\T_N)} \leq \frac \varepsilon 2 . \]

We now consider the event
\[ \left\{ \sup_{t\in[0,t_{\varepsilon,M}]} \left\|\mathbf W^{Q,N}(t)\right\|_{\ell^1_0(\T_N)} \leq \delta_\varepsilon \right\} . \]

For any $t\leq \TA^{\mathbf U^N} \wedge \TA^{\mathbf V^N} \wedge t_{\varepsilon,M}$, the four vectors $\mathbf U^N(t)$, $\mathbf V^N(t)$, $\mathbf u^N(t)$ and $\mathbf v^N(t)$ stay in the ball $\{\|\cdot\|_{\ell^1_0(\T_N)}\leq M+\varepsilon\}$, and thanks to the local Lipschitz continuity assumption on $\mathbf b$ we have
\begin{align*}
  &\left\|\mathbf U^N(t) - \mathbf u^N(t) \right\|_{\ell^1_0(\T_N)} + \left\|\mathbf V^N(t) - \mathbf v^N(t) \right\|_{\ell^1_0(\T_N)}\\
  &= \left\| \int_0^t \left(\mathbf b\left(\mathbf U^N(s)\right)-\mathbf b\left(\mathbf u^N(s)\right) \right) \dd s +\mathbf W^{Q,N}(t) \right\|_{\ell^1_0(\T_N)} + \left\| \int_0^t \left(\mathbf b\left(\mathbf V^N(s)\right)-\mathbf b\left(\mathbf v^N(s)\right) \right) \dd s + \mathbf W^{Q,N}(t) \right\|_{\ell^1_0(\T_N)} \\
&\leq \int_0^t \left( \left\|\mathbf b\left(\mathbf U^N(s)\right)-\mathbf b\left(\mathbf u^N(s)\right) \right\|_{\ell^1_0(\T_N)} + \left\|\mathbf b\left(\mathbf V^N(s)\right)-\mathbf b\left(\mathbf v^N(s)\right) \right\|_{\ell^1_0(\T_N)} \right) \dd s +2 \left\|\mathbf W^{Q,N}(t) \right\|_{\ell^1_0(\T_N)} \\
&\leq \mathsf{L}_{M+\varepsilon} \int_0^t \left( \left\|\mathbf U^N(s) - \mathbf u^N(s) \right\|_{\ell^1_0(\T_N)} + \left\|\mathbf V^N(s) - \mathbf v^N(s) \right\|_{\ell^1_0(\T_N)} \right) \dd s + 2\delta_\varepsilon,
\end{align*}
so by Gr\"onwall's lemma, we have
\begin{equation}\label{gron}
\left\|\mathbf U^N(t) - \mathbf u^N(t) \right\|_{\ell^1_0(\T_N)} + \left\|\mathbf V^N(t) - \mathbf v^N(t) \right\|_{\ell^1_0(\T_N)} \leq 2\delta_\varepsilon \e^{\mathsf{L}_{M+\varepsilon}t} \leq 2\delta_\varepsilon \e^{\mathsf{L}_{M+\varepsilon}t_{\varepsilon,M}} = \frac \varepsilon 2 ,
\end{equation}
for every $t\in [0,\TA^{\mathbf U^N} \wedge \TA^{\mathbf V^N} \wedge t_{\varepsilon,M}]$. But it appears that the case $\TA^{\mathbf U^N}\wedge \TA^{\mathbf V^N} < t_{\varepsilon,M}$ is impossible for small values of $\varepsilon$. Indeed, it would either imply $\| (\mathbf U^N-\mathbf u^N)(\TA^{\mathbf U^N}) \|_{\ell^1_0(\T_N)} \leq \varepsilon/2$ or $\| (\mathbf V^N-\mathbf v^N)(\TA^{\mathbf V^N}) \|_{\ell^1_0(\T_N)} \leq \varepsilon/2$ which is impossible because we have on the one hand
\[ \left\| \mathbf U^N\left(\TA^{\mathbf U^N}\right) \right\|_{\ell^1_0(\T_N)} \geq M+\varepsilon \qquad \left( \text{or } \left\|\mathbf V^N\left(\TA^{\mathbf V^N}\right) \right\|_{\ell^1_0(\T_N)} \geq M+\varepsilon \right) , \]
and on the other hand
\[ \left\| \mathbf u^N\left(\TA^{\mathbf U^N}\right) \right\|_{\ell^1_0(\T_N)}\leq\left\| \mathbf u^N\left(\TA^{\mathbf U^N}\right) \right\|_{\ell^2_0(\T_N)} \leq \left\| \mathbf u_0 \right\|_{\ell^2_0(\T_N)} \leq M \qquad \left( \text{or } \left\| \mathbf v^N\left(\TA^{\mathbf V^N}\right)\right\|_{\ell^1_0(\T_N)} \leq M \right) . \]

Therefore, Inequality~\eqref{gron} holds for all $t \in [0,t_{\varepsilon,M}]$. Thus,
\begin{align*}
  \left\|\mathbf U^N(t_{\varepsilon,M})\right\|_{\ell^1_0(\T_N)} + \left\|\mathbf V^N(t_{\varepsilon,M}) \right\|_{\ell^1_0(\T_N)} &\leq \left\|\mathbf U^N(t_{\varepsilon,M}) - \mathbf u^N(t_{\varepsilon,M}) \right\|_{\ell^1_0(\T_N)} + \left\|\mathbf V^N(t_{\varepsilon,M}) - \mathbf v^N(t_{\varepsilon,M}) \right\|_{\ell^1_0(\T_N)}\\
  &\quad + \left\| \mathbf u^N(t_{\varepsilon,M}) \right\|_{\ell^1_0(\T_N)} + \left\| \mathbf v^N(t_{\varepsilon,M}) \right\|_{\ell^1_0(\T_N)}\\
  & \leq \varepsilon ,
\end{align*}
and we have just shown that
\[ \left\{ \sup_{t\in[0,t_{\varepsilon,M}]} \left\|\mathbf W^{Q,N}(t)\right\|_{\ell^1_0(\T_N)} \leq \delta_\varepsilon \right\} \subset \left\{ \left\|\mathbf U^N(t_{\varepsilon,M}) \right\|_{\ell^1_0(\T_N)} + \left\|\mathbf V^N(t_{\varepsilon,M}) \right\|_{\ell^1_0(\T_N)} \leq \varepsilon \right\} . \]
and therefore,
\[ \P_{(\mathbf u_0,\mathbf v_0)} \left( \left\|\mathbf U^N(t_{\varepsilon,M}) \right\|_{\ell^1_0(\T_N)} + \left\|\mathbf V^N(t_{\varepsilon,M}) \right\|_{\ell^1_0(\T_N)} \leq \varepsilon \right) \geq \P \left( \sup_{t\in[0,t_{\varepsilon,M}]} \left\|\mathbf W^{Q,N}(t)\right\|_{\ell^1_0(\T_N)} \leq \delta_\varepsilon \right) . \]
Notice that the right-hand side does not depend on $\mathbf u_0$ nor $\mathbf v_0$. Furthermore, it is positive since $\mathbf W^{Q,N}$ is an $\R^N$-valued Wiener process. Hence, taking the infimum over $\mathbf u_0$ and $\mathbf v_0$ on the left-hand side yields the wanted result.
\end{proof}

\begin{proof}[Proof of Lemma~\ref{entranceLB}]
From It\^o's formula, we have for all $t\geq0$,
\begin{equation}\label{eq:itoentrance}
  \begin{aligned}
    &\left\|\mathbf U^N(\tau_M\wedge t)\right\|_{\ell^2_0(\T_N)}^2+\left\|\mathbf V^N(\tau_M\wedge t)\right\|_{\ell^2_0(\T_N)}^2\\
    &=\left\|\mathbf u_0\right\|_{\ell^2_0(\T_N)}^2+\left\|\mathbf v_0\right\|_{\ell^2_0(\T_N)}^2 + \int_0^{\tau_M\wedge t}\left\langle\mathbf b(\mathbf U^N(s)),\mathbf U^N(s)\right\rangle_{\ell^2_0(\T_N)} \dd s+\int_0^{\tau_M\wedge t}\left\langle\mathbf b(\mathbf V^N(s)),\mathbf V^N(s)\right\rangle_{\ell^2_0(\T_N)} \dd s\\
    &\quad +\int_0^{\tau_M\wedge t}\left\langle\mathbf U^N(s)+\mathbf V^N(s),\dd\mathbf{W}^{Q,N}(s)\right\rangle_{\ell^2_0(\T_N)}+2\sum_{k\geq1}\int_0^{\tau_M\wedge t}\left\|\mathbf{g}^k\right\|_{\ell^2_0(\T_N)}^2\dd s.
  \end{aligned}
\end{equation}
The fifth term of the right-hand side is a martingale. Indeed, by the Cauchy--Schwarz inequality, Inequality~\eqref{eq:sigma-glob}, and the bound~\eqref{eq:integralinduction}, we have
\begin{align*}
 &\E\left[\sum_{k\geq1}\int_0^{\tau_M\wedge t}\left|\left\langle\mathbf U^N(s)+\mathbf V^N(s),\mathbf{g}^k\right\rangle_{\ell^2_0(\T_N)}\right|^2\dd s\right]\\
 &\leq \left(\sum_{k\geq1}\left\|\mathbf{g}^k\right\|_{\ell^2_0(\T_N)}^2\right)\E\left[\int_0^t\left\|\mathbf U^N(s)+\mathbf V^N(s)\right\|_{\ell^2_0(\T_N)}^2\dd s\right]\\
 &\leq 2\mathsf{D}\left(\E\left[\int_0^t\left\|\mathbf U^N(s)\right\|_{\ell^2_0(\T_N)}^2\dd s\right]+\E\left[\int_0^t\left\|\mathbf V^N(s)\right\|_{\ell^2_0(\T_N)}^2\dd s\right]\right)\\
 &\leq 2\mathsf{D}\left(2\mathsf{c}_0^{(2)}+\mathsf{c}_1^{(2)}\left(\left\|\mathbf u_0\right\|_{\ell^2_0(\T_N)}^2+\left\|\mathbf v_0\right\|_{\ell^2_0(\T_N)}^2\right)+2\mathsf{c}_2^{(2)}t\right)\\
 &<+\infty.
\end{align*}
Thus, taking the expectation in~\eqref{eq:itoentrance}, applying Lemma~\ref{L1Cdrift}.(ii), Inequality~\eqref{eq:sigma-glob}, \eqref{eq:poinca-discr} and~\eqref{eq:entrancetimedef}, we get
\begin{align*}
&\E\left[ \left\|\mathbf U^N(\tau_M \wedge t) \right\|_{\ell^2_0(\T_N)}^2 + \left\|\mathbf V^N(\tau_M \wedge t)\right\|_{\ell^2_0(\T_N)}^2 \right] - \left(\left\| \mathbf u_0 \right\|_{\ell^2_0(\T_N)}^2 + \left\| \mathbf v_0 \right\|_{\ell^2_0(\T_N)}^2\right)\\
&= 2 \E \left[ \int_0^{\tau_M \wedge t} \left( \left\langle\mathbf b\left(\mathbf U^N(s)\right) ,\mathbf U^N(s) \right\rangle_{\ell^2_0(\T_N)} + \left\langle\mathbf b\left(\mathbf V^N(s)\right),\mathbf V^N(s) \right\rangle_{\ell^2_0(\T_N)} \right) \dd s \right] + 2 \E \left[ \int_0^{\tau_M \wedge t} \sum_{k\geq1} \left\| \mathbf{g}^k \right\|_{\ell^2_0(\T_N)}^2 \dd s \right] \\
&\leq - 2 \nu \E \left[ \int_0^{\tau_M \wedge t} \left( \left\|\mathbf D^{(1,+)}_N\mathbf U^N(s)\right\|_{\ell^2_0(\T_N)}^2 + \left\|\mathbf D^{(1,+)}_N\mathbf V^N(s)\right\|_{\ell^2_0(\T_N)}^2 \right) \dd s \right] + 2 \E \left[ \tau_M \wedge t \right] \mathsf{D} \\
&\leq - 2 \nu \E \left[ \int_0^{\tau_M \wedge t} \left( \left\|\mathbf U^N(s)\right\|_{\ell^2_0(\T_N)}^2 + \left\|\mathbf V^N(s)\right\|_{\ell^2_0(\T_N)}^2 \right) \dd s \right] + 2 \E \left[ \tau_M \wedge t \right] \mathsf{D} \\
&\leq 2 \left( \mathsf{D} - \nu M^2 \right) \E \left[ \tau_M \wedge t \right].
\end{align*}
So if we choose $M > \sqrt{\mathsf{D}/ \nu}$, we get
\[ \E [ \tau_M \wedge t ] \leq \frac{\| \mathbf u_0\|_{\ell^2_0(\T_N)}^2 + \| \mathbf v_0 \|_{\ell^2_0(\T_N)}^2}{2 \left( \nu M^2 - \mathsf{D} \right)} , \]
and we deduce from the monotone convergence theorem that $\E [ \tau_M ] = \lim_{t\to\infty} \E [ \tau_M \wedge t ] < +\infty$.
\end{proof}

\subsection{The split-step scheme}

\begin{proof}[Proof of Lemma~\ref{proba}]
First, let $\varepsilon>0$ and let us fix $\mathbf u_0, \mathbf v_0 \in \R^N_0$ such that $\|\mathbf u_0\|_{\ell^2_0(\T_N)}\leq M$ and $\|\mathbf v_0\|_{\ell^2_0(\T_N)}\leq M$.

Let $(\mathbf u^{N,\Delta t}_n)_{n\in\N}$ and $(\mathbf v^{N,\Delta t}_n)_{n\in\N}$ denote the noiseless counterparts of the sequences $(\mathbf U^{N,\Delta t}_n)_{n\in\N}$ and $(\mathbf V^{N,\Delta t}_n)_{n\in\N}$, \textit{i.e.}
\begin{equation}\label{noiseless}
  \mathbf u^{N,\Delta t}_{n+1} = \mathbf u^{N,\Delta t}_n + \Delta t\mathbf b \left( \mathbf u^{N,\Delta t}_{n+1} \right), \qquad \mathbf v^{N,\Delta t}_{n+1} = \mathbf v^{N,\Delta t}_n + \Delta t\mathbf b \left( \mathbf v^{N,\Delta t}_{n+1} \right),
\end{equation}
with initial conditions $\mathbf u_0$ and $\mathbf v_0$. Then $(\mathbf u^{N,\Delta t}_n)_{n\in\N}$ and $(\mathbf v^{N,\Delta t}_n)_{n\in\N}$ are subject to non-perturbed $\ell^2_0(\T_N)$ dissipativity, and consequently the sum of their energies decreases to $0$ over time. Indeed, we have
\begin{align*}
\left\| \mathbf u^{N,\Delta t}_n \right\|_{\ell^2_0(\T_N)}^2 + \left\| \mathbf v^{N,\Delta t}_n \right\|_{\ell^2_0(\T_N)}^2 &= \left\| \mathbf u^{N,\Delta t}_{n+1} - \Delta t\mathbf b\left( \mathbf u^{N,\Delta t}_{n+1}\right) \right\|_{\ell^2_0(\T_N)}^2 + \left\| \mathbf v^{N,\Delta t}_{n+1} - \Delta t\mathbf b\left( \mathbf v^{N,\Delta t}_{n+1}\right) \right\|_{\ell^2_0(\T_N)}^2 \\
&= \left\| \mathbf u^{N,\Delta t}_{n+1} \right\|_{\ell^2_0(\T_N)}^2 + \left\| \mathbf v^{N,\Delta t}_{n+1} \right\|_{\ell^2_0(\T_N)}^2 + ( \Delta t)^2 \left( \left\|\mathbf b\left( \mathbf u_{n+1}\right) \right\|_{\ell^2_0(\T_N)}^2 + \left\|\mathbf b\left( \mathbf v_{n+1}\right) \right\|_{\ell^2_0(\T_N)}^2 \right)\\
&\quad - 2 \Delta t \left( \left\langle \mathbf u^{N,\Delta t}_{n+1} ,\mathbf b\left( \mathbf u^{N,\Delta t}_{n+1}\right) \right\rangle_{\ell^2_0(\T_N)} + \left\langle \mathbf v^{N,\Delta t}_{n+1} ,\mathbf b\left( \mathbf v^{N,\Delta t}_{n+1}\right) \right\rangle_{\ell^2_0(\T_N)} \right)
\end{align*}
therefore, using successively Lemma~\ref{L1Cdrift}.(ii) and~\eqref{eq:poinca-discr}, we get
\begin{align*}
&\left\| \mathbf u^{N,\Delta t}_{n+1} \right\|_{\ell^2_0(\T_N)}^2 + \left\| \mathbf v^{N,\Delta t}_{n+1} \right\|_{\ell^2_0(\T_N)}^2 - \left( \left\| \mathbf u^{N,\Delta t}_n \right\|_{\ell^2_0(\T_N)}^2 + \left\| \mathbf v^{N,\Delta t}_n \right\|_{\ell^2_0(\T_N)}^2 \right)\\
&\leq 2 \Delta t \left( \left\langle \mathbf u^{N,\Delta t}_{n+1} ,\mathbf b \left( \mathbf u^{N,\Delta t}_{n+1} \right) \right\rangle_{\ell^2_0(\T_N)} + \left\langle \mathbf v^{N,\Delta t}_{n+1} ,\mathbf b \left( \mathbf v^{N,\Delta t}_{n+1} \right) \right\rangle_{\ell^2_0(\T_N)} \right) \\
&\leq -2 \Delta t \nu \left( \left\|\mathbf D^{(1,+)}_N \mathbf u^{N,\Delta t}_{n+1} \right\|_{\ell^2_0(\T_N)}^2 + \left\|\mathbf D^{(1,+)}_N \mathbf v^{N,\Delta t}_{n+1} \right\|_{\ell^2_0(\T_N)}^2 \right) \\
&\leq -2 \Delta t \nu \left( \left\| \mathbf u^{N,\Delta t}_{n+1} \right\|_{\ell^2_0(\T_N)}^2 + \left\| \mathbf v^{N,\Delta t}_{n+1} \right\|_{\ell^2_0(\T_N)}^2 \right)
\end{align*}
so that
\[ \left\| \mathbf u^{N,\Delta t}_{n+1} \right\|_{\ell^2_0(\T_N)}^2 + \left\| \mathbf v^{N,\Delta t}_{n+1} \right\|_{\ell^2_0(\T_N)}^2 \leq \frac{1}{1+2\Delta t\nu} \left( \left\| \mathbf u^{N,\Delta t}_n \right\|_{\ell^2_0(\T_N)}^2 + \left\| \mathbf v^{N,\Delta t}_n \right\|_{\ell^2_0(\T_N)}^2 \right) , \]
by induction, we get for all $n\in\N$,
\[ \left\| \mathbf u^{N,\Delta t}_n \right\|_{\ell^2_0(\T_N)}^2 + \left\| \mathbf v^{N,\Delta t}_n \right\|_{\ell^2_0(\T_N)}^2 \leq \left( \frac{1}{1+2\Delta t\nu} \right)^n \left( \left\| \mathbf u_0 \right\|_{\ell^2_0(\T_N)}^2 + \left\| \mathbf v_0 \right\|_{\ell^2_0(\T_N)}^2 \right) . \]
It appears now that if we fix the value
\[ n_{\varepsilon,M} : = \left\lceil \frac{-1}{\log ( 1+2\Delta t\nu)} \log\left(\frac{\varepsilon^2}{16M^2}\right) \right\rceil , \]
we get for all $n\geq n_{\varepsilon,M}$,
\[ \left\| \mathbf u^{N,\Delta t}_n \right\|_{\ell^1_0(\T_N)} + \left\| \mathbf v^{N,\Delta t}_n \right\|_{\ell^1_0(\T_N)} \leq \left\| \mathbf u^{N,\Delta t}_n \right\|_{\ell^2_0(\T_N)} + \left\| \mathbf v^{N,\Delta t}_n \right\|_{\ell^2_0(\T_N)} \leq \frac\varepsilon2 . \]

Now, we fix $\delta_\varepsilon :=\varepsilon/(4n_{\varepsilon,M})$ and we restrict ourselves to the event
\begin{equation}\label{set}
\left\{ \sup_{n=1,\dots,n_{\varepsilon,M}} \left\| \Delta\mathbf W^{Q,N}_n \right\|_{\ell^1_0(\T_N)} \leq \delta_\varepsilon \right\} .
\end{equation}

Let $(\mathbf U^{N,\Delta t}_n)_{n\in\N}$ and $(\mathbf V^{N,\Delta t}_n)_{n\in\N}$ be two solutions of~\eqref{SSBE} with the deterministic initial conditions $\mathbf u_0$ and $\mathbf v_0$ respectively. With similar arguments as for the proof of Proposition~\ref{L1C}, we get from~\eqref{SSBE},~\eqref{noiseless} and Lemma~\ref{L1Cdrift}.(ii), for all $n\in\N$,
\begin{align*}
&\left\|\mathbf U^{N,\Delta t}_{n+1} - \mathbf u^{N,\Delta t}_{n+1} \right\|_{\ell^1_0(\T_N)} + \left\|\mathbf V^{N,\Delta t}_{n+1} - \mathbf v^{N,\Delta t}_{n+1} \right\|_{\ell^1_0(\T_N)}\\
&\leq \left\|\mathbf U^{N,\Delta t}_{n+\frac12} - \mathbf u^{N,\Delta t}_{n+1} \right\|_{\ell^1_0(\T_N)}+ \left\|\mathbf V^{N,\Delta t}_{n+\frac12} - \mathbf v^{N,\Delta t}_{n+1} \right\|_{\ell^1_0(\T_N)} + 2 \left\| \Delta\mathbf W^{Q,N}_{n+1} \right\|_{\ell^1_0(\T_N)} \\
&= \left\langle \bm\sgn \left(\mathbf U^{N,\Delta t}_{n+\frac12} - \mathbf u^{N,\Delta t}_{n+1} \right) ,\mathbf U^{N,\Delta t}_n - \mathbf u^{N,\Delta t}_n \right\rangle_{\ell^2(\T_N)} \\
&\quad + \Delta t \left\langle \bm\sgn \left(\mathbf U^{N,\Delta t}_{n+\frac12} - \mathbf u^{N,\Delta t}_{n+1} \right) ,\mathbf b \left(\mathbf U^{N,\Delta t}_{n+\frac12} \right) -\mathbf b \left( \mathbf u^{N,\Delta t}_{n+1} \right) \right\rangle_{\ell^2(\T_N)} \\
&\quad + \left\langle \bm\sgn \left(\mathbf V^{N,\Delta t}_{n+\frac12} - \mathbf v^{N,\Delta t}_{n+1} \right) ,\mathbf V^{N,\Delta t}_n - \mathbf v^{N,\Delta t}_n \right\rangle_{\ell^2(\T_N)} \\
&\quad + \Delta t \left\langle \bm\sgn \left(\mathbf V^{N,\Delta t}_{n+\frac12} - \mathbf v^{N,\Delta t}_{n+1} \right) ,\mathbf b \left(\mathbf V^{N,\Delta t}_{n+\frac12} \right) -\mathbf b \left( \mathbf v^{N,\Delta t}_{n+1} \right) \right\rangle_{\ell^2(\T_N)} + 2\left\| \Delta\mathbf W^{Q,N}_{n+1} \right\|_{\ell^1_0(\T_N)} \\
&\leq \left\|\mathbf U^{N,\Delta t}_n - \mathbf u^{N,\Delta t}_n \right\|_{\ell^1_0(\T_N)} + \left\|\mathbf V^{N,\Delta t}_n - \mathbf v^{N,\Delta t}_n \right\|_{\ell^1_0(\T_N)} + 2 \left\| \Delta\mathbf W^{Q,N}_{n+1} \right\|_{\ell^1_0(\T_N)}.
\end{align*}
On the event~\eqref{set}, we have for all $n=1,\dots,n_{\varepsilon,M}$,
\[ \left\|\mathbf U^{N,\Delta t}_{n+1} - \mathbf u^{N,\Delta t}_{n+1} \right\|_{\ell^1_0(\T_N)} + \left\|\mathbf V^{N,\Delta t}_{n+1} - \mathbf v^{N,\Delta t}_{n+1} \right\|_{\ell^1_0(\T_N)} \leq \left\|\mathbf U^{N,\Delta t}_n - \mathbf u^{N,\Delta t}_n \right\|_{\ell^1_0(\T_N)} + \left\|\mathbf V^{N,\Delta t}_n - \mathbf v^{N,\Delta t}_n \right\|_{\ell^1_0(\T_N)} +2\delta_\varepsilon . \]
In particular, by induction, we have
\[ \left\|\mathbf U^{N,\Delta t}_{n_{\varepsilon,M}} - \mathbf u^{N,\Delta t}_{n_{\varepsilon,M}} \right\|_{\ell^1_0(\T_N)} + \left\|\mathbf V^{N,\Delta t}_{n_{\varepsilon,M}} - \mathbf v^{N,\Delta t}_{n_{\varepsilon,M}} \right\|_{\ell^1_0(\T_N)} \leq 2 n_{\varepsilon,M} \delta_\varepsilon = \frac\varepsilon2 .\]
Thus,
\begin{align*}
\left\|\mathbf U^{N,\Delta t}_{n_{\varepsilon,M}} \right\|_{\ell^1_0(\T_N)} + \left\|\mathbf V^{N,\Delta t}_{n_{\varepsilon,M}} \right\|_{\ell^1_0(\T_N)} & \leq \left\|\mathbf U^{N,\Delta t}_{n_{\varepsilon,M}} - \mathbf u^{N,\Delta t}_{n_{\varepsilon,M}} \right\|_{\ell^1_0(\T_N)} + \left\|\mathbf V^{N,\Delta t}_{n_{\varepsilon,M}} - \mathbf v^{N,\Delta t}_{n_{\varepsilon,M}} \right\|_{\ell^1_0(\T_N)} + \left\| \mathbf u^{N,\Delta t}_{n_{\varepsilon,M}} \right\|_{\ell^1_0(\T_N)}+ \left\| \mathbf v^{N,\Delta t}_{n_{\varepsilon,M}} \right\|_{\ell^1_0(\T_N)}\\
& \leq \frac\varepsilon2 + \frac\varepsilon2 = \varepsilon .
\end{align*}
We just have shown that
\[ \P_{(\mathbf u_0,\mathbf v_0)} \left( \left\|\mathbf U_{n_{\varepsilon,M}} \right\|_{\ell^1_0(\T_N)} + \left\|\mathbf V_{n_{\varepsilon,M}} \right\|_{\ell^1_0(\T_N)} \leq \varepsilon \right) \geq \P \left( \sup_{n=1,\dots,n_{\varepsilon,M}} \left\| \Delta\mathbf W^{Q,N}_n \right\|_{\ell^1_0(\T_N)} \leq \delta_\varepsilon \right) > 0 . \]
Since the event~\eqref{set} does not depend on $\mathbf u_0$ nor $\mathbf v_0$, we get the result.
\end{proof}

\section{Proofs of finite-time bounds for Sections~\ref{section:convergence} and~\ref{section:convergence2}}\label{app:estim}

\begin{proof}[Proof of Lemma~\ref{lem:FTBus}]
The proof of these estimates is largely based on refinements of computations made in~\cite{MR19}.

\emph{Proof of the $L^p_0(\T)$ estimate.} Let $p\in2\N^*$ and let us repeat the proof of \cite[Lemma~3]{MR19} up to \cite[Equation~(23)]{MR19}. When the initial condition $u^*_0$ is random and has distribution $\mu^*$, this equation writes
\begin{align*}
 \E\left[\left\|u^*\left(t\wedge T_r\right)\right\|_{L^p_0(\T)}^p\right] & =\E\left[\left\|u^*_0\right\|_{L^p_0(\T)}^p\right]-p\E\left[\int_0^{t\wedge T_r}\int_\T\partial_xA(u^*(s))u^*(s)^{p-1}\dd x\dd s\right]\\
 &\quad -\nu p(p-1)\E\left[\int_0^{t\wedge T_r}\int_\T\partial_xu^*(s)^2u^*(s)^{p-2}\dd x\dd s\right]+\frac{p(p-1)}2\sum_{k\geq1}\E\left[\int_0^{t\wedge T_r}\int_\T u^*(s)^{p-2}(g^k)^2\dd x\dd s \right],
\end{align*}
 for all $t \geq 0$ and $r\geq0$, where $T_r$ is a stopping time converging almost surely towards $+\infty$ as $r\to+\infty$ (by~\cite[Corollary~2]{MR19}). Using~\cite[Equation~(24)]{MR19}, the non-positivity of the third term of the right-hand side, and bounding the $g^k$'s by their supremum, we get the inequality
 \[ \E\left[\left\|u^*\left(t\wedge T_r\right)\right\|_{L^p_0(\T)}^p\right] \leq \E\left[\left\|u^*_0\right\|_{L^p_0(\T)}^p\right]+\frac{p(p-1)}2\left(\sum_{k\geq1}\left\|g^k\right\|_{L^\infty_0(\T)}^2\right)\E\left[\int_0^{t\wedge T_r}\left\|u^*(s)\right\|_{L^{p-2}_0(\T)}^{p-2}\dd s\right] . \]
 Using now Corollary~\ref{cor:relcom}, \eqref{eq:strongerpoincare}, \eqref{boundgk}, and \cite[Equation~(18)]{MR19}, we get
 \[ \E\left[\left\|u^*\left(t\wedge T_r\right)\right\|_{L^p_0(\T)}^p\right] \leq \mathsf{C}^{0,p} + \frac{p(p-1)}2\mathsf{D}\left(C_5^{(p-2)}\left(1+\E\left[\left\|u^*_0\right\|_{L^{p-2}_0(\T)}^{p-2}\right]\right)+C_6^{(p-2)}t\right) , \]
 where the constants $C_5^{(p-2)}$ and $C_6^{(p-2)}$, defined in~\cite{MR19}, depend only on $\nu$, $p$ and $\mathsf{D}$. Using once again Corollary~\ref{cor:relcom} and letting $r\to+\infty$, we obtain
 \[ \limsup_{r\to\infty} \E\left[\left\|u^*\left(t\wedge T_r\right)\right\|_{L^p_0(\T)}^p\right] \leq C^{0,p} + \frac{p(p-1)}2\mathsf{D}\left(C_5^{(p-2)}\left(1+\mathsf{C}^{0,p-2}\right)+C_6^{(p-2)}t\right) =: \mathsf{C}^{*,0,p}_t . \]
 Applying Fatou's lemma on the left-hand side, we get
 \[ \E\left[\left\|u^*(t)\right\|_{L^p_0(\T)}^p\right] \leq \mathsf{C}^{*,0,p}_t , \]
from which we easily get the claimed inequality in the case $p\in2\N^*$. The general case $p\in[1,+\infty)$ then follows from the Jensen inequality.
 
\emph{Proof of the $H^1_0(\T)$ and $H^2_0(\T)$ estimates.} We now start from~\cite[Lemma~4]{MR19} which, when $u^*_0$ is random, gives the estimate
 \[ \E\left[\left\|u^*(t\wedge T_r)\right\|_{H^1_0(\T)}^2\right] + \nu \E\left[\int_0^{t \wedge T_r} \left\|u^*(s)\right\|_{H^2_0(\T)}^2\dd s\right] \leq\E\left[\left\|u^*_0\right\|_{H^1_0(\T)}^2\right]+C_7\left(1+\E\left[\left\|u^*_0\right\|_{L^{2\mathsf{p}_A+2}_0(\T)}^{2\mathsf{p}_A+2}\right]\right)+C_8t , \]
and from which we deduce, by applying Fatou's lemma on the left-hand side and Corollary~\ref{cor:relcom} on the right-hand side:
\[\E\left[\left\|u^*(t)\right\|_{H^1_0(\T)}^2\right] + \nu \E\left[\int_0^t \left\|u^*(s)\right\|_{H^2_0(\T)}^2\dd s\right]\leq \mathsf{C}^{1,2}+C_7\left(1+\mathsf{C}^{0,2\mathsf{p}_A+2}\right)+C_8t =:\mathsf{C}^{*,1,2}_t.\]
We conclude that
\begin{equation*}
  \E\left[\left\|u^*(t)\right\|_{H^1_0(\T)}^2\right] \leq \mathsf{C}^{1,2}+C_7\left(1+\mathsf{C}^{0,2\mathsf{p}_A+2}\right)+C_8t =:\mathsf{C}^{*,1,2}_t,
\end{equation*}
and
\begin{equation*}
  \E\left[\int_0^t \left\|u^*(s)\right\|_{H^2_0(\T)}^2\dd s\right]\leq \frac{1}{\nu}\left(\mathsf{C}^{1,2}+C_7\left(1+\mathsf{C}^{0,2\mathsf{p}_A+2}\right)+C_8t\right) =:\mathsf{C}^{*,2,2}_t.\qedhere
\end{equation*}
\end{proof}

\begin{proof}[Proof of Lemma~\ref{lem:FTBUN}] Lemma~\ref{lem:FTBUN} is a refinement of the uniform $h^2_0(\T_N)$ estimate from Proposition~\ref{prop:unifestim}. We start from~\eqref{itoH1} and use the definition of $\mathbf{b}$ to write
\begin{align*}
  \left\|\mathbf D^{(1,+)}_N\mathbf U^N(t)\right\|_{\ell^2_0(\T_N)}^2 &=\left\|\mathbf D^{(1,+)}_N\mathbf U^N_0\right\|_{\ell^2_0(\T_N)}^2 +2\int_0^t\left\langle\mathbf D^{(1,+)}_N\mathbf U^N(s),\mathbf D^{(1,+)}_N \left(-\mathbf{D}^{(1,-)}_N\overline{\mathbf{A}}^N(\mathbf U^N(s)) + \nu\mathbf{D}^{(2)}_N\mathbf U^N(s)\right)\right\rangle_{\ell^2_0(\T_N)}\dd s\\
  &\quad +2\int_0^t\left\langle\mathbf D^{(1,+)}_N\mathbf U^N(s),\dd \left(\mathbf D^{(1,+)}_N\mathbf W^{Q,N}\right)(s)\right\rangle_{\ell^2_0(\T_N)} +t\sum_{k\geq1}\left\|\mathbf D^{(1,+)}_N\mathbf{g}^k\right\|_{\ell^2_0(\T_N)}^2.
\end{align*}

By~\eqref{eq:sbp2} and Young's inequality,
\begin{align*}
  -2\left\langle\mathbf D^{(1,+)}_N\mathbf U^N(s),\mathbf D^{(1,+)}_N \mathbf{D}^{(1,-)}_N\overline{\mathbf{A}}^N(\mathbf U^N(s))\right\rangle_{\ell^2_0(\T_N)} &= 2\left\langle\mathbf D^{(2)}_N\mathbf U^N(s), \mathbf{D}^{(1,-)}_{N}\overline{\mathbf{A}}^N(\mathbf U^N(s))\right\rangle_{\ell^2_0(\T_N)}\\
  & \leq 2\nu\|\mathbf{D}^{(2)}_N\mathbf U^N(s)\|^2_{\ell^2_0(\T_N)} + \frac{1}{2\nu}\|\mathbf{D}^{(1,-)}_{N}\overline{\mathbf{A}}^N(\mathbf U^N(s))\|^2_{\ell^2_0(\T_N)},
\end{align*}
while 
\begin{equation*}
  2\nu\left\langle\mathbf D^{(1,+)}_N\mathbf U^N(s),\mathbf D^{(1,+)}_N \mathbf{D}^{(2)}_N\mathbf U^N(s)\right\rangle_{\ell^2_0(\T_N)} = -2\nu \|\mathbf{D}^{(2)}_N\mathbf U^N(s)\|^2_{\ell^2_0(\T_N)},
\end{equation*}
so that, using~\eqref{eq:sigma-glob} in addition,
\begin{align*}
  \left\|\mathbf D^{(1,+)}_N\mathbf U^N(t)\right\|_{\ell^2_0(\T_N)}^2 & \leq\left\|\mathbf D^{(1,+)}_N\mathbf U^N_0\right\|_{\ell^2_0(\T_N)}^2 +\frac{1}{2\nu}\int_0^t\|\mathbf{D}^{(1,-)}_{N}\overline{\mathbf{A}}^N(\mathbf U^N(s))\|^2_{\ell^2_0(\T_N)}\dd s\\
  &\quad +2\int_0^t\left\langle\mathbf D^{(1,+)}_N\mathbf U^N(s),\dd \left(\mathbf D^{(1,+)}_N\mathbf W^{Q,N}\right)(s)\right\rangle_{\ell^2_0(\T_N)} +t\mathsf{D}.
\end{align*}

We deduce that for any $t \geq 0$,
\begin{equation}\label{eq:inequinter}
  \sup_{s \in [0,t]} \left\|\mathbf D^{(1,+)}_N\mathbf U^N(s)\right\|_{\ell^2_0(\T_N)}^2 \leq \left\|\mathbf D^{(1,+)}_N\mathbf U^N_0\right\|_{\ell^2_0(\T_N)}^2 + \frac{1}{2\nu}\int_0^t\|\mathbf{D}^{(1,-)}_{N}\overline{\mathbf{A}}^N(\mathbf U^N(s))\|^2_{\ell^2_0(\T_N)}\dd s + 2Z^N_t + t\mathsf{D},
\end{equation}
where $Z^N_t$ is defined by
\begin{equation*}
  Z^N_t = \sup_{s \in [0,t]} \int_0^s\left\langle\mathbf D^{(1,+)}_N\mathbf U^N(r),\dd \left(\mathbf D^{(1,+)}_N\mathbf W^{Q,N}\right)(r)\right\rangle_{\ell^2_0(\T_N)},
\end{equation*}
and it remains to control the expectation of the right-hand side of the inequality~\eqref{eq:inequinter}.

First, by Proposition~\ref{prop:unifestim}, we have
\begin{equation*}
  \E\left[\left\|\mathbf D^{(1,+)}_N\mathbf U^N_0\right\|_{\ell^2_0(\T_N)}^2\right] \leq \mathsf{C}^{1,2}.
\end{equation*}
Next, by stationarity of $\mathbf{U}^N$ and~\eqref{fluxbound}, we have
\begin{equation*}
  \E\left[\int_0^t\|\mathbf{D}^{(1,-)}_{N}\overline{\mathbf{A}}^N(\mathbf U^N(s))\|^2_{\ell^2_0(\T_N)}\dd s\right] \leq t \E\left[\|\mathbf{D}^{(1,-)}_{N}\overline{\mathbf{A}}^N(\mathbf U^N_0)\|^2_{\ell^2_0(\T_N)}\right] \leq t8 \mathsf{C}_{\bar A}^2 \left(\mathsf{C}^{1,2}+\frac{\mathsf{D}}{2\nu}\mathsf{C}^{0,2\mathsf{p}_{\bar A}} \right).
\end{equation*}
Finally, we recall that by~\eqref{martingale}, the process $(\int_0^t\langle\mathbf D^{(1,+)}_N\mathbf U^N(s),\dd (\mathbf D^{(1,+)}_N\mathbf W^{Q,N})(s)\rangle_{\ell^2_0(\T_N)})_{t\geq0}$ is a martingale. Therefore, applying successively the Jensen and the Doob inequalities, the It\^o isometry, the Cauchy--Schwarz inequality, Proposition~\ref{prop:unifestim} and~\eqref{eq:sigma-glob}, we get
\begin{align*}
\E\left[Z^N_t\right]&\leq \left( \E\left[\sup_{s\in[0,t]}\left|\int_0^s\left\langle\mathbf D^{(1,+)}_N\mathbf U^N(r),\dd\left(\mathbf D^{(1,+)}_N\mathbf W^{Q,N}\right)(r)\right\rangle_{\ell^2_0(\T_N)}\right|^2\right] \right)^{1/2}\\
&\leq2 \left( \E\left[\left|\int_0^t\left\langle\mathbf D^{(1,+)}_N\mathbf U^N(r),\dd\left(\mathbf D^{(1,+)}_N\mathbf W^{Q,N}\right)(r)\right\rangle_{\ell^2_0(\T_N)}\right|^2\right] \right)^{1/2}\\
&=2 \left( \E\left[\sum_{k\geq1}\int_0^t\left\langle\mathbf D^{(1,+)}_N\mathbf U^N(r),\mathbf D^{(1,+)}_N\mathbf{g}^k\right\rangle_{\ell^2_0(\T_N)}^2\dd r\right] \right)^{1/2}\\
&\leq2\sqrt t \left( \E\left[\left\|\mathbf D^{(1,+)}_N\mathbf U^N_0\right\|_{\ell^2_0(\T_N)}^2\right] \right)^{1/2}\left(\sum_{k\geq1}\left\|\mathbf D^{(1,+)}_N\mathbf{g}^k\right\|_{\ell^2_0(\T_N)}^2\right)^{1/2}\\
&\leq 2\sqrt{t\mathsf{C}^{1,2}\mathsf{D}},
\end{align*}
which completes the proof.
\end{proof}

\begin{proof}[Proof of Lemma~\ref{FTbound}] Lemma~\ref{FTbound} is a refinement of the proof of Proposition~\ref{existence}. We fix $t>0$, and for the sake of simplicity we write $\Delta t$ in place of $\Delta t_j$. We also introduce the notation $n_t = \lfloor \frac t {\Delta t} \rfloor$.

We start from Equation~\eqref{eq:discretederivative}. For all $n=0,\dots,n_t$, we write
\begin{align*}
\left\| \mathbf U^{N, \Delta t}_n \right\|_{\ell^2_0(\T_N)}^2 &= \left\| \mathbf U^{N, \Delta t}_0 \right\|_{\ell^2_0(\T_N)}^2+\sum_{l=0}^{n-1}\left( \left\| \mathbf U^{N, \Delta t}_{l+1} \right\|_{\ell^2_0(\T_N)}^2-\left\| \mathbf U^{N, \Delta t}_l \right\|_{\ell^2_0(\T_N)}^2\right)\\
& \leq \left\|\mathbf U^{N, \Delta t}_0 \right\|_{\ell^2_0(\T_N)}^2 -2\nu \Delta t \sum_{l=0}^{n-1} \left\|\mathbf D^{(1,+)}_N\mathbf U^{N, \Delta t}_{l+\frac 1 2} \right\|_{\ell^2_0(\T_N)}^2 +2 \sum_{l=0}^{n-1} \left\langle\mathbf U^{N, \Delta t}_{l+\frac 1 2} , \Delta\mathbf W^{Q,N}_{l+1} \right\rangle_{\ell^2_0(\T_N)} + \sum_{l=0}^{n-1} \left\| \Delta\mathbf W^{Q,N}_{l+1} \right\|_{\ell^2_0(\T_N)}^2\\
& \leq \left\|\mathbf U^{N, \Delta t}_0 \right\|_{\ell^2_0(\T_N)}^2 +2 \sum_{l=0}^{n-1} \left\langle\mathbf U^{N, \Delta t}_{l+\frac 1 2} , \Delta\mathbf W^{Q,N}_{l+1} \right\rangle_{\ell^2_0(\T_N)} + \sum_{l=0}^{n-1} \left\| \Delta\mathbf W^{Q,N}_{l+1} \right\|_{\ell^2_0(\T_N)}^2 .
\end{align*}
Taking the supremum in time and the expectation, we get
 \begin{equation}\label{supexpect}
 \E \left[ \sup_{n=0,\dots,n_t} \left\|\mathbf U^{N, \Delta t}_n \right\|_{\ell^2_0(\T_N)}^2 \right] \leq \E \left[ \left\|\mathbf U^{N, \Delta t}_0 \right\|_{\ell^2_0(\T_N)}^2 \right] + 2 \E \left[ \sup_{n=0,\dots,n_t} \left| \sum_{l=0}^{n-1} \left\langle\mathbf U^{N, \Delta t}_{l+\frac 1 2} , \Delta\mathbf W^{Q,N}_{l+1} \right\rangle_{\ell^2_0(\T_N)} \right| \right] + \E \left[ \sum_{l=0}^{n_t -1} \left\| \Delta\mathbf W^{Q,N}_{l+1} \right\|_{\ell^2_0(\T_N)}^2 \right] .
 \end{equation}
 First, by~\eqref{eq:poinca-discr} and Proposition~\ref{prop:unifestimSS}, we have
 \[ \E\left[\left\|\mathbf U^{N, \Delta t}_0 \right\|_{\ell^2_0(\T_N)}^2 \right] \leq \E\left[\left\|\mathbf D^{(1,+)}_N\mathbf U^{N, \Delta t}_0 \right\|_{\ell^2_0(\T_N)}^2 \right] \leq \mathsf{C}^{\Delta,1,2} . \]
 Noticing that the sequence $( \sum_{l=0}^{n-1} \langle\mathbf U^{N, \Delta t}_{l+\frac 1 2} , \Delta\mathbf W^{Q,N}_{l+1} \rangle_{\ell^2_0(\T_N)} )_{n\geq1}$ is a martingale, we get by applying successively Jensen's and Doob's inequalities to the second term of the right-hand side,
 \begin{align*}
  \E \left[ \sup_{n=0,\dots,n_t} \left| \sum_{l=0}^{n-1} \left\langle\mathbf U^{N, \Delta t}_{l+\frac 1 2} , \Delta\mathbf W^{Q,N}_{l+1} \right\rangle_{\ell^2_0(\T_N)} \right| \right] &\leq \left(\E \left[ \sup_{n=0,\dots,n_t} \left| \sum_{l=0}^{n-1} \left\langle\mathbf U^{N, \Delta t}_{l+\frac 1 2} , \Delta\mathbf W^{Q,N}_{l+1} \right\rangle_{\ell^2_0(\T_N)} \right|^2 \right] \right)^{1/2}\\
  &\leq 2\left(\E \left[ \left| \sum_{l=0}^{n_t-1} \left\langle\mathbf U^{N, \Delta t}_{l+\frac 1 2} , \Delta\mathbf W^{Q,N}_{l+1} \right\rangle_{\ell^2_0(\T_N)} \right|^2 \right] \right)^{1/2} . \\
 \end{align*}
 From~\eqref{SSBE}, we may observe that each increment $\Delta\mathbf W^{Q,N}_{l+1}$ is independent from the family $(\mathbf U^{N, \Delta t}_{m+\frac12},\Delta\mathbf W^{Q,N}_m)_{m=0,\dots,l}$. Therefore, defining $\mathbf{V}^{N,\Delta t}$ and $\mathbf{V}^{N,\Delta t}_{\frac12}$ as in Proposition~\ref{prop:unifestimSS}, we have
 \begin{align*}
  2\left( \E \left[ \sum_{l=0}^{n_t-1} \left| \left\langle\mathbf U^{N, \Delta t}_{l+\frac 1 2} , \Delta\mathbf W^{Q,N}_{l+1} \right\rangle_{\ell^2_0(\T_N)} \right|^2 \right]\right)^{1/2} &\leq2\left(\sum_{l=0}^{n_t-1}\E\left[\left\|\mathbf U^{N, \Delta t}_{l+\frac12}\right\|_{\ell^2_0(\T_N)}^2\right]\E\left[\left\|\Delta\mathbf W^{Q,N}_{l+1}\right\|_{\ell^2_0(\T_N)}^2\right]\right)^{1/2}\\
  &\leq 2\sqrt{\mathsf{D} \Delta t} \left( n_t\E \left[ \left\|\mathbf V^{N, \Delta t}_{\frac 1 2} \right\|_{\ell^2_0(\T_N)}^2 \right] \right)^{1/2} \\
  &\leq 2\sqrt{\mathsf{D} t \mathsf{C}^{\Delta,1,2}_{\frac12}},
 \end{align*}
 where we have used~\eqref{NB} at the second line and Proposition~\ref{prop:unifestimSS} together with~\eqref{eq:poinca-discr} at the third line. Injecting this bound into~\eqref{supexpect}, and using~\eqref{NB} again, we finally get
 \[ \E \left[ \sup_{n=0,\dots,n_t} \left\|\mathbf U^{N, \Delta t}_n \right\|_{\ell^2_0(\T_N)}^2 \right] \leq \mathsf{C}^{\Delta,1,2} + 2\sqrt{\mathsf{D} t \mathsf{C}^{\Delta,1,2}_{\frac12}} + t\mathsf{D} =: \mathsf{S}^{\Delta,0,2}_t. \qedhere \]
\end{proof}

\subsection*{Acknowledgements}

The authors are thankful to Charles-\'Edouard Br\'ehier for insightful discussions and to two anonymous referees for their careful reading and valuable comments on this work.


\begin{thebibliography}{10}

\bibitem{AG06}
Aureli Alabert and Istv{\'a}n Gyongy.
\newblock {\em On Numerical Approximation of Stochastic Burgers' Equation},
  pages 1--15.
\newblock In {\em From Stochastic Calculus to Mathematical Finance: The Shiryaev Festschrift}, Springer, Berlin, Heidelberg, 2006.

\bibitem{BCG16}
Caroline Bauzet, Julia Charrier, and Thierry Gallou{\"e}t.
\newblock {Convergence of flux-splitting finite volume schemes for hyperbolic
  scalar conservation laws with a multiplicative stochastic perturbation}.
\newblock {\em Mathematics of Computation}, 85:2777--2813, 2016.

\bibitem{BCG16b}
Caroline Bauzet, Julia Charrier, and Thierry Gallou{\"e}t.
\newblock {Convergence of monotone finite volume schemes for hyperbolic scalar
  conservation laws with multiplicative noise}.
\newblock {\em {Stochastics and Partial Differential Equations: Analysis and
  Computations}}, 4(1):150--223, 2016.

\bibitem{Bil99}
Patrick Billingsley.
\newblock {\em Convergence of probability measures}.
\newblock Wiley Series in Probability and Statistics: Probability and
  Statistics. John Wiley \& Sons Inc., New York, second edition, 1999.
\newblock A Wiley-Interscience Publication.

\bibitem{Bre14}
Charles-{\'E}douard Br{\'e}hier.
\newblock {Approximation of the invariant measure with an Euler scheme for
  Stochastic PDE{\rq}s driven by Space-Time White Noise}.
\newblock {\em Potential Analysis}, 1(40):1--40, 2014.

\bibitem{BK16}
Charles-\'Edouard Br{\'e}hier and Marie Kopec.
\newblock {Approximation of the invariant law of SPDEs: error analysis using a
  Poisson equation for a full-discretization scheme}.
\newblock {\em {IMA Journal of Numerical Analysis}}, 2016.

\bibitem{BG16}
Charles-\'Edouard Br\'ehier and Gilles Vilmart.
\newblock High order integrator for sampling the invariant distribution of a
  class of parabolic stochastic {PDE}s with additive space-time noise.
\newblock {\em SIAM Journal on Scientific Computing}, 38(4):A2283--A2306, 2016.

\bibitem{Bre20}
Charles-Edouard Bréhier.
\newblock Approximation of the invariant distribution for a class of ergodic
  SPDEs using an explicit tamed exponential Euler scheme.
\newblock arXiv:2010.00508, 2020.

\bibitem{BreDeb18}
Charles-Edouard Bréhier and Arnaud Debussche.
\newblock Kolmogorov equations and weak order analysis for spdes with nonlinear
  diffusion coefficient.
\newblock {\em Journal de Mathématiques Pures et Appliquées}, 119:193--254,
  2018.

\bibitem{CHW17}
Chuchu Chen, Jialin Hong, and Xu~Wang.
\newblock Approximation of invariant measure for damped stochastic nonlinear
  {S}chr{\"o}dinger equation via an ergodic numerical scheme.
\newblock {\em Potential Analysis}, 46(2):323--367, 2017.

\bibitem{CGW18}
Ziheng Chen, Siqing Gan, and Xiaojie Wang.
\newblock A full-discrete exponential {E}uler approximation of the invariant
  measure for parabolic stochastic partial differential equations.
\newblock {\em Applied Numerical Mathematics}, 157:135--158, 2020.

\bibitem{CHS18}
Jianbo Cui, Jialin Hong, and Liying Sun.
\newblock Weak convergence and invariant measure of a full discretization for
  non-globally {L}ipschitz parabolic {SPDE}.
\newblock arXiv:1811.04075, 2018.

\bibitem{DZ92}
Giuseppe Da~Prato and Jerzy Zabczyk.
\newblock {\em Stochastic Equations in Infinite Dimensions}.
\newblock Encyclopedia of Mathematics and its Applications. Cambridge
  University Press, 1992.

\bibitem{DZ96}
Giuseppe Da~Prato and Jerzy Zabczyk.
\newblock {\em Ergodicity for Infinite Dimensional Systems}.
\newblock Cambridge Monographs on Partic. Cambridge University Press, 1996.

\bibitem{DV15}
Arnaud Debussche and Julien Vovelle.
\newblock Invariant measure of scalar first-order conservation laws with stochastic forcing.
\newblock {\em Probability Theory and Related Fields}, 163(3):575--611, 2015.

\bibitem{Dei85}
Klaus Deimling.
\newblock {\em Nonlinear functional analysis}.
\newblock Springer-Verlag, 1985.

\bibitem{Dot17}
Sylvain Dotti.
\newblock {\em {Numerical approximation of hyperbolic stochastic scalar
  conservation laws}}.
\newblock PhD thesis, {Aix-Marseille Universit{\'e} (AMU)}, 2017.

\bibitem{DV18}
Sylvain Dotti and Julien Vovelle.
\newblock Convergence of approximations to stochastic scalar conservation laws.
\newblock {\em Archive for Rational Mechanics and Analysis}, 230(2):539--591, 2018.

\bibitem{DV19}
Sylvain Dotti and Julien Vovelle.
\newblock Convergence of the finite volume method for scalar conservation laws
  with multiplicative noise: an approach by kinetic formulation.
\newblock {\em Stochastics and Partial Differential Equations: Analysis and
  Computations}, 2019.

\bibitem{DL82}
D.~C. Dowson and Basil V. Landau.
\newblock The {F}r\'{e}chet distance between multivariate normal distributions.
\newblock {\em Journal of Multivariate Analysis}, 12(3):450--455, 1982.

\bibitem{EO80}
Bjorn Engquist and Stanley Osher.
\newblock Stable and entropy satisfying approximations for transonic flow
  calculations.
\newblock {\em Mathematics of Computation}, 34(149):45--75, 1980.

\bibitem{EGH00}
Robert Eymard, Thierry Gallou{\"e}t, and Rapha\`ele Herbin.
\newblock Finite volume methods.
\newblock In {\em Solution of Equation in $\R^n$ (Part 3), Techniques of
  Scientific Computing (Part 3)}, volume~7 of {\em Handbook of Numerical
  Analysis}, pages 713 -- 1018. Elsevier, 2000.

\bibitem{Gyo98}
Istv\'{a}n Gy\"{o}ngy.
\newblock Lattice approximations for stochastic quasi-linear parabolic partial
  differential equations driven by space-time white noise. {I}.
\newblock {\em Potential Anal.}, 9(1):1--25, 1998.

\bibitem{Hai06}
Martin Hairer.
\newblock {Ergodic Properties of Markov Processes}.
\newblock Unpublished notes.

\bibitem{Hai11}
Martin Hairer and Jochen Voss.
\newblock Approximations to the stochastic {B}urgers equation.
\newblock {\em Journal of Nonlinear Science}, 21(6):897--920, 2011.

\bibitem{Kop14}
Marie Kopec.
\newblock {Weak backward error analysis for overdamped Langevin processes}.
\newblock {\em IMA Journal of Numerical Analysis}, 35(2):583--614, 2014.

\bibitem{MP03}
Charalambos Makridakis and Beno{\^\i}t Perthame.
\newblock {Sharp CFL, discrete kinetic formulation, and entropic schemes for
  scalar conservation laws}.
\newblock {\em SIAM J. Numer. Anal.}, 41(3):1032--1051, 2003.

\bibitem{Mar19}
Sofiane Martel.
\newblock {\em {Theoretical and numerical analysis of invariant measures of
  viscous stochastic scalar conservation laws}}.
\newblock Ph{D} {T}hesis, {Universit\'e Paris-Est}, 2019.

\bibitem{MR19}
Sofiane Martel and Julien Reygner.
\newblock Viscous scalar conservation law: strong solution and invariant
  measure.
\newblock {\em Nonlinear Differential Equations and Applications NoDEA},
  27(3):34, 2020.

\bibitem{MSH02}
Jonathan~C. Mattingly, Andrew~M. Stuart, and Desmond~J. Higham.
\newblock Ergodicity for {SDE}s and approximations: locally {L}ipschitz vector
  fields and degenerate noise.
\newblock {\em Stochastic Processes and their Applications}, 101(2):185 -- 232,
  2002.

\bibitem{Tal90}
Denis Talay.
\newblock Second-order discretization schemes of stochastic differential
  systems for the computation of the invariant law.
\newblock {\em Stochastics: An International Journal of Probability and
  Stochastic Processes}, 29(1):13--36, 1990.

\bibitem{Vil08}
C\'edric Villani.
\newblock {\em Optimal Transport: Old and New}.
\newblock Grundlehren der mathematischen Wissenschaften. Springer Berlin
  Heidelberg, 2008.

\end{thebibliography}
\end{document}